\documentclass[10pt]{article}
\usepackage{amsmath}
\usepackage[T1]{fontenc}
\usepackage[utf8]{inputenc}
\usepackage{amsfonts}
\usepackage{amssymb}
\usepackage[english]{babel}
\usepackage{graphicx}
\usepackage{amsthm}
\usepackage{hyperref}
\usepackage{accents}
\usepackage [autostyle, english = american]{csquotes}
\MakeOuterQuote{"}
\numberwithin{equation}{section}
\theoremstyle{plain}
\newtheorem*{theorem*}{Theorem}
\newtheorem*{lemma*}{Lemma}
\newtheorem{theorem}{Theorem}
\newtheorem{lemma}{Lemma}[section]
\newtheorem{corollary}[lemma]{Corollary}
\newtheorem{question}[lemma]{Question}

\newtheorem{proposition}[lemma]{Proposition}

\theoremstyle{definition}
\newtheorem{definition}[lemma]{Definition}
\newtheorem{remark}[lemma]{Remark}

\newcommand{\be}{\begin{equation}}
\newcommand{\ee}{\end{equation}}
\newcommand{\ba}{\begin{array}{l}}
\newcommand{\ea}{\end{array}}

\newcommand{\R}{\mathcal{R}}

\renewcommand{\div}{{\mbox{div}\,}}

\usepackage{geometry}
\begin{document}

\title{Finite-Time Singularity Formation for $C^{1,\alpha}$ Solutions to the Incompressible Euler Equations on $\mathbb{R}^3$}
\author{Tarek M. Elgindi\footnote{Department of Mathematics, UC San Diego. E-mail: telgindi@ucsd.edu.}}

\date{\today}
\maketitle

\begin{abstract}
It has been known since work of Lichtenstein \cite{Lich25} and Gunther \cite{Gunther} in the 1920's that the $3D$ incompressible Euler equation is locally well-posed in the class of velocity fields with H\"older continuous gradient and suitable decay at infinity. It is shown here that these local solutions can develop singularities in finite time, even for some of the simplest three-dimensional flows. 
\end{abstract}
\tableofcontents
\section{Introduction}

The question of global regularity for solutions to the incompressible Euler equation has been studied by many authors over the years and is considered a major open problem in the study of partial differential equations. The purpose of this work is to solve one case of this problem and, additionally, to bring to light some methods which might prove useful for further studies of the general regularity problem. Our approach is relatively straightforward: we analyze the various terms of the Euler equation and identify regimes where some terms become negligible. It turns out that for solutions satisfying certain symmetries at regularity $C^{1,\alpha}$ with $\alpha>0$ small, it is possible to isolate a simple non-linear equation which encodes the leading order dynamics of the solution to the Euler equation. This simple non-linear equation is exactly solvable and possesses families of explicit solutions which become singular in finite time in a very regular way. In fact, after passing to self-similar variables, they satisfy a time-independent equation. We then search for solutions to the Euler equation which are also self-similar and are close to those found for the model. It turns out to be possible to deduce the existence of such solutions to the Euler equation itself using energy and compactness methods as well as basic modulation techniques since the self-similar solutions to the model equation are \emph{stable} in a very precise sense. 
\addtocontents{toc}{\protect\setcounter{tocdepth}{2}}

\subsection{The Euler equation}

Recall the incompressible Euler equation governing the motion of an ideal fluid on $\mathbb{R}^3$:
\begin{equation}\label{E}\partial_t u+u\cdot\nabla u+\nabla p=f, \end{equation} 
\begin{equation}\label{incompressibility} \div(u)=0, \end{equation}
\begin{equation}\label{IC} u|_{t=0}=u_0. \end{equation}
$u:\mathbb{R}^3\times [0,\infty) \rightarrow\mathbb{R}^3$ is the velocity field of the fluid. $p$ is the force of internal pressure which acts to enforce the incompressibility constraint \eqref{incompressibility}. $f$ is an external force. Incompressibility is a natural property of the fluid: the velocity field is not allowed to squeeze or expand the volume of a portion of fluid. This makes it difficult to imagine the formation of singularities in an ideal fluid since any attempt of squeezing the fluid in a certain direction is met with an expansion in another direction. The incompressibility condition also ensures that smooth and localized solutions to \eqref{E}-\eqref{incompressibility} on $\mathbb{R}^3\times[0,T)$ satisfy:
\begin{equation}\label{EnergyBalance}\frac{d}{dt}\int_{\mathbb{R}^3} |u(x,t)|^2dx=2\int u(x,t)\cdot f(x,t)\end{equation} for all $t\in [0,T)$. This is another reason one might believe that singularities are unlikely. The difficulty is that, as far as our current knowledge goes, to prevent a solution of \eqref{E}-\eqref{incompressibility} from forming a singularity as $t\rightarrow T$, we essentially need to know that $\int_0^t\sup_{x}|\nabla u(x,s)|ds$ is uniformly bounded as $t\rightarrow T$. This follows from viewing \eqref{E}-\eqref{incompressibility} as an ordinary differential equation, in some sense. This substantial gap between what we know and what we need to know about solutions to the Euler equation is what is behind the well-known global regularity problem for the incompressible Euler equation:

\begin{question}\label{question}
Given a solution $u\in C^\infty(\mathbb{R}^3\times [0,T))$ to \eqref{E}-\eqref{incompressibility} satisfying \eqref{EnergyBalance} and external force $f\in  (C^\infty\cap L^2)(\mathbb{R}^3\times [0,T])$, is it possible that $\lim_{t\rightarrow T}\sup_x |\nabla u(x,t)|=+\infty?$
\end{question}

This problem remains, before and after this work, a major open problem in the theory of partial differential equations. The goal of this work is to explore the case of "classical solutions," when $C^\infty$ in Question \ref{question} is replaced by $C^{1,\alpha}$ for some $\alpha>0$. This is the context within which the classical well-posedness theory of the Euler equation has been considered starting with the works of Lichtenstein \cite{Lich25} and Gunther \cite{Gunther}. 

\subsection{The vorticity equation}
An important quantity to consider when studying ideal fluids is the vorticity vector field \[\omega:= \nabla\times u.\] It satisfies the vorticity equation:
\begin{equation}\label{Vorticity} \partial_t\omega+(u\cdot\nabla)\omega=(\omega\cdot \nabla) u+\nabla\times f.\end{equation}
Since $\div(u)=0$ we have that $\nabla \times (\nabla \times u)=-\Delta u.$ Thus, $u$ can be recovered from $\omega$ by the so-called Biot-Savart law:
\begin{equation}\label{BSLaw} u=(-\Delta)^{-1}(\nabla\times \omega).\end{equation} For classical solutions (with $u\in C^{1,\alpha}$ or, equivalently, $\omega\in C^\alpha$ for some $\alpha>0$), solving \eqref{E}-\eqref{incompressibility} is equivalent to solving \eqref{Vorticity}-\eqref{BSLaw} (so long as the vorticity is taken to be divergence-free when solving \eqref{Vorticity}-\eqref{BSLaw}).

To the author's knowledge, the first works on the local well-posedness theory of the 3D Euler equation were completed by Lichtenstein \cite{Lich25} and Gunther \cite{Gunther} in the 1920's and early 1930's. They showed that if $u_0\in C^{1,\alpha}(\mathbb{R}^3)$ for some\footnote{See Subsection \ref{Notation} for the definition of these spaces.} $0<\alpha<1$ and the initial vorticity decays sufficiently rapidly, then there is a time $T>0$ and a unique solution $u\in C^{1,\alpha}(\mathbb{R}^3\times [0,T))$ to \eqref{E}-\eqref{IC}. We call the solutions constructed by Lichtenstein \cite{Lich25} and Gunther \cite{Gunther}  "Classical Solutions." Later, Kato \cite{Kato86} and Kato and Ponce \cite{KatoPonce} established similar results in the scale of Sobolev spaces. 

A well-known result of Beale, Kato, and Majda \cite{BKM} tells us that a classical solution to \eqref{Vorticity}-\eqref{BSLaw} loses its regularity as $t\rightarrow T$ if and only if \[\lim_{t\rightarrow T} \int_0^t \sup_x| \omega(x,s)|ds=+\infty.\] In the special case where we consider two-dimensional solutions, where $\omega_3\equiv 0$ and $\omega_1, \omega_2$ are dependent only on $x_1$ and $x_2$, we have that $\omega\cdot\nabla u\equiv 0$ so that \[\sup_x|\omega(x,t)|\leq \int_0^t \sup_{x} |\nabla\times f(x,s)|ds.\] Consequently, two-dimensional classical solutions to \eqref{Vorticity}-\eqref{BSLaw} cannot develop a singularity in finite time. 

For fully three-dimensional solutions such bounds are not available and, in fact, are known to be false in general \cite{EM1}. We will show here that this lack of bounds was actually a sign of a more alarming fact: that the classical local theory for solutions to the 3D Euler equation \emph{cannot} be made into a global one.

\subsection{Statement of the Main Theorem}

\begin{definition}\label{odd}
A velocity vector field $u:\mathbb{R}^3\rightarrow\mathbb{R}^3$ will be called {\bf odd} if $u_i$ is odd in $x_i$ and even in the other two variables for each $1\leq i\leq 3$. 
\end{definition}
The following theorem is the main result of the present work.
\begin{theorem}\label{MainTheorem}
There exists an $\alpha>0$ and a divergence-free and odd $u_0\in C^{1,\alpha}(\mathbb{R}^3)$ with initial vorticity $|\omega_0(x)|\leq\frac{C}{|x|^\alpha+1}$ for some constant $C>0$ so that the unique local odd solution to \eqref{E}-\eqref{IC} (with $f\equiv 0$) belonging to the class $C^{1,\alpha}_{x,t}([0,1)\times\mathbb{R}^3)$ satisfies \[\lim_{t\rightarrow 1}\int_0^t|\omega(s)|_{L^\infty}ds=+\infty.\]
\end{theorem}

\begin{remark}
The solution $\omega$ is exactly self-similar. That is, it takes the form: \[\omega(x,t)=\frac{1}{1-t}F\Big(\frac{x}{(1-t)^\xi}\Big)\] for some constant $\xi>0$. As $t$ approaches $t=1$ (from below), $\omega$ develops a singularity like $|x|^{-\beta}$ near the origin for some small $\beta>0$. 
\end{remark}

\begin{remark}\label{EnergyRemark}
The solutions of Theorem \ref{MainTheorem} have infinite energy and do not satisfy \eqref{EnergyBalance}; however, it is easy to show, using the preceding remark, that there are compactly supported classical solutions to \eqref{Vorticity}-\eqref{BSLaw} with compactly supported forcing $\nabla\times f\in C^{\alpha}([0,1]\times\mathbb{R}^3)$ satisfying \eqref{EnergyBalance} which become singular as $t\rightarrow 1$. Note that the force is uniformly $C^{1,\alpha}$ up to and including the time of blow-up. That is, if one allows for a \emph{uniformly} $C^{1,\alpha}$ external force, blow-up for finite energy solutions follows almost directly from the above result. 

More than this, we show in a joint work with T. Ghoul and N. Masmoudi \cite{EGM3dE} that these solutions can be localized to locally self-similar solutions with compactly supported vorticity and without external force since the blow-up is stable to certain kinds of perturbations that allow us to construct a $L^2\cap C^{1,\alpha}$ classical solution that becomes singular in finite time \cite{EGM3dE}. We will discuss this point in more detail on a model problem at the end of Section \ref{Examples}.

Besides, at a more heuristic level, this result should \emph{not} be confused with previous blow-up results for infinite energy solutions such as the ones in (\cite{Stuart1988},\cite{Con},\cite{GMS}); indeed, in all these cases the vorticity itself grows linearly at spatial infinity and the blow-up occurs on an infinite line or plane. The vorticity is decaying in our case and the blow-up occurs at a single point. This is what makes it possible to localize the blow-up. 
\end{remark}

\begin{remark}
The solutions of Theorem \ref{MainTheorem} that we construct here are axially symmetric and without swirl. It is known that sufficiently smooth (in particular, $C^\infty$) axi-symmetric solutions without swirl are globally regular; however, all the available global regularity results seem to require the velocity field to be at least $C^{1,\frac{1}{3}+}$ smooth. Heuristics suggest that this regularity threshold is actually sharp and that there exist axi-symmetric solutions in $C^{1,\frac{1}{3}-}$ which become singular in finite time. We also remark, importantly, that while the methods used here are applicable to axi-symmetric solutions without swirl, it is likely that they are also applicable in less rigid geometries and that in such settings one might be able to get much smoother solutions which develop singularities. 
\end{remark}

\subsection{Previous works on singularity formation}

There are numerous previous works on the global regularity problem and we will only discuss a few which are directly relevant to this work. A more extensive list of works can be found in the book \cite{MB}, the review papers \cite{Gibbon2008}, \cite{BardosTitiReview}, \cite{ConstantinReview}, and \cite{KiselevReview}, the numerical work \cite{HouLuo} as well as the author's work with I. Jeong \cite{EJE}. We will discuss three types of results here: blow-up criteria, infinite-time singularity formation, and model problems. We will not be discussing weak solutions in any detail but we refer the reader to the recent review papers \cite{DSReview} and \cite{BVReview}. 

The most well-known blow-up criterion is that of Beale, Kato, and Majda \cite{BKM} which we have already seen; it states that singularities in classical solutions occur if and only if the vorticity becomes unbounded. Another blow-up criterion is due to Constantin, Fefferman, and Majda \cite{CFM96} and dictates that if the velocity field remains uniformly bounded and the direction of the vorticity remains uniformly Lipschitz continuous up to time $T$, then there is no singularity at time $T$. This can be seen as a generalization of the global regularity for two-dimensional flows. Further advances in this direction have been made in \cite{JianHouYu}. Another line of work in the direction of ruling out singularities is devoted to self-similar singularities. That is, one postulates a form for the solution like \[\omega(x,t)=\frac{1}{(1-t)^{\alpha}}F(\frac{x}{(1-t)^\beta}).\] Then $F$ satisfies a time-independent equation which can be studied directly. Several authors have ruled out self-similar singularities for the Euler and Navier-Stokes equations (see \cite{Chae2007}, \cite{ChaeShv}, \cite{Tsai1998}, and \cite{NRS96}). In the case of the Euler equations, usually this is done under quite strong decay conditions on the vorticity. Since the profile we construct decays very slowly at spatial infinity, it does not contradict any of those results. 

In terms of results on singularity formation in the Euler equations, most of them have to do with infinite time singularity formation in two dimensions. We mention without details the results of Yudovich \cite{YLoss}, Nadirashvilli \cite{N}, Denisov (\cite{Denisov1}, \cite{Denisov2}), Kiselev and \v{S}ver\'ak \cite{KS}, and Zlato\v{s} \cite{Z}. There are also a few results on infinite time singularity formation in the 3D Euler equations such as \cite{YLoss}, \cite{EM1}, and \cite{TDo}. To the author's knowledge, the only result on finite-time singularity formation for finite-energy solutions to the 3D Euler equation prior to the present one is that of the author and I. Jeong \cite{EJE} on hour-glass shaped axi-symmetric domains with a sharp corner. It was shown that a natural local well-posedness theory can be established on those domains, but that solutions with (constant) finite energy could become singular in finite time. This was done by taking advantage of the scaling and rotational symmetries of the 3D Euler equation. It remains open whether those methods can be used to give a singularity on $\mathbb{R}^3$ though there seems to be some evidence that this can be done. The present work, however, follows a different philosophy which is closer to the study of simplified models of the Euler equation which we discuss next. 

Because the dynamics of solutions to \eqref{Vorticity}-\eqref{BSLaw} is still not well understood due to the many facets of the equations, many model equations have been devised to study some of the basic elements that make up the Euler equations. The first model problem we will discuss was introduced by Constantin, Lax, and Majda \cite{CLM} to investigate the amplifying effects of the vortex stretching term in a non-local model. For this model, almost all of the geometric properties of the vorticity equation are neglected, the advection term is neglected and we get:
\[\partial_t\omega=\omega\partial_x u.\] Moreover, the Biot-Savart law is replaced by \[u=(-\Delta)^{-\frac{1}{2}}\omega.\] After all these reductions, it is not surprising that the resulting model can be solved explicitly. Indeed, this was shown in \cite{CLM} and a necessary and sufficient condition for singularity formation for smooth and localized solutions was found. A skeptical observer might view these reductions as baseless, but the surprising fact is that these reductions turn out to be quite meaningful and serve as a motivation for the present work. We should remark that if one retains the advection term in the above model, not much is known about the equation though there have been a few recent advances on that problem (\cite{EJDG}, \cite{JSS}, \cite{LeiLiuRen}); it has been conjectured by several authors that retaining the advection term $u\partial_x\omega$ actually leads to global regularity (see Section \ref{Examples} for more on this point). One work in this direction which we are drawing inspiration from is \cite{EJDG} where it is shown that the regularizing effect of the advection term can be minimized by considering vorticity at $C^\alpha$ regularity with $\alpha$ small. 

After the numerical work of Luo and Hou \cite{HouLuo} and the work of Kiselev and \v{S}ver\'ak \cite{KS}, several other model problems related to the scenario in \cite{HouLuo} were considered (see \cite{Choi2017}, \cite{KRYZ}, for example). One of the ideas in these works is to study scenarios where the Biot-Savart law \eqref{BSLaw} can be decomposed into a main singular term and a more regular lower order term. This idea also informs what we do here. In addition to the above works, there have been also been a few recent works by T. Tao exploring singularities for other types of model problems and the possibility of finite-time singularity for the Euler equations on manifolds of high dimension (\cite{Tao2016}, \cite{TaoManifold}, \cite{Tao2}). 

\subsection{Classical vs. Smooth and $\mathbb{R}^3$ vs. $\mathbb{R}^3_+$}

It is important to say this directly: \emph{It is still open whether $C^\infty$ solutions to the incompressible Euler equation on $\mathbb{R}^3$ can develop a singularity in finite-time; we have merely shown singularity formation for $C^{1,\alpha}$ solutions for some $\alpha>0$.} Furthermore, the degree of regularity of solutions plays a key role in the construction presented here. It must also be emphasized, however, that this limitation on the regularity of the data can most likely be improved significantly in the presence of physical boundaries and by applying the methods to scenarios less rigid than zero-swirl axi-symmetric solutions (though the construction will have to be modified accordingly). Indeed, it is well known to specialists that if the vorticity of an Euler solution is non-zero on spatial boundaries, then this is analogous to considering solutions on $\mathbb{R}^3$ which have jumps in its vorticity (that is, the regularity of the velocity field would only be Lipschitz continuous on $\mathbb{R}^3$). A relevant case is when the domain is $\mathbb{R}^3_+$. Any solution to the incompressible Euler equation on $\mathbb{R}^3_+$ satisfying the (natural) no-penetration boundary condition can be extended to a solution on $\mathbb{R}^3$ by extending the first and second components of the vorticity as odd functions in the third variable, $x_3,$ and the third component of the vorticity as an even function in $x_3$. Likewise, any solution on $\mathbb{R}^3$ satisfying these symmetries can be restricted to $\mathbb{R}^3_+$ and will solve the Euler equation with the natural boundary condition. Consequently, if the first and second components of the vorticity of a solution on $\mathbb{R}^3_+$ do not vanish on $x_3=0$, it can actually be viewed as a solution on $\mathbb{R}^3$ which \emph{jumps} across the plane $x_3=0$. In this case, the regularity of the velocity field on $\mathbb{R}^3$ will not even be $C^1$. This point is also explored in the second example of Section \ref{Examples}. In this sense, it is not possible to compare blow-up on a smooth domain (when the vorticity is non-vanishing on the boundary), such as the one which is numerically predicted to occur in \cite{HL}, with the result of the present work. Each blow-up result has different advantages and deficiencies but both would answer fundamental questions, in my view. 

To wrap this point up, I should say that it is conceivable that some of the methods that already exist in the literature (including this work) could be used to produce an example of singularity formation for smooth solutions on a domain with smooth boundary (like $\mathbb{R}^3_+$) or even for $C^{1,\alpha}(\mathbb{R}^3)$ solutions for any $\alpha<1$. The global regularity problem for $C^\infty$ and localized solutions on $\mathbb{R}^3$, on the other hand, seems quite far as of now, though there are claims of numerical evidence for breakdown in that case as well (see \cite{Ker1}-\cite{Ker2} and \cite{LPTW}, for example).

\subsection{Organization}

The introductory material comprises the first three sections of this work. The first section is general. Section \ref{Setup} describes the exact setup of this work. Section \ref{Examples} provides a few simple examples which demonstrate the ideas behind this work. Section \ref{FundamentalModelSection} provides a basic analysis of the "Fundamental Model" which encodes the leading order dynamics of the type of solutions we are looking for as described in Section \ref{Setup}. Section \ref{LinearizedFundamentalModel} describes the coercivity of the linearization of the fundamental model around its self-similar solutions. Section \ref{LinearAngularTransport} gives the coercivity estimates for the linearization of the fundamental model along with the relevant angular transport term. Section \ref{EllipticEstimates} gives elliptic estimates which allow us to approximate the main non-local terms as described in Subsection \ref{Reductions}. Section \ref{H2} gives some useful information about the function spaces we are working in. In Section \ref{Modulation} we set up the exact equation for the perturbation to the solution of the fundamental model, prove the relevant a-priori estimates on the perturbation, and construct the full self-similar solution.

\subsection{Notation}\label{Notation}

In this subsection we give a guide to the notation used in the rest of the paper. 

\subsubsection{Functions, variables, and parameters}
With the exception of introductory parts of this work, $r$ will generally denote the two dimensional radial variable: \[r=\sqrt{x_1^2+x_2^2}.\]
$\theta$ will denote\footnote{Except in Section 2.1-2.2 where it can also be taken to denote the two dimensional polar angle.} the angle between $r$ and $x_3$: \[\theta=\arctan(\frac{x_3}{r}),\] so that $\theta=0$ corresponds to the plane $x_3=0$ while $\theta=\pm \frac{\pi}{2}$ corresponds to the $x_3$ axis. $\rho$ will denote the three dimensional radial variable \[\rho=\sqrt{r^2+x_3^2}.\] $R$ will denote $\rho^\alpha$: \[R=\rho^\alpha\] (where $\alpha>0$ is a constant which will be small). $z$, on the other hand, will generally denote the self-similar radial variable: \[z=\frac{R}{(1-(1+\mu)t)^{1+\lambda}}\] where $\lambda$ and $\mu$ are small constants. Functions in this paper will generally take the variables $z$ and $\theta$ or $R$, $\theta,$ and $t$ (dependence on $t$ is usually suppressed). Because the axial vorticity will be odd in the third variable, the $\theta$ variable will generally be in $[0,\pi/2]$ while the $z$ variable will usually be in $[0,\infty)$.  
The main parameters we will use are:
\[\eta=\frac{99}{100}, \qquad \alpha>0,\qquad \gamma=1+\frac{\alpha}{10}.\] $\alpha$ will be chosen at the end to be very small. 
In the later sections we use the functions \[\Gamma(\theta)=(\sin(\theta)\cos^2(\theta))^{\alpha/3}\] and \[K(\theta)=3\sin(\theta)\cos^2(\theta).\] Often there will be a constant $c$ and a constant $C$. The constant $c$ always satisfies $\frac{1}{10}\leq c\leq 10$ and it is a normalization constant. The constant $C$ will change from line to line but will be universal and independent of the main parameters $\alpha$ and $\gamma$. 

\subsubsection{Norms and Operators}
We first define the $L^2$ inner product: \[(f,g)_{L^2(\Omega)}=\int_{\Omega} f g\] and norm \[|f|_{L^2}=\sqrt{(f,f)_{L^2(\Omega)}},\] where $\Omega$ is the spatial domain. Often we will not write the subscript $L^2$ in the norm and/or the inner product and the meaning will have to be understood from context. For a bounded continuous function $f$, we define \[|f|_{L^\infty(\Omega)}=\sup_{x\in \Omega}|f(x)|.\] We also define the H\"older spaces using the norms:
\[|f|_{C^{\beta}(\Omega)}=\sup_{x\in\Omega} |f|+\sup_{x\not=y} \frac{|f(x)-f(y)|}{|x-y|^\beta}.\] If $f\in C^1$ we say that $f\in C^{1,\beta}$ when $\nabla f\in C^\beta$.
When the domain $\Omega$ is clearly understood from context, we often omit writing it.

{\emph{ Warning}:} In most of this paper, we will be working in some form of polar or spherical coordinates and will be using spaces like $L^2([0,\infty)\times [0,\pi/2])$ or similar spaces where the relevant variables are a radial and angular variable. The norm on this space is the usual $L^2$ norm with the measure $dr d\theta$ and not the measure $r dr d\theta$. 

We define the weights \[w(z)=\frac{(1+z)^2}{z^2},\]  \[w_\theta(\theta)=\frac{1}{\sin(2\theta)^\frac{\gamma}{2}},\]  and \[W=w \cdot w_\theta.\]
We also define the differential operators:
\[D_\theta(f)=\sin(2\theta)\partial_\theta f, \qquad D_R(f)=R\partial_R f,\] and
\[D_z(f)=z\partial_z f.\]
We define the space $\mathcal{H}$ by the norm: \begin{equation}\label{HNorm} |f|_{\mathcal{H}}=|f \frac{w}{\sin(2\theta)^{\eta/2}}|_{L^2}.\end{equation}
We define the $\mathcal{H}^k([0,\infty)\times[0,\pi/2])$ norm: \begin{equation}\label{HkNorm} |f|_{\mathcal{H}^k}^2=\sum_{i=0}^k|D_R^if \frac{w}{\sin^{\eta/2}(2\theta)}|_{L^2}^2+\sum_{0\leq i+j\leq k, i\geq 1}|D_R^jD_\theta^ifW|^2_{L^2}.\end{equation} 
We also define the $\mathcal{W}^{l,\infty}$ norm:
\[|f|_{\mathcal{W}^{l,\infty}}=\sum_{0\leq k\leq l}|(z+1)^k\partial_z^kf|_{L^\infty}+ \sum_{1\leq k+j\leq l, j\geq 1}|(z+1)^k\partial_z^kD_\theta^j f \frac{\sin(2\theta)^{-\frac{\alpha}{5}}}{\alpha+\sin(2\theta)}|_{L^\infty}.\] It is clear that any smooth function with sufficient $z$ decay belongs to $\mathcal{W}^{l,\infty}$ due to the inequality:
\[\sup_{x\in [0,1], \epsilon\in [0,1]}\frac{x^{1-\epsilon}}{\epsilon+x}\leq 1.\] The basic example of a $\mathcal{W}^{l,\infty}$ function is the function \[\Gamma(\theta) \frac{z}{(1+z)^2}.\]

Define the integral operator $L_{12}: L^2([0,\infty)\times [0,\pi/2])\rightarrow L^2([0,\infty))$ by \[L_{12}(f)(z)=\int_{z}^\infty \int_0^{\pi/2}f(r,\theta)\frac{K(\theta)}{r}d\theta dr.\] 

\subsubsection{Linearized Operators}
Also define the linearized operators $\mathcal{L}$, $\mathcal{L}_\Gamma$, and $\mathcal{L}_\Gamma^T$ as follows:
\[\mathcal{L}(f)=f+z\partial_z f-2\frac{f}{1+z},\]
\[\mathcal{L}_{\Gamma}(f)= f+z\partial_z f-2 \frac{f}{1+z}-\frac{2z\Gamma(\theta)}{c(1+z)^2}L_{12}(f),\]
and
\[\mathcal{L}_{\Gamma}^T(f)=\mathcal{L}_\Gamma(f)-\mathbb{P}(\frac{3}{1+z}\sin(2\theta)\partial_\theta f),\] where\[\mathbb{P}(f)(z,\theta)=f(z,\theta)- \frac{\Gamma(\theta)}{c}\frac{2 z^2}{(1+z)^3}L_{12}(f)(0).\]

\section{The Setup}\label{Setup}
A natural idea to use to establish singularity formation for solutions to the 3D Euler equation is to try to reduce as much as possible the complexity of the solutions we are studying. One of the simplest examples of three dimensional flows are the axi-symmetric flows without swirl. Such velocity fields are symmetric with respect to rotations which preserve the $x_3$ axis and have zero axial velocity (see \cite{MB} for more details). In this case, the vorticity equation and Biot-Savart law become the much simplified system \eqref{axisymmetric3DE}-\eqref{PsiTou} below. 

We start with the axi-symmetric 3D incompressible Euler equations (with vanishing swirl):

\[\partial_t \omega+u\cdot\nabla_{r,x_3}\omega=\frac{1}{r}u_r \omega, \]
where $\nabla_{r,x_3}=(\partial_r, \partial_{x_3})$ and $u=(u_r, u_3)$ is determined as follows. 
First we solve the elliptic problem\footnote{Note that the $-$ sign on the left hand side is not conventionally added, but there is no difference up to a change of variables.}:
\[\partial_r(\frac{1}{r}\partial_r\tilde\psi)+\frac{1}{r}\partial_{33}\tilde\psi=-\omega\] and then we set \[u_r=\frac{1}{r}\partial_3\tilde\psi\qquad u_3=-\frac{1}{r}\partial_r\tilde\psi.\] Next, in order to fix the homogeneity, we set $\tilde\psi=r\psi$. 
Then we have:
\[u_r=\partial_3\psi\qquad u_3=-\frac{1}{r}\psi-\partial_r\psi\] and 
\[\partial_r(\frac{1}{r}\partial_r(r\psi))+\partial_{33}\psi=-\omega,\] which leads us to the system:
\begin{equation}\label{axisymmetric3DE}
\partial_t \omega+u_r\partial_r\omega+u_3\partial_3\omega=\frac{1}{r}u_r \omega,
\end{equation}
\begin{equation}\label{3dBSL}
-\partial_{rr}\psi-\partial_{33}\psi-\frac{1}{r}\partial_r\psi+\frac{\psi}{r^2}=\omega,
\end{equation}
\begin{equation}\label{PsiTou}
u_r=\partial_3\psi\qquad u_3=-\frac{1}{r}\psi-\partial_r\psi.
\end{equation}
The problem is normally set on the spatial domain $\{(r,x_3)\in [0,\infty)\times (-\infty,\infty)\}$ and  the elliptic problem \eqref{3dBSL} is solved with the boundary condition $\psi=0$ on $r=0$. We will start by imposing an odd symmetry on $\omega$ with respect to $x_3$. That is, we search for solutions with: \[\omega(r,x_3)=-\omega(r,-x_3)\] for all $r,x_3$. Consequently, we may reduce to solving on the domain $[0,\infty)\times[0,\infty)$ while enforcing that $\psi$ vanish on $r=0$ and $x_3=0$ when solving \eqref{3dBSL}:
\begin{equation} \label{BCForPsi} \psi(r,0)=\psi(0,x_3)=0, \end{equation} for all $r,x_3\in [0,\infty)$. 
We note that with these conditions, the original $\tilde\psi$ actually vanishes quadratically on $r=0$. Note also, that for the full three dimensional vorticity to be $C^\infty$ a necessary condition is that $\omega$ vanish at least linearly on $r=0$. We are only interested in H\"older continuous solutions so we only impose that $\omega$ vanishes on $r=0$ for now. 

Let us make a few remarks about the system \eqref{axisymmetric3DE}-\eqref{PsiTou}. Since solutions to this system are automatically solutions to the $3D$ Euler equation, any $C^\alpha$ solution to \eqref{axisymmetric3DE} with sufficient decay at infinity and which vanishes on $r=0$ is a classical solution to the full $3D$ Euler system and thus falls into the range of applicability of the local well-posedness results of Lichtenstein \cite{Lich25} and Gunther \cite{Gunther}. Global well-posedness for this system has been established by Ukhovskii and Yudovich \cite{UY} under the \emph{additional assumption} that $\frac{\omega_0}{r}\in L^\infty.$ This assumption was later relaxed to $\frac{\omega_0}{r}\in L^{3,1}(\mathbb{R}^3)$ by Saint-Raymond \cite{SaintRaymond}, Abidi, Himidi, and Keraani \cite{AHK}, Shirota and Yanagisawa \cite{ShirotaY}, and Danchin \cite{Danchin} in various settings.  In particular, in the scale of H\"older spaces, global regularity of axi-symmetric solutions without swirl remained open if $u\in C^{1,\alpha}$ for $0<\alpha\leq \frac{1}{3}$. Here we construct a self-similar solution with a finite-time singularity when $\alpha$ is small. 

We will now proceed to explain how we are going to prove existence of a self-similar blow-up solution to \eqref{axisymmetric3DE}-\eqref{PsiTou}. The reader may find the following schematic helpful: \[\text{Full 3D Euler}\implies\text{Axisymmetric without swirl}\implies\text{Neglect the regular part of the singular integral}\]\[\implies\text{Remove the transport terms}\implies \text{Solve}\implies\text{Stability}\]

\subsection{Passing to a form of polar coordinates}

First we define $\rho=\sqrt{r^2+x_3^2}$ and $\theta=\arctan(\frac{x_3}{r})$ and set $R=\rho^\alpha$ for some (small) constant $\alpha>0$. Then we introduce new functions $\omega(r,x_3)=\Omega(R,\theta)$ and $\psi(r,x_3)=\rho^2 \Psi(R,\theta)$. We now show the forms of \eqref{axisymmetric3DE}, \eqref{3dBSL}, and \eqref{PsiTou} in the new coordinates.
Note that 
\[
\partial_r\rightarrow  \frac{\cos(\theta)}{\rho}\alpha R \partial_R-\frac{\sin(\theta)}{\rho} \partial_\theta \qquad \partial_3\rightarrow \frac{\sin(\theta)}{\rho}\alpha R\partial_R+\frac{\cos(\theta)}{\rho} \partial_\theta
\]

\subsubsection*{$u$ in terms of $\Psi$}
From \eqref{PsiTou} and the above facts we see:
\[u_r=\rho\Big(2\sin(\theta)\Psi+\alpha\sin(\theta)R\partial_R\Psi+\cos(\theta)\partial_\theta\Psi\Big)\] while \[u_3=\rho\Big(-\frac{1}{\cos(\theta)}\Psi-2\cos(\theta)\Psi-\alpha\cos(\theta)R\partial_R\Psi+\sin(\theta)\partial_\theta\Psi\Big)\]

\subsubsection*{Evolution equation for $\Omega$}
Observe that using the above calculations, \eqref{axisymmetric3DE} becomes 
\begin{equation}\label{OmegaEvolution}
\partial_t\Omega+(-3\Psi-\alpha R\partial_R\Psi)\partial_\theta\Omega+(\partial_\theta\Psi-\tan(\theta)\Psi)\alpha R\partial_R\Omega=\frac{1}{\cos(\theta)}\Big(2\sin(\theta)\Psi+\alpha\sin(\theta)R\partial_R\Psi+\cos(\theta)\partial_\theta\Psi\Big) \Omega.
\end{equation}
One can notice that the quantity $\frac{\Omega}{\cos(\theta)R^{\frac{1}{\alpha}}}$ (which is $\frac{\omega}{r}$) is exactly transported. 

\subsubsection*{Relation between $\Psi$ and $\Omega$}
After some calculations\footnote{See the calculation preceding \eqref{PolarBSL}.} \eqref{3dBSL} becomes: \begin{equation}\label{PolarBSL1}-\alpha^2R^2\partial_{RR}\Psi-\alpha(5+\alpha)R\partial_R\Psi-\partial_{\theta\theta}\Psi+\partial_\theta\big(\tan(\theta)\Psi\big)-6\Psi=\Omega.\end{equation} with the boundary conditions:
\[\Psi(R,0)=\Psi(R,\frac{\pi}{2})=0\] for all $R\in [0,\infty)$. 

\subsection{Reductions by taking $\alpha$ small and looking at $R=0$}\label{Reductions}
Up to now all we have done is a change of variables. Now we start to make reductions. First, by analyzing the equation \eqref{PolarBSL1} (according to the analysis done in Section \ref{EllipticEstimates}), we realize that \[\Psi=\frac{1}{4\alpha}\sin(2\theta)L_{12}(\Omega)+\text{lower order terms},\] with \[L_{12}(\Omega)=\int_{R}^\infty \int_{0}^{\frac{\pi}{2}}\Omega(s,\theta)\frac{K(\theta)}{s}dsd\theta,\] with $K(\theta)=3\sin(\theta)\cos^2(\theta)$. The idea behind this is that one first tries to derive $L^2$ estimates for solutions of \eqref{PolarBSL1}. If one multiplies by $\Psi$ and integrates, it becomes apparent that the a-priori estimate blows up as $\alpha\rightarrow 0$. This leads to studying \eqref{PolarBSL1} when $\alpha=0$. It then becomes apparent that $\sin(2\theta)$ is in the kernel of the operator \[L_0(\Psi)=-\partial_{\theta\theta}\Psi+\partial_\theta(\tan(\theta)\Psi)-6\Psi\] while $\sin(\theta)\cos^2(\theta)$ is the unique element of the kernel of the adjoint. Thus, a necessary (and sufficient) condition to solve $L_0(\Psi)=\Omega$ is that $\Omega$ is orthogonal to $\sin(\theta)\cos^2(\theta)$. When $\alpha>0$ there is no solvability condition but $\alpha$ independent bounds are gotten by first subtracting a specific term which is the main term in the expansion above. 

Next, we neglect all terms which vanish to quadratic order at $R=0$ \emph{and} contain a factor of $\alpha$. The reason we do this is that the equation which we will eventually get has self-similar blow-up which is stable under these kinds of perturbations. 
We thus write:
\[-3\Psi-\alpha R\partial_R\Psi\approx -\frac{3}{4\alpha}\sin(2\theta)L_{12}(\Omega), \qquad \partial_\theta\Psi-\tan(\theta)\Psi\approx \frac{1}{4\alpha}(2\cos(2\theta)-2\sin^2(\theta))L_{12}(\Omega)\]
\[\frac{1}{\cos(\theta)}\Big(2\sin(\theta)\Psi+\alpha\sin(\theta)R\partial_R\Psi+\cos(\theta)\partial_\theta\Psi\Big)\approx \frac{2}{4\alpha}L_{12}(\Omega)\]

After (time) scaling out a constant factor and neglecting the above-mentioned terms in \eqref{OmegaEvolution} we get:
\begin{equation}\label{OmegaEvolution2}
\partial_t\Omega-\frac{3}{2\alpha}\sin(2\theta)L_{12}(\Omega)\partial_\theta\Omega+L_{12}(\Omega)(\cos(2\theta)-\sin^2(\theta))R\partial_R\Omega=\frac{1}{\alpha}L_{12}(\Omega)\Omega.
\end{equation}

Notice that the second transport term on the left looks much smaller than the other two non-linear terms in the equation. The reason we have kept it is to balance the first transport term. Indeed, for this model, it is very likely that if $\Omega$ is smooth in $\theta$ there is global regularity. However, if one considers solutions which roughly behave like $R(\sin(\theta)\cos^2(\theta))^{\alpha/3}$ near $R=0$, a simple computation shows that the first two terms are annihilated to leading order. This is a key observation which now leads us to neglect the transport terms. 

\subsection{Dropping the transport term}

From the discussion above, if we view the solution $\Omega$ as being of the form: $\Omega(R,t,\theta)=(\sin(\theta)\cos^2(\theta))^{\alpha/3}\Omega_*(R,t)$, then the transport term becomes negligible in front of the term $L_{12}(\Omega)\Omega$ since \[|\sin(2\theta)\partial_\theta(\sin(\theta)\cos^2(\theta))^{\alpha/3}|\leq 2\alpha (\sin(\theta)\cos^2(\theta))^{\alpha/3}\] and $\alpha$ is small. For this reason, we drop the transport terms\footnote{Note that when we come to estimating the effects of dropping the transport terms we will only do so in an energy-type argument using integration by parts since otherwise we would incur a loss of derivatives.} and now study the equation: 
\begin{equation} \partial_t\Omega=\frac{1}{\alpha}L_{12}(\Omega)\Omega,\end{equation}
\begin{equation}L_{12}(\Omega)=\int_{R}^\infty \int_{0}^{\frac{\pi}{2}}\Omega(s,\theta)\frac{K(\theta)}{s}dsd\theta. \end{equation} This is what we call the fundamental model in this paper. It turns out that this equation possesses simple self-similar blow-up solutions which 
\begin{enumerate} 
\item have a fixed dependence on $\theta$ which can be freely chosen
\item are of order $\alpha,$  
\item are spectrally stable to perturbations which vanish quadratically at $R=0$.
\end{enumerate}
In particular, there are solutions to the fundamental model of the form: \[\Omega(R,\theta,t)=\tilde\Gamma(\theta) \frac{1}{1-t} F\Big(\frac{R}{1-t}\Big)\] for $F(z)=\frac{2z}{(1+z)^2}$ and for essentially \emph{any} $\tilde\Gamma$ (in particular we can take $\Gamma=c(\sin(\theta)\cos^2(\theta))^{\alpha/3}$ for some fixed constant $c>0$ close to $1$ uniformly in $\alpha$). 

(1) and (3) above is what allows us to indeed neglect the transport terms (to first order in $R$, we can choose an angular dependence which forces the transport terms to vanish). (2) and (3) is what allows us to neglect the rest of the terms. By carefully choosing the spaces where we are working, the preceding considerations can be made rigorous and the reductions can be justified. After all this is done, we thus prove existence of a self similar solution to \eqref{OmegaEvolution}-\eqref{PolarBSL1} near the one for the fundamental model with the angular dependence prescribed by the transport terms. 

\begin{remark}
It is important to mention the exact geometry of the solution constructed. Particles flow down the $x_3$ axis and outward on the $x_3=0$ plane. Because of the weak vanishing of vorticity on the axis of symmetry, vorticity accumulates near the origin and becomes infinite at the time of singularity.  
\end{remark}

\section{Three Examples}\label{Examples}

In this section we give examples of two equations with structure similar to the 3D Euler equation which highlight the effects of $C^\alpha$ regularity of the vorticity and/or the effects of spatial boundaries. We also give an example of how to continue a self-similar blow-up in a very simple model problem (which will be useful to understand the general scheme of the proof of Theorem \ref{MainTheorem}). 
The first two examples are based on the following general principle:

\vspace{3mm}

\emph{The vortex stretching term in \eqref{Vorticity} tends to cause vorticity growth while the advection term tends to deplete that growth. Thus, singularities should be found in scenarios where the depletion from advection is minimized.}

\vspace{3mm}

The following two examples show how low regularity in the vorticity or solid boundaries on which the vorticity does not vanish (which, as we mentioned, is essentially equivalent to a jump in the vorticity!) can stop the regularizing effect of the advection term. As far as the author knows, these are the only scenarios known to have this effect, but there may be others. We remark that the idea that the advection term in 3D Euler and Navier-Stokes is regularizing is present in work of Hou and Lei \cite{HouLei}. Also see work of Larios and Titi in this direction \cite{LariosTiti}.

\subsection{First Example}
 We consider the following active scalar model: 
\begin{equation}\label{FunnyModel1} \partial_t\omega+u\partial_x\omega=\omega\partial_x u \end{equation}
\begin{equation}\label{FunnyModel2} -\partial_{xx} u=\omega.\end{equation}
If we are solving this equation on $\mathbb{S}^1$, we should impose that $\int_{\mathbb{S}^1}\omega=0$ (which we may assume on the initial data). 
Now we recall from \cite{SarriaSaxton} that this system satisfies 
\begin{enumerate}
\item If we solve \eqref{FunnyModel1}-\eqref{FunnyModel2} in $[0,\pi]$ with the natural boundary conditions and if $\omega$ is non-vanishing on $[0,\pi]$, then $\omega$ may become singular in finite time. 
\item If we solve \eqref{FunnyModel1}-\eqref{FunnyModel2} on $\mathbb{S}^1$ with $C^\alpha$ data for some $\alpha<1$, the unique local solution may become singular in finite time. 
\item If we solve \eqref{FunnyModel1}-\eqref{FunnyModel2} on $\mathbb{S}^1$ with $C^1$ data, the solution is global. 
\end{enumerate}
These points lead us to the following conclusion: 
\begin{itemize}
\item Either by imposing solid boundaries at which the vorticity does not vanish or by taking the vorticity to vanish on that spatial boundary to order $|x|^\alpha$ for some $\alpha<1$, the regularizing effect of the advection term can be overcome. Otherwise, solutions are global due to the regularizing effect of the advection term.
\end{itemize}

\subsection{Second Example}

We now present a second example which can be seen as the motivation for this whole work. 
Consider the following 2D system 
\begin{equation}\label{LocalModel1}\partial_t\omega-(x_1\lambda(t),-x_2\lambda(t))\cdot\nabla\omega=\partial_1 \rho \end{equation}
\begin{equation}\label{LocalModel2} \partial_t\rho-(x_1\lambda(t),-x_2\lambda(t))\cdot\nabla \rho=0, \end{equation}
\begin{equation}\label{LocalModel3} \lambda(t)=\int_{\mathbb{R}^2} \frac{y_1y_2}{|y|^4}\omega(y,t)dy. \end{equation}
This can be seen as a local model of the dynamics of solutions to the 3D axi-symmetric Euler equation (with swirl) away from the axis of symmetry near a hyperbolic stagnation point which we take to be $(0,0)$. We remark that this also serves as a toy model of the scenario discovered in the numerical work \cite{HouLuo}. We consider solutions with $\omega$ odd in $x_1$ and $x_2$ separately and $\rho$ odd in $x_2$ and even in $x_1$. For such solutions, we have the following:
\begin{enumerate}
\item If $\omega_0,\partial_1\rho_0\in C^2_c(\mathbb{R}^2)$ the unique local solution to \eqref{LocalModel1}-\eqref{LocalModel3} is global.
\item There exist $\omega_0,\partial_1\rho_0\in C^\infty_c(\mathbb{R}^2_+)$ so that the unique solution to \eqref{LocalModel1}-\eqref{LocalModel3} develops a singularity in finite time. 
\end{enumerate}
The proof of both statements follows essentially by solving the equation. First we introduce \[\mu(t)=\exp\Big(\int_0^t \lambda(s)ds\Big).\] For simplicity we assume that $\omega_0\equiv 0$ (this assumption can be easily removed). Then we see that the unique local solution of \eqref{LocalModel1}-\eqref{LocalModel3} can be written as:
\[\omega(x_1,x_2,t)=\partial_1\rho_0(\mu(t)x_1,\frac{x_2}{\mu(t)})\int_0^t\mu(s)ds.\]
Consequently, 
\[\frac{\dot\mu(t)}{\mu(t)}=\Big(\int_0^t\mu(s)ds\Big)\Big(\int_0^\infty\int_0^\infty\frac{y_1y_2}{|y|^4}\partial_1\rho_0(\mu(t)y_1,\frac{y_2}{\mu(t)})dy_1dy_2\Big).\]

\emph{Case 1: Smooth on $\mathbb{R}^2$:} 

The important point is that when $\rho$ is $C^2$ and compactly supported on $\mathbb{R}^2$, we must have that \[|\partial_1\rho_0(x_1,x_2)|\leq |x_1x_2|D(x_1,x_2)\] for $x_1, x_2$ small for $D$ a uniformly bounded and compactly supported function. 
Consequently, \[\Big|\frac{\dot\mu(t)}{\mu(t)}\Big|\leq C\int_0^t \mu(s)ds \int_0^{\infty} \int_0^\frac{A}{\mu(t)} \frac{(y_1y_2)^2}{|y|^4}dy_1dy_2\leq C\int_0^t\mu(s)ds \int_0^{\frac{A}{\mu(t)}} y_1dy_1\leq \frac{C}{\mu(t)^2}\int_0^t\mu(s)ds.\]
Thus, $\mu$ remains bounded for all finite times. 

\emph{Case 2: Smooth on $\mathbb{R}^2_+$}.

Now let's look at the case when $\rho_0$ is just smooth on $\mathbb{R}^2_+$ and not vanishing on $x_2=0$. We take  $\partial_1\rho_0(x_1,x_2)$ to be a smooth odd-in-$x_1$ function on $(-\infty,\infty)\times[0,\infty)$ equal to $x_1$ on $[0,1]^2$, vanishing outside of $[0,2]^2$ and non-negative on $[0,\infty)^2$. In this case, we again have:
\[\frac{\dot\mu(t)}{\mu(t)}=\Big(\int_0^t\mu(s)ds\Big)\Big(\int_0^\infty\int_0^\infty\frac{y_1y_2}{|y|^4}\partial_1\rho_0(\mu(t)y_1,\frac{y_2}{\mu(t)})dy_1dy_2\Big)\geq\Big(\int_0^t\mu(s)ds\Big)\mu(t)\Big(\int_0^{\mu(t)}\int_0^{\frac{1}{\mu(t)}}\frac{y_1^2y_2}{|y|^4}dy_1dy_2\Big)\]
\[=\Big(\int_0^t\mu(s)ds\Big)\mu(t)\Big(\int_0^{\mu(t)}\int_0^{\frac{1}{\mu(t)}}\frac{y_1^2y_2}{|y|^4}dy_1dy_2\Big)\] \[=\frac{1}{2}\Big(\int_0^t\mu(s)ds\Big)\mu(t)\int_0^{\frac{1}{\mu(t)}}y_1^2\Big(\frac{1}{y_1^2}-\frac{1}{y_1^2+\mu(t)^2}\Big)dy_1\geq \frac{1}{10}\int_0^t\mu(s)ds,\] so long as $\mu(t)\geq 1$. 
Thus, 
\[\dot\mu(t)\geq c\mu(t)\int_0^t\mu(s)ds\] for some fixed $c>0$. 
Since $\mu(0)=1$, $\mu$ becomes infinite in finite time. 
\begin{remark}
The above calculation shows that if $\rho_0$ vanishes to order $y^\alpha$ at $y=0$ with $\alpha$ sufficiently small, then there will still be singularity in finite time on $\mathbb{R}^2$. 
\end{remark}

\subsection{Stable singularity formation in the simplest setting}

In this subsection we explore the problem of finite time-singularity formation in the ODE:
\begin{equation}\label{ODEBlowUp} \partial_t f = f^2+\epsilon N(f),\end{equation} for $x,t\in [0,\infty)$. Here  $N$ is a quadratic non-linearity with total degree zero\footnote{By this we mean that for $\lambda>0$, if $f_\lambda(\cdot)=f(\lambda \cdot),$ then $N(f_\lambda)=N(f)_\lambda$} satisfying some natural conditions and $\epsilon$ is a small constant. The question we wish to consider is: how can we efficiently show that the blow-up for the $\epsilon=0$ problem persists when $\epsilon>0$. Our goal is to have a method which is flexible enough to handle non-linearities $N$ which may include derivatives and non-local operators at least. This is, admittedly, just an exercise, so those familiar with these types of questions and the methods to solve them can skip this part all together. 

We first observe that when $\epsilon=0$ we have the self similar solution:
\[f(x,t)=\frac{1}{1-t}F_*(\frac{x}{1-t}),\] with $F_*(z)=\frac{1}{1+z}$. Now we search similarly for a solution to \eqref{ODEBlowUp} now of the form
\[f(x,t)=\frac{1}{1-(1+\mu) t} F(\frac{x}{(1-(1+\mu) t)^{1+\lambda}}).\]
If $\epsilon$ is small, we should think that $\mu,\lambda= O(\epsilon)$ and $F= F_*+O(\epsilon)$. 
We see that
\[(1+\mu) F+(1+\lambda)(1+\mu)z\partial_z F=F^2 +\epsilon N(F).\]
Now we write $F=F_*+g.$
Then, 
\begin{equation}\label{BabyModulation} g+z\partial_z g-\frac{2}{1+z} g=-\mu F_*-(\mu+\lambda+\lambda\mu) z \partial_z F_*-(\mu+\lambda+\lambda\mu) g-(\mu+\lambda+\lambda\mu) z\partial_z g+\epsilon N(F_*+g)+g^2.\end{equation}
We re-write this as:
\[\mathcal{L}(g)=-\mu F_*-(\mu+\lambda+\lambda\mu) z \partial_z F_*-(\mu+\lambda+\lambda\mu) g-(\mu+\lambda+\lambda\mu) z\partial_z g+\epsilon N(F_*+g)+g^2.\]
Now, after studying $\mathcal{L}=\text{Id}+z\partial_z -\frac{2}{1+z}\text{Id}$ a little bit, it becomes apparent that $\mathcal{L}$ is a coercive operator on the weighted $L^2$ space with weight $\frac{(z+1)^2}{z^2}$ and properly weighted $H^s$ spaces. In particular, in order to solve for $g$, we need $g$ and the RHS of \eqref{BabyModulation} to vanish at least quadratically at $z=0$. We just choose $\mu$ and $\lambda$ so that the right hand side vanishes quadratically at $z=0$ assuming that $g$ itself vanishes quadratically. In particular, evaluating the right hand side at $z=0$ to $0^{th}$ and $1^{st}$ order we see:
\[0=-\mu+\epsilon N(F_*+g)(0) \]
\[0=\mu+(\mu+\lambda+\lambda\mu)+\epsilon\partial_z(N(F_*+g))(0).\]
In particular, we can solve for $\mu$ and $\lambda$ explicitly in terms of $\epsilon N(F_*+g)$. 
Making the above choices for $\mu$ and $\epsilon$, we now observe that if $X$ is a weighted $H^1$ space in which $\mathcal{L}$ satisfies \[(\mathcal{L} g,g)_{X}\geq c|g|_{X}^2,\] then we have from \eqref{BabyModulation} that
\[c|g|_{X}^2\leq C\epsilon(|g|_{X}+|g|_{X}^2)+|g|_{X}^3,\] which yields the a-priori estimate:
\[|g|_{X}\leq C\epsilon,\] when $\epsilon$ is small enough.

One difficulty we will face is that the $\epsilon=0$ problem in our setting is actually the "fundamental model" which we describe in the coming section which is, itself, non-local and multivariable. As can be expected, the linearized operator requires a non-local condition to be coercive. For this reason, the actual argument is not as simple as the one above, but there is a a lot to be gained from studying \eqref{ODEBlowUp} and \eqref{BabyModulation} first. 

\subsubsection{Localization}

Taking the above example further, we want to explain how the above stability at the level of self-similar solutions can be used to establish non-linear stability of the blow-up (and, ultimately, a finite-energy version of the main theorem). As remarked above, the full details of this argument for the Euler equation is given in the joint work with Ghoul and Masmoudi \cite{EGM3dE}, but here we explain how this can be done on the above simple model. 

Let us start with 
\begin{equation}\label{Ric} 
\partial_t f=f^2. 
\end{equation} 
We know that there is a self-similar solution. Now let us suppose that we want to know whether this solution is "stable" among general solutions (not just self-similar solutions to nearby problems). We formalize this as follows. Let us assume that nice functions $\bar\mu(t), \bar\lambda(t)$ are given and let us write the solution to \eqref{Ric} as:
\[f(x,t)=\frac{1}{\bar\lambda(t)}F\Big(\frac{\bar\mu(t)x}{\bar\lambda(t)},t\Big).\] Let us define the variable \[z=\frac{\bar\mu(t)}{\bar\lambda(t)}x,\qquad s=\frac{1}{\bar\lambda(t)}.\] The above change of variables can always be done (locally in time) under mild assumptions on $\bar\lambda,\bar\mu$. Now let us see what equation we get for $F$:
\begin{equation}
\partial_s F-\bar\lambda'(t)F-\Big({\bar\lambda'(t)}-\frac{\bar\lambda(t)^2\bar\mu'(t)}{\bar\mu(t)}\Big)zF'=F^2.
\end{equation}
Our self-similar solution is a solution triple:
\[(F(z,t),\bar\lambda(t),\bar\mu(t))=(F_*(z),1-t,1).\]
Now all we wish to know is whether this solution triple is spectrally stable. Looking at solutions of the form $(F_*, 1-t, 1)$ we get:
\[\partial_s g+g+zg'-2F_* g=g^2+\lambda'(t)F_*-(1-t)^2\mu'(t)zF_*',\] which is just 
\[\partial_s g+\mathcal{L}(g)=\lambda'(t)F_*-(1-t)^2\mu'(t)zF_*'.\]
From the preceding discussion, we know that $\mathcal{L}$ is positive on a certain space; the (free) constants $\lambda(t)$ and $\mu(t)$ are then used to correct the solution so that it stays in the correct space. The proof then becomes akin to showing asymptotic stability of a stable fixed point in an ODE. 

We will not pursue this example any further but what should be clear is that the heart of the matter in is understanding the linearized operator $\mathcal{L}$. For the Euler equation there is a similar linear operator (what we call $\mathcal{L}_\Gamma^T$ below). The preceding discussion, of course, is just a very rough sketch of the idea and the interested reader is encouraged to study \cite{EGM3dE} for more information.

\section{The Fundamental Model}\label{FundamentalModelSection}

In this section we describe the basic model which we will use to approximate some solutions to the Euler equation. First we describe how the model originally came about and then we exhibit the specific solutions to the model which we will be using later on. We remark that this model can also be used to model situations other than the one discussed in Section \ref{Setup}; in fact, I believe that some form of this model is also behind the singularity in the numerical work \cite{HouLuo}.

\subsection{Origin}
We begin by introducing the model:
\begin{equation}
\label{FundamentalModel}
\partial_t f(\rho,\theta, t)=f(\rho,\theta, t)L_{12}^K(f)(\rho,t),\qquad L_{12}^K(f)(\rho,t)=\int_{\rho}^\infty\int_0^{2\pi}\frac{K(\theta')f(s,\theta',t)}{s}d\theta' ds,
\end{equation} with $K$ some $2\pi$ periodic function whose identity we will discuss later. The choice of $K$ really depends upon the scenario we are trying to model; specifically, what kind of symmetries we impose on the vorticity.
One can view this model in the spirit of the Constantin-Lax-Majda model \cite{CLM} for the vortex stretching term in the 3D Euler equation: \[\partial_t f=f H(f),\] where $H$ is the Hilbert transform. To arrive at this model, one first builds a more realistic model\footnote{We remark that this type of model appeared in a work of Constantin and Sun \cite{ConstantinSun} and a note of A. Kiselev in the list of open problems \cite{problems}.} \[\partial_t f=f R_{12} f,\] where $R_{12}$ is just the singular integral operator with Fourier symbol $-\frac{\xi_1\xi_2}{|\xi|^2}$. The advantage of this model is that the non-linearity \[f R_{12}f\] appears in some form in the vortex stretching term of the actual 3D Euler equation. It appears that $C^\infty$ solutions to this model which are odd in $x$ and $y$ separately and non-negative on $[0,\infty)^2$ become singular in finite time (though this remains open). One disadvantage of this model is that it seems much more difficult to analyze than the Constantin-Lax-Majda model. However, in the odd scenario described above, it turns out that when we replace $R_{12}$ by $L_{12}^K$ (when $K=\sin(2\theta)$), the problem becomes solvable again. Moreover, replacing $R_{12}$ by $L_{12}^K$ (with $K=\sin(2\theta)$) is actually justifiable! This is an important observation which has its origins in the work of Kiselev and \v{S}ver\'ak \cite{KS} and further refinements in previous works of the author \cite{E1} and the author and Jeong (\cite{EJB}, \cite{EJE},\cite{EJVP}). 

\subsection{Analysis}
Now we turn to a basic analysis of \eqref{FundamentalModel} and the main result here is Lemma \ref{FMSelfSimilar}.
First, in order to get local well-posedness for solutions to \eqref{FundamentalModel}, we should only search for solutions which vanish at $\rho=0$, which are at least H\"older continuous in $(\rho,\theta)$ (and thus on $\mathbb{R}^2$), and which vanish at infinity like $\rho^{-\delta}$ for some $\delta>0$. It is not difficult to establish local well-posedness in this class by using that the mapping $L_{12}^K$ is a bounded operator on the class of functions we just described (on $\mathbb{R}^2$ it can be viewed as local well-posedness on the class $C^\alpha \cap L^p$ for some $0<\alpha<1$ and $p<\infty$). 

Next, it is not difficult to see that smooth solutions to this equation can become singular in finite time. Indeed, upon multiplying both sides of the equation by $\frac{K(\theta)}{\rho}$ and integrating on the whole space we see that: \[\frac{d}{dt}L_{12}^K(f)(0)=\frac{1}{2}L_{12}^K(f)(0)^2.\] However, in order to use solutions to this equation to approximate some solutions to the Euler equation, it is necessary to get a finer understanding of the blow-up behavior.

We now show how to solve \eqref{FundamentalModel} explicitly. This is not difficult to achieve since $L_{12}^K(f)$ is a radial function and thus it is possible to reduce this problem to an ODE. Indeed, upon multiplying by $\frac{K(\theta)}{s}$ and now integrating on the region $[\rho,\infty)\times [0,2\pi]$ we see that: \[\partial_t L_{12}^Kf(\rho,t)=\frac{1}{2}L_{12}^Kf(\rho,t)^2.\] This gives us a formula for $L_{12}^Kf$ in terms of $L_{12}^Kf_0.$ Then we further have that \[f=f_0\exp\Big(\int_0^tL_{12}^Kf(\cdot,s)ds\Big)=f_0\exp\Big(\int_0^t \frac{L_{12}^Kf_0}{1-\frac{1}{2}sL_{12}^Kf_0}ds\Big)=\frac{f_0}{(1-\frac{1}{2}tL_{12}^Kf_0)^2}.\]

In fact, it is not so important for us that this problem is explicitly solvable. What is important for us is that it possesses many families of self-similar blow-up solutions (which are, of course, easy to find when we have a solution formula!).
One such family is described in the following. We search for solutions of the form \[f(\rho,\theta,t)=\frac{\Gamma(\theta)}{c}\frac{1}{1-t}F_{*,rad}(\frac{\rho}{1-t}),\] where $c=\int_{0}^{2\pi}\Gamma(\theta)K(\theta)d\theta$ and where \[F_{*,rad}(z)=\frac{2z}{(1+z)^2}.\] We see that, after plugging the ansatz into \eqref{FundamentalModel}, for this to be truly a self-similar solution $F_*$ should satisfy: \[F_{*,rad}+z\partial_z F_{*,rad}=F_{*,rad}\int_z^\infty \frac{F_{*,rad}(\rho)}{\rho}d\rho.\]
Now plugging in $F_{*,rad}=\frac{2z}{(1+z)^2}$ we note: \[F_{*,rad}+z\partial_z F_{*,rad}=\frac{2z}{(1+z)^2}+\frac{2z}{(1+z)^2}-\frac{4z^2}{(1+z)^3}=\frac{4z}{(1+z)^3}\]\[=F_{*,rad}(z)\int_z^\infty\frac{F_{*,rad}(\rho)}{\rho}d\rho.\] Consequently, we get the following lemma.
\begin{lemma}\label{FMSelfSimilar}
The fundamental model \eqref{FundamentalModel} possesses a family of self similar solutions of the form: \[f(r,\theta,t)=2\alpha\frac{\Gamma(\theta)}{c} \frac{1}{1-t}F_{*,rad}\Big(\frac{r^\alpha}{1-t}\Big),\] where \[F_{*,rad}(z)=\frac{z}{(1+z)^2},\] $K\Gamma\in L^1([0,2\pi]),$ and \[c=\int_{0}^{2\pi}K(\theta)\Gamma(\theta)d\theta\] whenever $c\not=0$ and $\alpha>0$. 
\end{lemma}

\subsubsection*{Specification of $K$ and angular domain}
In this work, we will be working in a situation where \[K(\theta)=3\sin(\theta)\cos^2(\theta)\] and the spatial domain is $[0,\infty)\times [0,\pi/2].$ For this reason, from now on, we will take \[L_{12}(f)=\int_{r}^{\infty}\int_{0}^{\pi/2}\frac{3f(\rho,\theta)\sin(\theta')\cos^2(\theta')}{\rho}d\theta'd\rho\] and \[c=\int_{0}^{\pi/2}K(\theta)\Gamma(\theta)d\theta.\] 

\section{Linearization of the Fundamental Model in Self-Similar Variables} \label{LinearizedFundamentalModel}

By solving the system \eqref{FundamentalModel} directly, it is not difficult to see that the solutions described in Lemma \ref{FMSelfSimilar} are \emph{stable} in that there are open sets of functions which all blow-up in the same way. Since we will not be able to solve explicitly after this section, it will be more useful to see this stability in terms of spectral properties of the linearization around the self similar solutions of Lemma \ref{FMSelfSimilar}. First we will define the relevant linear operator, then we will discuss its coercivity properties which are motivated by previous work with Ghoul and Masmoudi \cite{EGM} and Jeong \cite{EJDG}. The main result of this subsection is Proposition \ref{prop:L2Coercivity} which shows coercivity of the relevant linear operator in a weighted $L^2$ space. 
 
\begin{definition}\label{Definition_Linearization}
We define the operators \[\mathcal{L}_{\Gamma}(f)= f+z\partial_z f-2 \frac{f}{1+z}-\frac{2z\Gamma(\theta)}{c(1+z)^2}L_{12}(f),\]
\[\mathcal{L}(f)=f+z\partial_z f-2\frac{f}{1+z}.\]
\end{definition}

To study the coercivity properties of $\mathcal{L}$ and $\mathcal{L}_\Gamma$, we begin by defining the weight function $w$ which will be used throughout the paper. 
\begin{definition}
Define $w:(0,\infty)\rightarrow(1,\infty)$ by: \[w=\frac{(1+z)^2}{z^2}.\]
\end{definition}
Next, we have the following useful lemma. 
\begin{lemma}
We have that \begin{equation}\label{L12L}  L_{12}\Big(\mathcal{L}_\Gamma(f)\Big)=\mathcal{L}\Big(L_{12}(f)\Big)\end{equation}
 \begin{equation}\label{LW} \mathcal{L}(g)w=gw+z\partial_z(gw).\end{equation}
\end{lemma}
\begin{proof}
Both of these statements are simple computations which we give now.  
To show \eqref{L12L} we compute directly:
\[L_{12}(\mathcal{L}_\Gamma(f))= L_{12}\Big(f+z\partial_z f-2 \frac{f}{1+z}-\frac{2z\Gamma(\theta)}{c(1+z)^2}L_{12}(f)\Big)\]
\[=L_{12}(f)+z\partial_z L_{12}(f)-\frac{2}{1+z}L_{12}(f).\] 
For \eqref{LW} we have that \[\mathcal{L}(g)w=gw+z\partial_z (gw)-\frac{2}{1+z}gw_z-g z\partial_z w=gw+z\partial_z (gw)-\frac{2}{1+z}gw+g z\Big(\frac{2}{z^2}+\frac{2}{z^3}\Big)=gw +z\partial_z(gw).\]
\end{proof}

We need the following Hardy-type inequality.

\begin{lemma}\label{HardyWeight}
Assume $fw\in L^2$ and $L_{12}(f)w\in L^2$. Then, \[|L_{12}(f)w|_{L^2}\leq 4|fw|_{L^2}.\]
\end{lemma}

\begin{proof}
We will establish the result for smooth functions with $f$ and $L_{12}(f)$ vanishing (at least) quadratically at zero. The general case will follow by approximation. 
Note that $w(z)^2=\frac{1}{z^4}+\frac{2}{z^2}+1$. Thus, it suffices to prove that \[|L_{12}(f)|_{L^2}^2\leq 4|f|_{L^2}^2,\qquad |z^{-1}L_{12}(f)|_{L^2}^2\leq 4|z^{-1}f|_{L^2}^2, \qquad |z^{-2}L_{12}(f)|_{L^2}^2\leq 4|z^{-2}f|_{L^2}^2.\]
We leave the first two to the reader and establish the last one. 
\[\int z^{-4}L_{12}(f)^2=\Big|\frac{1}{3}\int \partial_z z^{-3}L_{12}(f)^2\Big|=\Big|\frac{2}{3}\int z^{-3}L_{12}(f)\partial_z L_{12}(f)\Big|\] \[=\frac{2}{3}\int z^{-4} L_{12}(f)(z) \int_{0}^{\pi/2}K(\theta)f(z,\theta)d\theta dz\leq \frac{2}{3}|K|_{L^2}|z^{-2}L_{12}(f)|_{L^2}|z^{-2}f|_{L^2}<\frac{2}{3}|z^{-2}L_{12}(f)|_{L^2}|z^{-2}f|_{L^2}.\]
\end{proof}

We make the following observation about the function $\Gamma$: 
\begin{equation}\label{GammaAssumption} |\frac{\Gamma}{c}-K|_{L^2[0,\pi/2])}\leq \frac{7}{10}.
\end{equation}  Recall that $\Gamma$ takes the form $\Gamma=(\sin(\theta)\cos^2(\theta))^\beta$ for some  $0\leq\beta\leq 1.$ The fact that such examples actually satisfy this assumption is a simple exercise which is easiest to check when $\beta=0$ and $\beta=1$. 

We now proceed to establish weighted $L^2$ coercivity estimates on $\mathcal{L}_\Gamma$. 

\begin{proposition}
\label{prop:L2Coercivity}
We have that \begin{equation}\label{L2Coercivity} (\mathcal{L}_\Gamma(f)w, fw)_{L^2}\geq \frac{1}{4}|fw|_{L^2}^2. \end{equation}
\end{proposition}
\begin{proof}
Observe that \[\mathcal{L}_\Gamma(f)=\mathcal{L}(f)-\frac{2\Gamma z}{c(1+z)^2}L_{12}(f).\]
Thus, 
\[
(\mathcal{L}_\Gamma(f)w,fw)_{L^2}=\Big(\mathcal{L}(f)w,fw\Big)_{L^2}-2\Big(K\frac{z}{(1+z)^2} f w^2, L_{12}(f)\Big)_{L^2}-2\Big((\frac{\Gamma}{c}-K)\frac{z}{(1+z)^2} f w^2, L_{12}(f)\Big)_{L^2}
\]
\[
=\frac{1}{2}|fw|^2_{L^2}-2\Big(K\frac{f}{z} w, L_{12}(f)\Big)_{L^2}-2\Big(f w, (\frac{\Gamma}{c}-K)\frac{L_{12}(f)}{z}\Big)_{L^2}
\]
\[=\frac{1}{2}|fw|^2_{L^2}+\Big(\partial_z\big(L_{12}(f)^2\big) , w\Big)_{L^2_z}-2\Big(f w, (\frac{\Gamma}{c}-K)\frac{L_{12}(f)}{z}\Big)_{L^2}\]
\[\geq \frac{1}{2}|fw|_{L^2}^2-(L_{12}(f)^2,\partial_z w)_{L^2_z}-2|fw|_{L^2}|\frac{\Gamma}{c}-K|_{L^2(\mathbb{S}^1)}|\frac{1}{z}L_{12}(f)|_{L^2_z}\]
\[\geq \frac{1}{2}|fw|_{L^2}^2+2|\frac{1}{z}L_{12}(f)|_{L^2_z}^2-\frac{7}{5}|fw|_{L^2}|\frac{1}{z}L_{12}(f)|_{L^2_z},\]
where we used \eqref{LW} and the definition of $w$ in the second equality, the definition of $L_{12}$ in the third equality, integration by parts in the first inequality, and \eqref{GammaAssumption} in the second inequality. 

Since $\Big(\frac{7}{5}\Big)^2<4(\frac{1}{4})(2),$ we have
\[(\mathcal{L}_\Gamma(f)w,fw)_{L^2}\geq \frac{1}{4}|fw|_{L^2}^2. \] 
\end{proof}

\section{Linearization with angular transport}\label{LinearAngularTransport}

To move back toward the Euler equation from the fundamental model as explained in Section 2 (read backwards), it will be necessary to also study the coercivity properties of the following operator which is the same as $\mathcal{L}_\Gamma$ but with an extra transport term in the angular direction. We begin by defining this operator in Definition \ref{Def:LinearOperator}. The goal of this section will then be to prove that $\mathcal{L}_\Gamma^T$ is coercive on $\mathcal{H}^k$ as is explained in Proposition \ref{prop:HkCoercivity} below and the remarks preceding it. Note that the weights in the definition of $\mathcal{H}^k$ (as in \eqref{HkNorm}) are chosen to have favorable properties when we take the inner product between $\mathcal{L}_\Gamma^T(g)$ and $g$ when $g\in\mathcal{H}^k$. The use of the spatial weights is clear from Proposition \ref{prop:L2coercivity} and the importance of the angular weights $\frac{1}{\sin(2\theta)^{99/100}}$ and $\frac{1}{\sin(2\theta)^\gamma}$ are that they will allow us to hide the effect of the angular transport term. The reason for having two different angular weights comes from the elliptic estimates as is explained in Section \ref{H2Elliptic}. 

\begin{definition}\label{Def:LinearOperator}
We define the following operator acting first on $C^1_c$ functions:\[\mathcal{L}_{\Gamma}^T(f)=\mathcal{L}_\Gamma(f)-\mathbb{P}(\frac{3}{1+z}\sin(2\theta)\partial_\theta f),\]
\end{definition}
\noindent where $\mathbb{P}$ is an operator which we will now define. 
First recall that \[ \Gamma(\theta)= \Big(\sin(\theta)\cos^2(\theta)\Big)^{\alpha/3}\qquad c=\int_{0}^{\pi/2}\Gamma(\theta)K(\theta)d\theta.\]
\begin{definition}
For $f\in \mathcal{H}$ we define \[\mathbb{P}(f)(z,\theta)=f(z,\theta)- \frac{\Gamma(\theta)}{c}\frac{2 z^2}{(1+z)^3}L_{12}(f)(0).\]
\end{definition}
\begin{remark}
Note that $L_{12}(\mathbb{P}(f))(0)=0$ for every $f$.  The reason for including the projector $\mathbb{P}$ in the definition of $\mathcal{L}_\Gamma^T$ is that we want to be able to say that if $g$ vanishes quadratically at $z=0$ and $L_{12}(g)(0)=0$ then the same can be said about $\mathcal{L}_\Gamma^T(g)$. Let us also note that many projectors could have been chosen to achieve these properties, but this is the only one which also arises naturally from relaxing certain scaling parameters in the problem (see the calculation preceding equation \eqref{gEquationFinal}).
\end{remark}
Observe the pointwise inequality
\begin{equation}\label{GammaAssumption2} |\sin(2\theta)\partial_\theta \Gamma|\leq 2\alpha \Gamma. \end{equation} 
To avoid cumbersome notation, we define the operator \[D_\theta:=\sin(2\theta)\partial_\theta.\]
We also define 
\begin{equation} 
\label{gamma} \gamma=1+\frac{\alpha}{10}.
\end{equation}

\subsection{$L^2$ coercivity for $\mathcal{L}_\Gamma^T$ with one $\theta$-derivative}

We begin with an $L^2$ estimate which directly follows from Proposition \ref{prop:L2Coercivity}.
\begin{proposition}\label{prop:L2CoercivityTransport}
We have that \[(\mathcal{L}_\Gamma^T(f)w,fw)_{L^2}\geq \frac{1}{5}|fw|_{L^2}^2-100|D_\theta f w|_{L^2}^2.\] 
\end{proposition}
\begin{proof}
The proof is a direct application of Proposition \ref{prop:L2Coercivity} and the Cauchy-Schwarz inequality on the last term. Indeed, 
\[(\mathcal{L}_\Gamma^T(f)w, fw)_{L^2}=(\mathcal{L}_\Gamma(f)w, fw)_{L^2}-(\mathbb{P}(\frac{3}{1+z}D_\theta f)w, fw)_{L^2}\]
\[\geq \frac{1}{4}|fw|_{L^2}^2-|\mathbb{P}(\frac{3}{1+z}D_\theta f)w|_{L^2}|fw|_{L^2}\geq \frac{1}{5}|fw|_{L^2}^2-5|\mathbb{P}(\frac{3}{1+z}D_\theta f)w|_{L^2}^2\geq \frac{1}{5}|fw|_{L^2}^2-100|D_\theta f w|_{L^2}^2.\]
\end{proof}

\begin{proposition}\label{prop:thetaderivative}
\[\Big(\big(D_\theta \mathcal{L}_\Gamma^T(f)\big),\big(D_\theta f\big) \frac{w^2}{\sin(2\theta)^\gamma}\Big)_{L^2}\geq (\frac{1}{4}-\alpha)\Big|\big(D_\theta f\big) \frac{w}{\sqrt{\sin(2\theta)^\gamma}}\Big|_{L^2}^2-10^7 \alpha|fw|_{L^2}^2.\]
\end{proposition}

\begin{remark}
The idea behind the proof is simple. First, $D_\theta$ commutes with the transport term and all of $\mathcal{L}_\Gamma$ except the term involving $\Gamma$. Because of \eqref{GammaAssumption2}, when $D_\theta$ hits $\Gamma$, this produces a term of size $\alpha$. Similarly, when $D_\theta$ hits the extra term in $\mathbb{P}$ we get the same factor of $\alpha$. Finally, the purpose of the weight $\frac{1}{\sin(2\theta)^\gamma}$ is to give a mostly-favorable term when the inner product is taken with the transport term. 
\end{remark}

\begin{proof}
We write: 
\[\mathcal{L}_\Gamma^T(f)=\mathcal{L}(f)-\frac{2\Gamma z}{c(1+z)^2}L_{12}(f)-\frac{3}{1+z}\sin(2\theta)\partial_\theta f+\frac{\Gamma(\theta)}{c}\frac{2 z^2}{(1+z)^3}L_{12}(\frac{3}{1+z}\sin(2\theta)\partial_\theta f)(0).\]
Thus, 
\[D_\theta\mathcal{L}_\Gamma^T(f)=\mathcal{L}(D_\theta f)+D_\theta\Gamma \Big(-\frac{2z}{c(1+z)^2}L_{12}(f)+\frac{1}{c}\frac{2 z^2}{(1+z)^3}L_{12}(\frac{3}{1+z}\sin(2\theta)\partial_\theta f)(0)\Big)-\frac{3}{1+z}\sin(2\theta)\partial_\theta D_\theta f.\]
Now, it is easy to check that 
\[
\int_0^{\pi/2}\frac{1}{\sin(2\theta)^\gamma}|D_\theta\Gamma(\theta)|^2d\theta\leq\int_{0}^{\pi/2} \frac{4\alpha^2\Gamma^2}{\sin(2\theta)^{\gamma}}\ \leq  4\alpha^2\int_{0}^{\pi/2} \frac{d\theta}{\sin(2\theta)^{1-\frac{\alpha}{2}}}\leq  4\alpha^2\int_{0}^{\pi/2} \frac{1}{(\frac{2}{\pi}\theta)^{1-\alpha/2}} \leq10\alpha.
\] 
Consequently, if we multiply $D_\theta\mathcal{L}_\Gamma^T(f)$ by $w^2\frac{1}{\sin(2\theta)^\gamma}D_\theta f$ and integrate, we get:
\[\Big(\big(D_\theta \mathcal{L}_\Gamma^T(f)\big),\big(D_\theta f\big) \frac{w^2}{\sin(2\theta)^\gamma}\Big)_{L^2}\geq \frac{1}{2}\Big|\big(D_\theta f\big) \frac{w}{\sqrt{\sin(2\theta)^\gamma}}\Big|_{L^2}^2-10^3\sqrt{\alpha}|fw|_{L^2}|(D_\theta f)\frac{w}{\sqrt{\sin(2\theta)^\gamma}}|_{L^2}\] \[+\frac{3}{2}\Big(\partial_\theta(\sin(2\theta)^{-\alpha/10}), \frac{w^2}{(1+z)}(D_\theta f)^2\Big).\]
The third term comes from integrating the transport term by parts. 
Thus, 
\[\Big(\big(D_\theta \mathcal{L}_\Gamma^T(f)\big),\big(D_\theta f\big) \frac{w^2}{\sin(2\theta)^\gamma}\Big)_{L^2}\geq (\frac{1}{4}-\alpha)\Big|\big(D_\theta f\big) \frac{w}{\sqrt{\sin(2\theta)^\gamma}}\Big|_{L^2}^2-2\times 10^4 \alpha|fw|_{L^2}^2.\]
\end{proof}
A similar calculation as above gives the following proposition:

\begin{proposition} \label{prop:L2CoercivityTransport2}
Let $\eta=\frac{99}{100}$. Then, \[(\mathcal{L}_\Gamma^T(f)w,fw\frac{1}{\sin(2\theta)^\eta})_{L^2}\geq \frac{1}{5}\Big|f \frac{w}{\sqrt{\sin(2\theta)^\eta}}\Big|_{L^2}^2-10^{5}|\frac{1}{z}L_{12}f|_{L^2}^2.\]
\end{proposition}

\begin{corollary}\label{cor:firsttransportcoercivity}
If $\alpha<10^{-14}$ and $\eta=\frac{99}{100}$ we have: \[10(\mathcal{L}_\Gamma^T(f)w,fw\frac{1}{\sin(2\theta)^\eta})_{L^2}+10^8(\mathcal{L}_\Gamma^T(f)w,fw)_{L^2}+10^{12}\Big(\big(D_\theta \mathcal{L}_\Gamma^T(f)\big),\big(D_\theta f\big) \frac{w^2}{\sin(2\theta)^\gamma}\Big)_{L^2}\]\[\geq \Big|f \frac{w}{\sqrt{\sin(2\theta)^\eta}}\Big|_{L^2}^2 +\Big|\big(D_\theta f\big) \frac{w}{\sqrt{\sin(2\theta)^\gamma}}\Big|_{L^2}^2+|fw|_{L^2}^2.\]
\end{corollary}
\begin{proof}
We combine the results of Propositions \ref{prop:L2CoercivityTransport}, \ref{prop:thetaderivative}, and \ref{prop:L2CoercivityTransport2}.
\end{proof}

\subsection{$L^2$ coercivity for $\mathcal{L}_\Gamma^T$ with one $z$-derivative}
With Corollary \ref{cor:firsttransportcoercivity} in hand, we now move to give higher order coercivity results on $\mathcal{L}_\Gamma^T.$ We introduce the weighted differential operator: \[D_z:=z\partial_z,\] set \begin{equation}\label{eta} \eta=\frac{99}{100}, \end{equation} and define the energies:
\[E_\theta^1:=\Big|\big(D_\theta f\big) \frac{w}{\sqrt{\sin(2\theta)^\gamma}}\Big|_{L^2}^2+|f\frac{w}{\sqrt{\sin(2\theta)^\eta}}|_{L^2}^2,\] 
\[E_{z,\theta}^1:=\Big|\big(D_z f\big)\frac{w}{\sqrt{\sin(2\theta)^\eta}}|_{L^2}^2+\Big|\big(D_\theta f\big) \frac{w}{\sqrt{{\sin(2\theta)^\gamma}}}\Big|_{L^2}^2+|f\frac{w}{\sqrt{\sin(2\theta)^\eta}}|_{L^2}^2.\] 

\begin{proposition} Under the assumptions of Corollary \ref{cor:firsttransportcoercivity}, we have
\[\Big(\big(D_z \mathcal{L}_\Gamma^T(f)\big),\big(D_z f\big) \frac{w^2}{\sin(2\theta)^\eta})_{L^2}\geq \frac{1}{4}\Big|\big(D_z f\big) \frac{w}{\sqrt{\sin(2\theta)^\eta}}\Big|_{L^2}^2-10^8E_\theta^1.\]
\end{proposition}

\begin{proof}

\[D_z\mathcal{L}_\Gamma^T(f)=\mathcal{L}(D_z(f))+\frac{2z}{(1+z)^2}f-D_z\Big(\frac{2\Gamma z}{c(1+z)^2}L_{12}(f)\Big)-3D_z\mathbb{P}\Big(\frac{1}{1+z}\sin(2\theta) \partial_\theta f\Big).\]
\[=\mathcal{L}(D_z(f))+\frac{2z}{(1+z)^2}f+\frac{2\Gamma z}{c(1+z)^2}\big(K,f(z,\theta)\big)_{L^2_\theta}-2\frac{\Gamma z(1-z)}{c(1+z)^3}L_{12}(f)+\frac{3z}{(1+z)^2}D_\theta f\]\[-\frac{3}{1+z}\sin(2\theta)\partial_\theta D_z f+3D_z(\frac{\Gamma}{c}\frac{2z^2}{(1+z)^3}L_{12}(\frac{1}{1+z}\sin(2\theta) \partial_\theta f)(0))\]
\[=\sum_{i=1}^7I_{i}.\]
Let us take a brief look at each term before proceeding. 
$I_{1}$ gives us coercivity once we include the weight $\frac{w^2}{\cos(\theta)^\eta}$. 
$I_{2}$ can be seen as lower order since it contains no derivative on $f$ and we already have an $L^2$ estimate from Corollary \ref{cor:firsttransportcoercivity}. $I_3$ and $I_4$ is also lower order in this sense but they contain $\frac{1}{1-\eta}$ as a factor. 
$I_5$ can also be seen as lower order since we have already controlled $D_\theta f$ in Corollary \ref{cor:firsttransportcoercivity}.
For $I_6$, after integrating by parts in $\theta$ we will get a positive term which we forget about and a negative term which contains $1-\eta$ as a coefficient just as we argued in the proof of Proposition \ref{prop:thetaderivative}. More precisely we have the following bounds which are not difficult to check:
\[(I_1, D_z(f)\frac{w^2}{\sin(2\theta)^\eta})_{L^2}\geq \frac{1}{2}\Big|\big(D_z f\big) \frac{w}{\sqrt{\sin(2\theta)^\eta}}\Big|_{L^2}^2. \]
\[\Big|I_{2}\frac{w}{\sqrt{\sin(2\theta)^\eta}}\Big|_{L^2}\leq |f\frac{w}{\sqrt{\cos(\theta)^\eta}}|_{L^2}\]
\[\Big|I_3\frac{w}{\sqrt{\sin(2\theta)^\eta}}\Big|_{L^2}\leq \frac{10}{\sqrt{1-\eta}}|fw|_{L^2}.\]
\[\Big|I_4\frac{w}{\sqrt{\sin(2\theta)^\eta}}\Big|_{L^2}\leq \frac{10}{\sqrt{1-\eta}}|\frac{1}{z}L_{12}(f)|_{L^2_z}.\]
\[\Big|I_5\frac{w}{\sqrt{\sin(2\theta)^\eta}}\Big|_{L^2}\leq 3|D_\theta f \frac{w}{\sqrt{\cos(\theta)^\eta}}|_{L^2}\]
\[\Big|(I_6, D_z(f)\frac{w^2}{\sin(2\theta)^\eta})_{L^2}\Big|\leq 3(1-\eta)|(D_zf)\frac{w}{\sqrt{\cos(\theta)^\eta}}|_{L^2}.\]
\[\Big|I_7\frac{w}{\sqrt{\sin(2\theta)^\eta}}\Big|_{L^2}\leq \frac{100}{\sqrt{1-\eta}}|fw|_{L^2}.\]

The result now follows using the Cauchy-Schwarz inequality. 
\end{proof}

\begin{definition}
We define the $\mathcal{H}^1$ inner product by \[(f,g)_{\mathcal{H}^1}=10\Big(\big(D_z f\big),\big(D_z g\big) \frac{w^2}{\sin(2\theta)^\eta}\Big)_{L^2}+10^{10}(fw,gw\frac{1}{\sin(2\theta)^\eta})_{L^2}\] \[+{10^{17}}(fw,gw)_{L^2}+10^{21}\Big(\big(D_\theta f\big),\big(D_\theta g\big) \frac{w^2}{\sin(2\theta)^\gamma}\Big)_{L^2},\] which induces a norm equivalent to 
\[|f|_{\mathcal{H}^1}^2=\sum_{k=0}^1|(D_z)^kf \frac{w}{\sqrt{\sin(2\theta)^\eta}}|_{L^2}^2+|(D_\theta) f\frac{w}{\sqrt{\sin(2\theta)^\gamma}}|_{L^2}^2.\]
\end{definition}

\begin{corollary}\label{FirstDerivativeCoercivity}
Let $\eta=\frac{99}{100}$ and $\alpha<\frac{1}{10^{14}}$. Then, 
\begin{equation}\label{eq:FirstDerivativeCoercivity} (\mathcal{L}_\Gamma^T(f),f)_{\mathcal{H}^1}\geq |f|_{\mathcal{H}^1}^2.\end{equation}
\end{corollary}
\begin{remark}
The reader should take note that $(f,f)_{\mathcal{H}^1}\not=|f|_{\mathcal{H}^1}^2$ but $|f|_{\mathcal{H}^1}^2\leq (f,f)_{\mathcal{H}^1}^2\leq 10^{21}|f|_{\mathcal{H}^2}^2$
\end{remark}

\subsection{Higher order derivatives and the inner product on $\mathcal{H}^k$}

In this section, we show how to inductively define an inner product on $\mathcal{H}^k$ from the inner product on $\mathcal{H}^1$ to show that $\mathcal{L}_\gamma^T$ is coercive on $\mathcal{H}^k$ for each $k\geq 2$. Toward this end, fix some $k\geq 2$ and assume that we have defined an inner product $(\cdot,\cdot)_{\mathcal{H}^{k-1}}$ with the following properties:
\begin{enumerate}
\item  $|f|_{\mathcal{H}^{k-1}}^2\leq (f,f)_{\mathcal{H}^{k-1}}\leq C_{k-1}|f|_{\mathcal{H}^{k-1}}^2$
\item $(\mathcal{L}^{T}_\Gamma f, f)_{\mathcal{H}^{k-1}}\geq c_{k-1}|f|_{\mathcal{H}^{k-1}}^2.$
\end{enumerate} 
Then, we will construct an inner product on $\mathcal{H}^k$ satisfying the above two properties with $k-1$ replaced by $k$. Observing that we have already established these in the case $k=2$ the main result of this section will follow by induction:

\begin{proposition}\label{prop:HkCoercivity}
Fix $\alpha<10^{-14}$  and $k\in\mathbb{N}$.Then, there exists $c_{k}>0$ so that for all $f\in \mathcal{H}^k$ we have:
\begin{equation} (\mathcal{L}_\Gamma^T(f), f)_{\mathcal{H}^k}\geq c_k|f|_{\mathcal{H}^2}^2.\end{equation}
\end{proposition}
\begin{remark}
The reader should take note that $(f,f)_{\mathcal{H}^k}\not=|f|_{\mathcal{H}^k}^2$ but $|f|_{\mathcal{H}^k}^2\leq (f,f)_{\mathcal{H}^k}^2\leq C_k|f|_{\mathcal{H}^k}^2$
\end{remark}

\begin{proof}
Suppose $k\geq 2$ and we have defined an inner product on $\mathcal{H}^{k-1}$ on which $\mathcal{L}_{\Gamma}^T$ is coercive as explained above. We will now define 
\[(f,g)_{\mathcal{H}^k}= (f,g)_{\mathcal{H}^{k-1}}+c_{1,k} (D_\theta f, D_\theta g)_{\mathcal{H}^{k-1}}+c_{2,k} (D_z f, D_z g)_{\mathcal{H}^{k-1}},\] where $c_{1,k}, c_{2,k}\geq 1$ will be chosen depending on $k$ only. First, let us observe that the first condition is satisfied automatically. To avoid unnecessary repetition let us introduce the notation $\approx$ here to mean $a\approx b$ if $ \bar c_k b\leq a\leq \bar C_k b$ for some positive universal constants\footnote{In the following, $\bar C_k$ and $\bar c_{k}$ are constants that depend on $k$ but which may change from line to line. 
} depending only on $k$. 

\[(f,f)_{\mathcal{H}^k}\approx |f|_{\mathcal{H}^{k-1}}^2+|D_z f|_{\mathcal{H}^{k-1}}^2+|D_\theta f|_{\mathcal{H}^{k-1}}^2\approx |f|_{\mathcal{H}^k}^2.\] Now let us define $c_{1,k}$ and $c_{2,k}$ in such a way that $\mathcal{L}_{\Gamma}^T$ will be coercive. Now let us compute:
\[(\mathcal{L}_{\Gamma}^T (f), f)_{\mathcal{H}^k}\geq c_{k-1}|f|_{\mathcal{H}^{k-1}}^2+c_{1,k}(D_\theta \mathcal{L}_{\Gamma}^T(f), D_\theta f)_{\mathcal{H}^{k-1}}+c_{2,k}(D_z \mathcal{L}_\Gamma^T(f), D_z f)_{\mathcal{H}^{k-1}}.\]
Now, as in the proof of Proposition \ref{prop:thetaderivative}, we note that 

\[D_\theta\mathcal{L}_\Gamma^T(f)=\mathcal{L}(D_\theta f)+D_\theta\Gamma \Big(-\frac{2z}{c(1+z)^2}L_{12}(f)+\frac{1}{c}\frac{2 z^2}{(1+z)^3}L_{12}(\frac{3}{1+z}\sin(2\theta)\partial_\theta f)(0)\Big)-\frac{3}{1+z}\sin(2\theta)\partial_\theta D_\theta f
\]
\[=\mathcal{L}_{\Gamma}^T(D_\theta f)+D_\theta\Gamma \Big(-\frac{2z}{c(1+z)^2}L_{12}(f)+\frac{1}{c}\frac{2 z^2}{(1+z)^3}L_{12}(\frac{3}{1+z}\sin(2\theta)\partial_\theta f)(0)\Big)-\frac{3}{1+z}\sin(2\theta)\partial_\theta D_\theta f\]
\[+2z\frac{\Gamma(\theta)}{c(1+z)^2}L_{12}(D_\theta f)\]
\[:=\mathcal{L}_{\Gamma}^T(D_\theta f)+E_1.\]
Observe that \[|E_1|_{\mathcal{H}^{k-1}}\leq \bar{C}_{k}|f|_{\mathcal{H}^{k-1}}.\] This is because $L_{12}$ is actually smoothing in $\theta$ so that $L_{12}(D_\theta f)$ and $L_{12}(f)$ actually have the same regularity (they can both be bounded in $\mathcal{H}^{k-1}$ by $f$ in the same space). 

Thus, 
\[(D_\theta\mathcal{L}_\Gamma^T(f), D_\theta f)_{\mathcal{H}^{k-1}}=(\mathcal{L}_\Gamma^T(D_\theta f), D_\theta f)_{\mathcal{H}^{k-1}}+(E_1, D_\theta f)_{\mathcal{H}^{k-1}}\geq c_{k-1}|D_\theta f|_{\mathcal{H}^{k-1}}^2-C_{k-1}|E_1|_{\mathcal{H}^{k-1}}|D_\theta f|_{\mathcal{H}^{k-1}}\]
Thus, 
\[(D_\theta\mathcal{L}_\Gamma^T(f), D_\theta f)_{\mathcal{H}^{k-1}}\geq c_{k-1}|D_\theta f|_{\mathcal{H}^{k-1}}^2 -\bar{C}_k |f|_{\mathcal{H}^{k-1}}|D_\theta f|_{\mathcal{H}^{k-1}}.\]
In particular, we now choose $c_{1,k}$ so that \[c_{1,k}\bar C_k=\frac{1}{2}c_{k-1}^2,\] we will have:
\[(\mathcal{L}_{\Gamma}^T (f), f)_{\mathcal{H}^k}\geq \frac{1}{2}c_{k-1}|f|_{\mathcal{H}^{k-1}}^2+\frac{1}{2}c_{1,k}c_{k-1}|D_\theta f|_{\mathcal{H}^{k-1}}^2+c_{2,k}(D_z \mathcal{L}_\Gamma^T(f), D_z f)_{\mathcal{H}^{k-1}}.\] Next, in an almost identical way to the above calculation, we observe that:
\[D_z(\mathcal{L}_\Gamma^T(f))=\mathcal{L}_\Gamma^T(D_zf) +E_2, \] where 
\[|E_2|_{\mathcal{H}^{k-1}}\leq \bar C_k (|f|_{\mathcal{H}^{k-1}}+|D_\theta f|_{\mathcal{H}^{k-1}}).\] Observe that the derivative term comes from the commutator with the angular transport term and $D_z$ and this is why we choose $c_{1,k}$ first as in the calculations in the preceding section. Thus, choosing $c_{2,k}$ so that: \[c_{2,k}\bar C_k=\frac{1}{8}c_{1,k}^2c_{k-1}^2,\] we get:
\[(\mathcal{L}_{\Gamma}^T (f), f)_{\mathcal{H}^k}\geq c_{k}|f|_{\mathcal{H}^k}^2\] for some constant $c_k>0$. This concludes the proof.

\end{proof}

\section{Elliptic Regularity Estimates}\label{EllipticEstimates}

The purpose of this section is to establish the necessary weighted $L^2$ and Sobolev estimates for the elliptic operator which relates the stream function and the vorticity. This is where the relationship between $\frac{u_r}{r}$ and $L_{12}$ as explained in Section \ref{Reductions} is made precise. The main technical results of this section are Propositions \ref{prop:L2} and \ref{prop:H2}. From these we establish Theorem \ref{RemovingL12} which is one of the pillars of this work. We should remark that Theorem \ref{RemovingL12} is related to the "Key Lemma" in \cite{KS}, the point being to isolate a main term in the Biot-Savart law for functions that are merely bounded. This connection and related issues are discussed in some detail in \cite{E1} and Section 6 of \cite{EJVP}. 

Consider the axi-symmetric Biot-Savart law: 
\[-\partial_{rr}\psi-\partial_{33}\psi-\frac{1}{r}\partial_r\psi+\frac{1}{r^2}\psi=f.\]
We begin by writing this in polar coordinates. We define \[\rho=\sqrt{r^2+x_3^2} \qquad \theta=\arctan(\frac{x_3}{r}).\]
Then we see:
\[-\partial_{\rho\rho}\psi-\frac{2}{\rho}\partial_\rho\psi-\frac{1}{\rho^2}\partial_{\theta\theta}\psi +\frac{\tan(\theta)}{\rho^2}\partial_\theta\psi +\frac{\sec^2(\theta)}{\rho^2}\psi=f\]
Next, let's write $f=F(\rho^\alpha,\theta)$ and (postulate that) $\psi=\rho^2\Psi(\rho^\alpha,\theta)$. 
It is convenient to introduce another variable \[R=\rho^\alpha.\]
Then we see:
\[-(2\Psi+\alpha(1+\alpha) R\partial_R \Psi+\alpha^2 R^2\partial_{RR}\Psi)-4\Psi-2\alpha R\partial_R\Psi -\partial_{\theta\theta}\Psi+\tan(\theta)\partial_\theta\Psi+\sec^2(\theta)\Psi=F.\]
One way to rewrite this is:
\begin{equation}\label{PolarBSL} L(\Psi)=-\alpha^2R^2\partial_{RR}\Psi-\alpha(5+\alpha)R\partial_R\Psi-\partial_{\theta\theta}\Psi+\partial_\theta\big(\tan(\theta)\Psi\big)-6\Psi=F.\end{equation}
We couple this equation with the natural boundary conditions on $\Psi$:
\[\Psi(R,0)=\Psi(R,\pi/2)=0,\qquad \lim_{R\rightarrow\infty} \Psi(R,\theta)=0.\]
\subsection{$L^2$ Estimates}
Notice that the first four terms of \eqref{PolarBSL} form a "positive" operator in the $L^2$ sense. The dangerous term is the $-6\Psi$ term--especially when $\alpha$ is very small. 
Despite this problem, we have the following proposition which is the backbone of this work. 
\begin{proposition}\label{prop:L2}
Let $F\in L^2$ be given and $0<\alpha\leq 1$. Assume that for every $R$ we have \[\int_{0}^{\pi/2}F(R,\theta)\cos^2(\theta)\sin(\theta)d\theta=0.\] Then, the unique $L^2$ solution to \eqref{PolarBSL} with Dirichlet boundary conditions on $[0,\infty)\times [0,\pi/2]$ satisfies:
\begin{equation}\label{L2Estimate} \Big|\partial_\theta\Big(\frac{\Psi}{\cos(\theta)}\Big)\Big|_{L^2}+|\partial_{\theta\theta}\Psi|_{L^2}+\alpha^2|R^2\partial_{RR}\Psi|_{L^2}\leq 100|F|_{L^2}. \end{equation}
\end{proposition}

\subsubsection*{The proof of Proposition \ref{prop:L2}}

\begin{proof}[Proof of Proposition \ref{prop:L2}]
We only establish the a-priori estimate as existence and uniqueness follows from the standard $L^p$ theory. 

\vspace{3mm}
\emph{Step 1: $\Psi$ is orthogonal to $\sin(\theta)\cos^2(\theta)$}

\vspace{3mm}

An important observation is that under the conditions of the lemma, $\Psi$ must also be orthogonal to $\sin(\theta)\cos^2(\theta)$. Indeed, define \[\Psi_\star(R):=\int_0^{\pi/2}\Psi(R,\theta)\sin(\theta)\cos^2(\theta)d\theta.\]
Then, we see:
\[\alpha^2R^2\partial_{RR}\Psi_\star+\alpha(5+\alpha)R\partial_R\Psi_\star=0.\] This is because $\sin(\theta)\cos^2(\theta)$ is in the kernel of the adjoint problem when $\alpha=0$. 
This ODE for $\Psi_\star$ (an Euler equation!) can be solved explicitly and its solutions are determined by solving 
\[\alpha^2 \lambda(\lambda-1)+\alpha(5+\alpha)\lambda=0\] which gives $\lambda_1=0$ and $\lambda_2=-\frac{5+\alpha}{\alpha}+1.$ Thus, \[\Psi_\star(R)=c_1+c_2 R^{1-\frac{5+\alpha}{\alpha}}\]
and the condition that $\Psi_\star\rightarrow 0$ as $R\rightarrow \infty$ and that $R^2\Psi$ vanishes at $0$ implies that $c_1=c_2=0$. Therefore,\[\Psi_\star\equiv 0.\]

\vspace{3mm}
\emph{Step 2: Energy estimates}
\vspace{3mm}

As usual, we multiply the equation by $\Psi$ and integrate by parts. For this part, we use the notation: \[|\cdot|=|\cdot|_{L^2}.\]
Multiplying \eqref{PolarBSL} by $\Psi$ and integrating we get \[\alpha^2|R\partial_R\Psi|^2-\alpha^2|\Psi|^2+\frac{\alpha(5+\alpha)}{2}|\Psi|^2+|\partial_\theta\Psi|^2-6|\Psi|^2+\frac{1}{2}|\sec(\theta)\Psi|^2=(F,\Psi).\] 
In particular, since $0<\alpha\leq 1$ we have:
\[|\partial_\theta\Psi|^2-6|\Psi|^2\leq |F|_{L^2}|\Psi|_{L^2}.\] 
Now let's expand the left hand side in a series (recalling the boundary conditions):\[\Psi(R,\theta)=\sum_{n\in\mathbb{N}} \Psi_n(R)\sin(2n\theta).\]
In particular, 
\[\sum_{n\geq 2} (4n^2-6)|\Psi_n(R)|_{L^2_R}^2\leq 2|\Psi_1(R)|_{L^2}^2+|F|_{L^2}|\Psi|_{L^2}.\]
But we also know that $\Psi_\star\equiv 0$. 
Thus, 
\[0=\frac{4}{\pi}\sum_{n}\Psi_n(R)\int_0^{\pi/2}\sin(\theta)\cos^2(\theta)\sin(2n\theta)d\theta=\frac{4}{\pi}\sum_{n}\Psi_{n}(R)(-1)^n\frac{4n}{(2n-3)^2(2n+1)^2}.\]
In particular
\[|\Psi_1|_{L^2}\leq \sum_{n\geq 2}|\Psi_n|_{L^2}\frac{9n}{(2n-3)^2(2n+1)^2}.\]
Thus, 
\[|\Psi_1|_{L^2}^2\leq 81\sum_{n\geq 2}|\Psi_n|_{L^2}^2\sum_{n\geq 2}\frac{n^2}{(2n-3)^4(2n+1)^4}<\sum_{n\geq 2}|\Psi_n|_{L^2}^2.\]
The last inequality is clear since \[\sum_{n\geq 2}\frac{n^2}{(2n-3)^4}<4\frac{\pi^2}{6}<7\] and $(2n+1)^4\geq 625$ if $n\geq 2$.
Thus, 
\[\sum_{n\geq 1} 4n^2|\Psi_n(R)|_{L^2}\leq 4|F||\Psi|_{L^2}.\]
In particular, \begin{equation}\label{L2Est} |\partial_\theta\Psi|_{L^2}\leq 4|F|_{L^2}\end{equation} At this point we are done and the rest is standard, but let us give more details. 
Now we come back to equation \eqref{PolarBSL} and multiply by $-\partial_{\theta\theta}\Psi$ and integrate. 
Integrating by parts in $R$ and $\theta$ we get:
\[\alpha^2| R\partial_{R\theta}\Psi|_{L^2}^2-\alpha^2|\partial_\theta\Psi|^2+\frac{\alpha(5+\alpha)}{2}|\partial_\theta\Psi|^2+|\partial_{\theta\theta}\Psi|^2-6|\partial_\theta\Psi|^2-\int\partial_\theta\Big(\tan(\theta)\Psi\Big)\partial_{\theta\theta}\Psi=\int F\partial_{\theta\theta}\Psi.\]
Since we have already controlled $|\partial_\theta\Psi|_{L^2}^2$, we only need to study the last term before the equality sign. Indeed, it has the same scaling as the term $|\partial_{\theta\theta}\Psi|^2$ and could destroy the energy estimates if handled foolishly. First let \[\tilde\Psi=\frac{\Psi}{\cos(\theta)}.\] By Lemma \ref{HardyInequality}, we have already established an $L^2$ a-priori estimate on $\tilde\Psi$. Moreover, \[-\int\partial_\theta\Big(\tan(\theta)\Psi\Big)\partial_{\theta\theta}\Psi=-\int \partial_\theta(\sin(\theta)\tilde\Psi)\partial_{\theta\theta}(\cos(\theta)\tilde\Psi)\]
\[=-\int(\cos(\theta)\tilde\Psi+\sin(\theta)\partial_\theta\tilde\Psi)(-\cos(\theta)\tilde\Psi-2\sin(\theta)\partial_\theta\tilde\Psi+\cos(\theta)\partial_{\theta\theta}\tilde\Psi).\]
\[=2\int\sin^2(\theta)(\partial_\theta\tilde\Psi)^2+\int\cos^2(\theta)(\partial_\theta\tilde\Psi)^2-\int\sin(\theta)\cos(\theta)\partial_\theta\tilde\Psi\partial_{\theta\theta}\tilde\Psi+G\]
\[=\frac{3}{2}\int(\partial_\theta\tilde\Psi)^2+G\] where $|G|\leq \frac{3}{2}|\tilde\Psi|_{L^2}^2\leq 15|\partial_\theta\Psi|_{L^2}^2$ using Lemma \ref{HardyInequality}.
Thus, 
\[|\partial_{\theta\theta}\Psi|_{L^2}^2+\frac{3}{2}|\partial_\theta\tilde\Psi|_{L^2}^2\leq 21|\partial_\theta\Psi|_{L^2}^2+|F|_{L^2}|\partial_{\theta\theta}\Psi|_{L^2}\]
Thus, 
\[|\partial_{\theta\theta}\Psi|_{L^2}^2+3|\partial_\theta\tilde\Psi|_{L^2}^2\leq (2(21)(16)+1)|F|_{L^2}^2.\] using \eqref{L2Est}
The estimate on $\alpha^2|R^2\partial_{RR}f|_{L^2}$ follows easily. 
\end{proof}

\subsection{The $\mathcal{H}^2$ Norm}

We now define the main norm which we will use which depends on two parameters: $\eta$ and $\gamma$. Recall first the weight $w=\frac{(1+z)^2}{z^2}$ and the derivatives $D_R=R\partial_R$ and $D_\theta=\sin(2\theta)\partial_\theta$. Then the $\mathcal{H}^2$ norm is defined as: \[|f|_{\mathcal{H}^2}^2=\sum_{k=0}^2\Big|D_R^k f \frac{w}{\sqrt{\sin^\eta(2\theta)}}\Big|_{L^2}^2+\Big|D_\theta f\frac{w}{\sqrt{\sin^{\gamma}(2\theta)}}\Big|_{L^2}^2+\Big|D_\theta D_R f\frac{w}{\sqrt{\sin^{\gamma}(2\theta)}}\Big|_{L^2}^2+\Big|D_\theta^2 f\frac{w}{\sqrt{\sin^{\gamma}(2\theta)}}\Big|_{L^2}^2\] As before, we take $\eta=\frac{99}{100}$ and $\gamma=1+\frac{\alpha}{10}.$ It is important to point out: \emph{when we prove elliptic estimates in $\mathcal{H}^2$ the constants will be independent of $\gamma$ and $\alpha$ and thus universal since we fix $\eta=\frac{99}{100}$.} 
Toward proving elliptic estimates in $\mathcal{H}^2$  we need a few Hardy-type inequalities.

\subsubsection*{Hardy Inequalities}

\begin{lemma}\label{HardyInequality}
Let $f\in H^1([0,\pi/2])$. Assume that $f(0)=f(\pi/2)=0$. Then, \[\int_{0}^{\pi/2}\frac{|f(\theta)|^2}{\sin^2(2\theta)}d\theta\leq 10\int_0^{\pi/2} |f'(\theta)|^2d\theta.\]
\end{lemma}
The proof of this lemma follows from the original Hardy inequality by noting that $\sin(2\theta)\geq 1-\frac{4}{\pi}|\theta-\frac{\pi}{4}|$ for $\theta\in [0,\pi/2]$, splitting the integral into two pieces, and making a change of variables. Later on we will also need the following two inequalities.
\begin{lemma}\label{HardyInequality2}
Let $f\in H^2([0,\pi/2])$. Assume that $f(0)=f(\pi/2)=0.$ Then,
\[\int_0^{\pi/2} \Big(\partial_\theta\big(\frac{f(\theta)}{\sin(2\theta)}\big)\Big)^2d\theta\leq 10\int_{0}^{\pi/2}|f''(\theta)|^2d\theta. \] 
\end{lemma}
\begin{proof}
For simplicity, we give a proof of a simpler version and leave the stated result to the reader:
\[\int_{0}^\infty \Big(\partial_x\big(\frac{f}{x}\big)\Big)^2\leq \frac{1}{2}\int_0^{\infty} (f''(x))^2.\]
The proof of this is as follows:
Set $g=\frac{f}{x}$. Then, $f''=(gx)''=xg''+2g'.$
\[\int (f'')^2=\int x^2 (g'')^2+4(g')^2+4xg''g'=\int x^2 (g'')^2+2(g')^2.\]
\end{proof}
We also need the following \emph{sharp} version of Lemma \ref{HardyInequality}
\begin{lemma}\label{SharpHardyInequality1}
Let $f\in H^1([0,\pi])$ and $0\leq \eta\leq 1$. Assume that $f(0)=f(\pi)=0$. Then, \[\int_{0}^{\pi}\frac{|f(\theta)|^2}{\sin^{2+\eta}(\theta)}d\theta\leq \frac{4}{(\eta+1)^2}\int_0^{\pi} \frac{|f'(\theta)|^2}{\sin^\eta(\theta)}d\theta+100|f|_{H^1}^2.\]
\end{lemma}
\begin{remark}
The lemma is sharp in terms of the size of the first constant; the size of the second constant is irrelevant for our purposes. 
\end{remark}
\begin{proof}
Observe that for $\theta\in [0,\pi/2]$ we have: \[\Big|\frac{1}{\sin^{2+\eta}(\theta)}-\frac{1}{\theta^{2+\eta}}\Big|\leq\frac{|\theta^{2+\eta}-\sin^{2+\eta}(\theta)|}{\sin^{4+2\eta}(\theta)}\leq \frac{C}{\sin^{\eta}(\theta)}.\] Thus, \[\Big|\int_0^{\pi/2}\frac{f(\theta)^2}{\sin^{2+\eta}(\theta)}d\theta-\int_{0}^{\pi/2}\frac{f(\theta)^2}{\theta^{2+\eta}}d\theta\Big|\leq 2(1+\eta)\int_0^{\pi}\frac{|f(\theta)|^2}{\sin^\eta(\theta)}d\theta.\]
Now, \[\int_{0}^{\pi/2} \frac{f(\theta)^2}{\theta^{2+\eta}}\leq \int_{0}^{\pi} \frac{f(\theta)^2}{\theta^{2+\eta}} =-\frac{2}{1+\eta}\int_{0}^\pi \frac{f(\theta)f'(\theta)}{\theta^{1+\eta}}d\theta\leq-\frac{2}{1+\eta}\int_{0}^{\pi/2} \frac{f(\theta)f'(\theta)}{\theta^{1+\eta}}d\theta+\frac{2^{2+\eta}}{(1+\eta)}\int_{\pi/2}^\pi f(\theta)f'(\theta)d\theta.\] Thus, using the Cauchy-Schwarz inequality we see:
\[\int_{0}^{\pi/2} \frac{f(\theta)^2}{\theta^{2+\eta}}d\theta\leq \frac{4}{(1+\eta)^2}\int_{0}^{\pi/2}\frac{f'(\theta)^2}{\theta^\eta}d\theta+\frac{2^{3+\eta}}{(1+\eta)}|f|_{H^1}^2.\]
Now note that\[\Big|\frac{1}{\theta^\eta}-\frac{1}{\sin^\eta(\theta)}\Big|\leq 10\] for $\theta\in [0,\pi/2]$ since $0\leq \eta\leq 1$. Now we do a similar calculation on the interval $[\pi/2,\pi]$ and we are done. 
\end{proof}
We now have the following corollary which follows from Lemma \ref{SharpHardyInequality1} by scaling.
\begin{corollary}\label{SharpHardyInequality2}
Let $f\in H^1([0,\pi/2])$ and $0\leq \eta\leq 1$. Assume that $f(0)=f(\pi/2)=0$. Then, \[\int_{0}^{\pi/2}\frac{|f(\theta)|^2}{\sin^{2+\eta}(2\theta)}d\theta\leq \frac{1}{(\eta+1)^2}\int_0^{\pi/2} \frac{|f'(\theta)|^2}{\sin^\eta(2\theta)}d\theta+100|f|_{H^1}^2.\]
\end{corollary}

\subsection{$\mathcal{H}^2$ Estimates}\label{H2Elliptic}

We now move to establish the $\mathcal{H}^2$ estimates for solutions to \eqref{PolarBSL} which is the heart of this section. 
\begin{proposition}\label{prop:H2}
Under the same assumptions as Proposition \ref{prop:L2} along with the assumptions that $0\leq \alpha\leq \frac{1}{4},$ $1<\gamma\leq \frac{5}{4}$ and $|F|_{\mathcal{H}^2}<\infty,$ we have:
\[\alpha^2 |R^2\partial_{RR}\Psi|_{\mathcal{H}^2}+|\partial_{\theta\theta}\Psi|_{\mathcal{H}^2}\leq C|F|_{\mathcal{H}^2}\] for some universal constant $C>0$ independent of $\alpha$ and $\gamma$. 
\end{proposition}
\begin{remark}
The assumption that $\alpha\leq \frac{1}{4}$ is probably technical, but we are only going to use this when $\alpha$ is very small. 
\end{remark}
\begin{proof}
Recall the $\mathcal{H}$ norm defined in \eqref{HNorm}.

\vspace{3mm}
\emph{Step 1: Only radial weights}
\vspace{3mm}

The goal of this step is to establish a weighted version of Proposition \ref{prop:L2}.
We start by multiplying \eqref{PolarBSL} by $\Psi w^2$ and integrating (note that we are only putting a weight in $R$ to begin with). We see: 

\[\alpha^2|\partial_R\Psi Rw|^2-\frac{\alpha^2}{2}(\Psi^2,\partial_{R}^2(R^2w^2))_{L^2}+\frac{\alpha(5+\alpha)}{2}(\Psi^2,\partial_R(Rw^2))_{L^2}+|\partial_{\theta\theta}\Psi w|^2-6|\Psi|^2+\frac{1}{2}|\sec^2(\theta)\Psi w|^2=(F,\Psi w^2).\]
Moreover, it can be checked directly that \[|\partial_{R}^2(R^2 w^2)|\leq 6 w^2,\qquad |\partial_R(Rw^2)|\leq 3w^2.\]
Thus, 
\[\alpha^2|\partial_R\Psi Rw|^2+|\partial_{\theta\theta}\Psi w|^2-6|\Psi w|^2+\frac{1}{2}|\sec^2(\theta)\Psi w|^2\leq (F,\Psi w^2)+(3\alpha^2+\frac{3\alpha(5+\alpha)}{2})|w\Psi|^2.\]
In particular, since $0\leq\alpha\leq \frac{1}{4}$ we see:
\[|\partial_{\theta\theta} \Psi w|^2-9.5|\Psi w|^2\leq |(Fw,\Psi w)_{L^2}|\]
Now we argue as in Step 2 of the proof of Proposition \ref{prop:L2}. Using that $w\Psi_\star\equiv 0,$ we see that \[\sum_{n\geq 2} (4n^2-9.5)|\Psi_n w|^2_{L^2}\leq 5.5|\Psi_1w|_{L^2}^2+|Fw|_{L^2}|\Psi w|_{L^2}.\]
Thus, since \[|\Psi_1w|_{L^2}^2<\sum_{n\geq 2}|\Psi_n w|_{L^2}^2\] we get
\[\sum_{n\geq 2} (4n^2-15)|\Psi_n w|^2_{L^2}\leq |Fw|_{L^2}|\Psi w|_{L^2}.\]
Thus, 
\[\sum_{n\geq 1} 4n^2|\Psi_n w|^2_{L^2}\leq 20|Fw|_{L^2}|\Psi w|_{L^2}.\]
Now the proof follows the same as before to give:
\begin{equation}\label{radialweightonlyL2} \alpha^2 |R^2\partial_{RR}\Psi w|_{L^2}+|\partial_{\theta\theta}\Psi w|_{L^2}\leq C_1|F w|_{L^2}.\end{equation}

\vspace{3mm}

\emph{Step 2: Radial and (weak) angular weights}

\vspace{3mm}

In this step we will prove \begin{equation}\label{radialandangularweightL2} \alpha^2 |R^2\partial_{RR}\Psi \frac{w}{\sin(2\theta)^{\eta/2}}|_{L^2}+|\partial_{\theta\theta}\Psi \frac{w}{\sin(2\theta)^{\eta/2}}|_{L^2}\leq C_1|F \frac{w}{\sin(2\theta)^{\eta/2}}|_{L^2}.\end{equation}

As in Step 1, we multiply \eqref{PolarBSL} by $-\partial_{\theta\theta}\Psi\frac{w^2}{\sin(2\theta)^\eta}$ and integrate. We get:
\[\sum_{i=1}^5I_i=-\Big(F,\partial_{\theta\theta}\Psi\frac{w^2}{\sin(2\theta)^\eta}\Big),\] where
\[I_1=\alpha^2\Big(R^2\partial_{RR}\Psi, \partial_{\theta\theta}\Psi\frac{w^2}{\sin(2\theta)^\eta}\Big)\qquad I_2=\Big(\alpha(5+\alpha)R\partial_R\Psi, \partial_{\theta\theta}\Psi\frac{w^2}{\sin(2\theta)^\eta}\Big), I_3=\Big(\partial_{\theta\theta}\Psi, \partial_{\theta\theta}\Psi\frac{w^2}{\sin(2\theta)^\eta}\Big)\] \[\qquad I_4=-\Big(\partial_\theta\big(\tan(\theta)\Psi\big), \partial_{\theta\theta}\Psi\frac{w^2}{\sin(2\theta)^\eta}\Big),\qquad I_5=6\Big(\Psi,\partial_{\theta\theta}\Psi\frac{w^2}{\sin(2\theta)^\eta}\Big).\]
Note that $I_3$ is a positive term which we will not touch. After integrating by parts in the right way, $I_1 $ and $I_4$ contain positive terms and some terms which we control by the information we gained from Step 1. For $I_2$ and $I_5$ we just estimate estimate them directly using the Cauchy-Schwarz inequality.  
For $I_5,$ that \[\Big|\frac{\Psi}{\sin(2\theta)^{\eta/2}}w\Big|_{L^2}\leq |\partial_\theta\Psi w|_{L^2}\leq |Fw|_{L^2}\] using the Hardy inequality and \eqref{radialweightonlyL2} from Step 1. 
Similarly, for $I_2$, observe that \[\alpha \Big|R\partial_R\Psi\frac{w}{\sin(2\theta)^{\eta/2}}\Big|_{L^2}\leq C\alpha |R\partial_{R\theta}\Psi w|_{L^2}\leq C|Fw|_{L^2}\] again using the Hardy inequality and \eqref{radialweightonlyL2}.
We now move to $I_1$. {\bf In what follows we will denote by $E$ an error which changes from line to line but can be controlled in a similar way to how $I_5$ and $I_2$ were just estimated. }
\[I_1=\alpha^2\Big(R^2\partial_{RR}\Psi, \partial_{\theta\theta}\Psi\frac{w^2}{\sin(2\theta)^\eta}\Big)=-\alpha^2\Big(R^2\partial_R\Psi,\partial_{R\theta\theta}\Psi\frac{w^2}{\sin(2\theta)^\eta}\Big)+E\]
\[=\alpha^2\Big(R^2\partial_{R\theta}\Psi,\partial_{R\theta}\Psi\frac{w^2}{\sin(2\theta)^\eta}\Big)+\alpha^2\Big(R^2\partial_{R}\Psi,\partial_{R\theta}\Psi \partial_{\theta}\frac{w^2}{\sin(2\theta)^{\eta}}\Big)+E\]
Now, we have to be very careful in how we handle \[\alpha^2\Big(R^2\partial_{R}\Psi,\partial_{R\theta}\Psi \partial_{\theta}\frac{w^2}{\sin(2\theta)^{\eta}}\Big)=-\frac{\alpha^2}{2}(R^2\partial_R\Psi,\partial_R\Psi\partial_{\theta\theta}\frac{w^2}{\sin^{\eta}(2\theta)}\Big).\] Observe that \[\partial_{\theta\theta}\frac{1}{\sin(2\theta)^\eta}=4\eta(\eta+1)\frac{\cos^2(2\theta)}{\sin^{2+\eta}(2\theta)}+\frac{4\eta}{\sin^{\eta}(2\theta)}.\]
Thus, 
\[I_1=\alpha^2\Big(R^2\partial_{R\theta}\Psi,\partial_{R\theta}\Psi\frac{w^2}{\sin(2\theta)^\eta}\Big)-\alpha^2(2\eta(\eta+1))\Big(R^2\partial_R\Psi,\partial_R\Psi \frac{w^2}{\sin^{2+\eta}(2\theta)}\Big)+E.\]
Now, by the sharp Hardy inequality \eqref{SharpHardyInequality2}, we have:
\[I_1\geq \alpha^2\Big[\Big(R^2\partial_{R\theta}\Psi,\partial_{R\theta}\Psi\frac{w^2}{\sin(2\theta)^\eta}\Big)-\frac{ 2\eta}{\eta+1}\Big(R^2\partial_{R\theta}\Psi,\partial_{R\theta}\Psi\frac{w^2}{\sin(2\theta)^\eta}\Big)\Big]-|Fw|_{L^2}^2\]
\[=\frac{1-\eta}{1+\eta}\alpha^2|R\partial_{R\theta}\Psi \frac{w}{\sin(2\theta)^{\eta/2}}|_{L^2}^2-C|Fw|_{L^2}\sqrt{I_3}.\]

We now turn to $I_4$. 
To estimate \[I_4=-\Big(\partial_\theta\big(\tan(\theta)\Psi\big), \partial_{\theta\theta}\Psi\frac{w^2}{\sin(2\theta)^\eta}\Big),\]
we again introduce \[\bar\Psi=\frac{\Psi}{\cos(\theta)}.\]
As before, we denote by $E$ an error term which is easily controlled. Then, 
\[I_4=-(\sin(\theta)\partial_\theta\bar \Psi \partial_{\theta\theta}(\cos(\theta)\bar\Psi),\frac{w^2}{\sin^\eta(2\theta)}\Big)+E=I_{4,1}+I_{4,2}+I_{4,3}+E,\] where \[I_{4,1}=-\frac{1}{2}\Big(\sin(2\theta)\partial_\theta\bar\Psi\partial_{\theta\theta}\bar\Psi, \frac{w^2}{\sin^{\eta}(2\theta)}\Big)\qquad I_{4,2}=2\Big(\sin^2(\theta)(\partial_\theta\bar\Psi)^2,\frac{w^2}{\sin^\eta(2\theta)}\Big)\qquad I_{4,3}=-\frac{1}{2}\Big(\sin(2\theta)\partial_\theta\bar\Psi\bar\Psi,\frac{w^2}{\sin^\eta(2\theta)}\Big).\]
Integrating by parts and using \eqref{radialweightonlyL2}, it is easy to see that \[|I_{4,3}|\leq C|Fw|_{L^2}.\]
$I_{4,2}$ is a positive term which we will use and $I_{4,1}$ can be re-written as:
\[I_{4,1}=\frac{1-\eta}{2}\Big(\cos(2\theta) (\partial_{\theta}\bar\Psi)^2, \frac{w^2}{\sin^\eta(2\theta)}\Big).\]. Thus, since $\cos(2\theta)=\cos^2(\theta)-\sin^2(\theta)$ we see:
\[I_{4}\geq \Big(\sin^2(\theta)(\partial_\theta\bar\Psi)^2,\frac{w^2}{\sin^\eta(2\theta)}\Big)-C|Fw|_{L^2}^2. \]
Summing up the estimates of $I_{i}$ for $1\leq i\leq 5$, we get:
\[|\partial_{\theta\theta}\Psi \frac{w}{\sin(2\theta)^{\eta/2}}|_{L^2}\leq C_1|F \frac{w}{\sin(2\theta)^{\eta/2}}|_{L^2}.\] From here, it is not difficult to get \eqref{radialandangularweightL2}.
\vspace{3mm}

\emph{Step 3: Radial and (weak) angular weights with radial derivatives}

\vspace{3mm}

We note that we can rewrite \eqref{PolarBSL} in the following convenient form:
\[L(\Psi)=-\alpha^2(R\partial_R)^2\Psi-5\alpha R\partial_R\Psi-\partial_{\theta\theta}\Psi+\partial_\theta(\tan(\theta)\Psi)-6\Psi=F.\] Recall the notation $D_R=R\partial_R$. 
Consequently, 
\[L(D_R^k\Psi)=D_R^k F\] for $k=0,1,2$. 
Thus, using Steps 1 and 2 (in particular, \eqref{radialandangularweightL2}) we have:
\begin{equation}\label{radialweightonlyH2} \alpha^2 |R^2\partial_{RR}(D_R)^{k}\Psi \frac{w}{\sin(2\theta)^{\eta/2}}|_{L^2}+|\partial_{\theta\theta}(D_R)^k\Psi \frac{w}{\sin(2\theta)^{\eta/2}}|_{L^2}\leq C_1|(D_R)^kF \frac{w}{\sin(2\theta)^{\eta/2}}|_{L^2}\end{equation}
for $k=0,1,2$.

\vspace{3mm}

\emph{Step 4: Radial and angular weights with one angular derivative.}

\vspace{3mm}

Now notice that from Step 3 have shown that $|D_R^2\partial_{\theta\theta}\Psi w|_{L^2}\leq |D_R^2F w|_{L^2}.$
Consequently, we can write:
\[-\partial_{\theta\theta}\Psi+\partial_{\theta}(\tan(\theta)\Psi)=F+6\Psi+\alpha^2D_R^2\Psi-3\alpha D_R\Psi:=F_1.\]
Now let's apply $\partial_{\theta}$ to this equation and multiply both sides by $-\sin(2\theta)^{2-\gamma}\partial_{\theta}^3\Psi w^2$. We get:
\[\int |\partial_\theta^3\Psi|^2\sin(2\theta)^{2-\gamma}w^2-\int\partial_{\theta\theta}(\tan(\theta)\Psi)\partial_{\theta}^3\Psi\sin(2\theta)^{2-\gamma}w^2=-\Big(\partial_\theta F_1\sin(2\theta)^{\frac{2-\gamma}{2}}w,\partial_{\theta}^3\Psi \sin(2\theta)^{\frac{2-\gamma}{2}}w\Big)_{L^2}.\]
By assumption as well as \eqref{radialweightonlyH2} we have that \[|\partial_{\theta}F_1\sin(2\theta)^{2-\gamma}w|\leq C|F|_{\mathcal{H}^2}.\]
Thus, our concern is to deal with the term:
\[I:=-\int\partial_{\theta\theta}(\tan(\theta)\Psi)\partial_{\theta}^3\Psi\sin(2\theta)^{2-\gamma}w^2.\]
As in Step 1, we define: \[\bar\Psi=\frac{\Psi}{\cos(\theta)}.\]
Then,
\[I=-\int\Big(\sin(\theta)\partial_{\theta}^2\bar\Psi+2\cos(\theta)\partial_\theta\bar\Psi-\sin(\theta)\bar\Psi\Big)\partial_\theta^3\Psi\sin(2\theta)^{2-\gamma}w^2\]
\[=-\int \sin(\theta)\partial_{\theta}^2\bar\Psi \Big(\cos(\theta)\partial_\theta^3\bar\Psi-3\sin(\theta)\partial_\theta^2\bar\Psi-3\cos(\theta)\partial_\theta\bar\Psi+\sin(\theta)\bar\Psi\Big)\sin(2\theta)^{2-\gamma}w^2\]
\[-\int \Big(2\cos(\theta)\partial_\theta\bar\Psi-\sin(\theta)\bar\Psi\Big)\partial_\theta^3\Psi \sin(2\theta)^{2-\gamma}w^2\]
\[=I_1+I_2.\]
First we estimate $I_2$. 
Note that $\cos(\theta)\partial_\theta\bar\Psi=\partial_\theta(\cos(\theta)\bar\Psi)+\sin(\theta)\bar\Psi.$
Thus, 
\[I_2=-\int\Big(2\partial_\theta\Psi+\sin(\theta)\bar\Psi)\partial_\theta^3\Psi\sin(2\theta)^{2-\gamma}w^2.\]
Thus, \[|I_2|\leq C|\partial_\theta^3\Psi\sin(2\theta)^{\frac{2-\gamma}{2}}w||Fw|_{L^2}\] using \eqref{radialweightonlyL2} and the Hardy inequality. 
Now we turn to $I_1$. 
\[I_1=-\int \sin(\theta)\partial_{\theta}^2\bar\Psi \Big(\cos(\theta)\partial_\theta^3\bar\Psi-3\sin(\theta)\partial_\theta^2\bar\Psi-3\cos(\theta)\partial_\theta\bar\Psi+\sin(\theta)\bar\Psi\Big)\sin(2\theta)^{2-\gamma}w^2\]
\[=\int \Big[-\frac{1}{4}\sin(2\theta)^{3-\gamma}\partial_{\theta}\Big((\partial_{\theta}^2\bar\Psi)^2\Big)+3\sin^2(\theta)\Big(\partial_{\theta}^2\bar\Psi\Big)^2\sin(2\theta)^{2-\gamma}+\frac{3}{4}\sin(2\theta)^{3-\gamma}\partial_{\theta}\Big((\partial_{\theta}\bar\Psi)^2\Big)-\sin^2(\theta)\sin(2\theta)^{2-\gamma}\bar\Psi\partial_{\theta}^2\bar\Psi\Big]w^2\]
\[=\int \Big[(\frac{3-\gamma}{2}\cos(2\theta)+3\sin^2(\theta))\Big(\partial_{\theta}^2\bar\Psi\Big)^2\sin(2\theta)^{2-\gamma}-\frac{3(3-\gamma)}{2}\cos(2\theta)\sin(2\theta)^{2-\gamma}(\partial_\theta\bar\Psi)^2-\sin^2(\theta)\sin(2\theta)^{2-\gamma}\bar\Psi\partial_{\theta}^2\bar\Psi\Big]w^2\]
\[\geq \frac{1}{2}\int \Big(\partial_\theta^2\bar\Psi\Big)^2\sin(2\theta)^{2-\gamma}-C|Fw|^2_{L^2}\] using the Hardy inequality and \eqref{radialweightonlyL2} since $1\leq\gamma\leq\frac{3}{2}$. 
Thus we conclude that:
\begin{equation}\label{angularweightH1}\Big|\partial_\theta^3\Psi\sin(2\theta)^{\frac{2-\gamma}{2}}w\Big|_{L^2}\leq C|F|_{\mathcal{H}^2}.\end{equation}
Next we want to estimate two derivatives in $\theta$. 

\vspace{3mm}

\emph{Step 5: Radial and angular weights with two angular derivatives.}

\vspace{3mm}

We now come to the last step of the proof which handles the case of two angular derivatives. As in the previous step, where we only took one angular derivative, we just take the equation:
\[-\partial_{\theta\theta}\Psi+\partial_{\theta}(\tan(\theta)\Psi)=F+6\Psi+\alpha^2D_R^2\Psi-5\alpha D_R\Psi:=F_1,\]
apply $\partial_{\theta\theta}$, multiply by $-\sin(2\theta)^{4-\gamma}\partial_{\theta}^4\Psi w^2$, and integrate. 
We get:
\[\int |\partial_\theta^4\Psi|^2\sin(2\theta)^{4-\gamma}w^2=-\Big(\partial_\theta^2 F_1\sin(2\theta)^{\frac{4-\gamma}{2}}w,\partial_{\theta}^4\Psi \sin(2\theta)^{\frac{4-\gamma}{2}}w\Big)_{L^2}+\int\partial_{\theta}^3(\tan(\theta)\Psi)\partial_{\theta}^4\Psi\sin(2\theta)^{4-\gamma}w^2.\]
As before, \[\Big|\Big(\partial_\theta^2 F_1\sin(2\theta)^{\frac{4-\gamma}{2}}w,\partial_{\theta}^4\Psi \sin(2\theta)^{\frac{4-\gamma}{2}}w\Big)_{L^2}\Big|\leq |F|_{\mathcal{H}^2}|\partial_{\theta}^4\Psi \sin(2\theta)^{\frac{4-\gamma}{2}}w|_{L^2}.\]
Thus we are left to study:
\[I:=\int\partial_{\theta}^3(\tan(\theta)\Psi)\partial_{\theta}^4\Psi\sin(2\theta)^{4-\gamma}w^2.\] We will show that this quantity consists of negative terms and terms which we have already controlled. 
Again we introduce the function $\frac{\Psi}{\cos(\theta)}=\bar\Psi$. 
Then \[I=\int\partial_\theta^3(\sin(\theta)\bar\Psi)\partial_\theta^4(\cos(\theta)\bar\Psi)\sin(2\theta)^{4-\gamma}w^2=\sum_{i=1}^{14}\int I_{i}\sin(2\theta)^{4-\gamma}w^2,\] with: 
\[I_{1}=-\cos^2(\theta) \bar\Psi^2 \qquad I_2=-\frac{7}{2}\sin(2\theta)\bar\Psi\partial_\theta\bar\Psi \qquad I_3=9\cos^2(\theta)\bar\Psi \partial_{\theta\theta}\bar\Psi  \qquad I_4=\frac{5}{2}\sin(2\theta)\bar\Psi \partial_\theta^3\bar\Psi,\] 
\[I_{5}=-\cos^2(\theta) \bar\Psi \partial_\theta^{4}\bar\Psi \qquad I_6=-12\sin^2(\theta)(\partial_\theta\bar\Psi)^2 \qquad I_7=15\sin(2\theta)\partial_\theta\bar\Psi \partial_{\theta}^2\bar\Psi  \qquad I_8=16\sin^2(\theta)\partial_\theta\bar\Psi \partial_\theta^3\bar\Psi\] 
\[I_{9}=-\frac{3}{2}\partial_\theta \bar\Psi \partial_\theta^4\bar\Psi  \qquad I_{10}=-18\cos^2(\theta)(\partial_\theta^2\bar\Psi)^2 \qquad I_{11}=-9\sin(2\theta)\partial_\theta^2\bar\Psi \partial_\theta^3\bar\Psi  \qquad I_{12}=3\cos^2(\theta)\partial_\theta^2\bar\Psi \partial_\theta^4\bar\Psi \] 
\[I_{13}=-4\sin^2(\theta)(\partial_\theta^3\bar\Psi)^2 \qquad I_{14}=\frac{1}{2}\sin(2\theta)\partial_\theta^3\bar\Psi \partial_\theta^4\bar\Psi.\]
Note that $\int\partial_\theta^2\tilde\Psi\sin^{2-\gamma}(2\theta)w^2$ has already been controlled (see the inequality right before \eqref{angularweightH1}). Similarly, all lower order terms have been controlled. In particular, $I_{i}$ for $1\leq i\leq 11$ can be controlled by $C|F|_{\mathcal{H}^2}^2$ as before using integration by parts in some terms (though there are many good terms as well). Thus we will only need to consider $I_{12},I_{13},$ and $I_{14}$. 
\[I \leq C|F|_{\mathcal{H}^2}^2 -3\int \cos^2(\theta)(\partial_\theta^3\bar\Psi)^2\sin^{4-\gamma}(2\theta)w^2-4\int \sin^2(\theta)(\partial_\theta^3\bar\Psi)^2\sin^{4-\gamma} (2\theta)w^2 +\frac{5-\gamma}{2}\int \cos(2\theta) (\partial_\theta^3\bar\Psi)^2\sin^{4-\gamma}(2\theta)w^2\]
\[\leq C|F|_{\mathcal{H}^2}^2-\int  (\partial_\theta^3\bar\Psi)^2\sin^{4-\gamma}(2\theta)w^2\]
This concludes the proof. 

\end{proof}

\subsection{General $\mathcal{H}^k$ Case}

In this subsection we extend Proposition \ref{prop:H2} which established elliptic estimates on $\mathcal{H}^2$ to all $\mathcal{H}^k$ spaces for $k\geq 3$. 
\begin{proposition}\label{prop:Hk}
Fix $k\geq 2$. Under the same assumptions as Proposition \ref{prop:H2} if $|F|_{\mathcal{H}^k}<\infty,$ we have:
\[\alpha^2 |R^2\partial_{RR}\Psi|_{\mathcal{H}^k}+|\partial_{\theta\theta}\Psi|_{\mathcal{H}^k}\leq C_k|F|_{\mathcal{H}^k}\] for some constant $C_k>0$ depending only on $k$ and independent of $\alpha$ and $\gamma$. 
\end{proposition}

The proof is quite similar to the proof in the $\mathcal{H}^2$ case so we only give a sketch. 
\begin{proof}

As with the coercivity estimates, since we already have an $\mathcal{H}^2$ we can proceed by induction on $k$. Let us first rewrite \eqref{PolarBSL} as:
\[-\partial_{\theta\theta}\Psi+\partial_\theta(\tan(\theta)\Psi)=F+6\Psi+\alpha^2 D_R^2 \Psi+\alpha D_R\Psi:=G.\]
Since estimates on the radial derivatives are relatively simple to get (because the equation commutes with $D_R$ derivatives), it suffices to establish $\mathcal{H}^k$ estimates on just the angular part of the equation: 
\[-\partial_{\theta\theta}\Psi+\partial_\theta(\tan(\theta)\Psi)=G,\] for $|G|_{\mathcal{H}^k}\leq C_k |F|_{\mathcal{H}^k}$. 
Toward this end, we wish to show that the following quantity is non-positive up to lower order terms \[(\partial_\theta^{k+1}(\tan(\theta)\Psi), \partial_\theta^{k+2} \Psi \sin(2\theta)^{2k-\gamma})_{L^2_\theta}.\]
It is natural to consider $\tilde\Psi= \frac{\Psi}{\cos(\theta)}$ so that we wish to study:
\[(\partial_\theta^{k+1}(\sin(\theta)\tilde\Psi), \partial_\theta^{k+2} (\cos(\theta)\tilde\Psi) \sin(2\theta)^{2k-\gamma})_{L^2_\theta}.\]
By induction on $k$, it suffices to consider only the following three terms: 
\[\frac{1}{2}\int \partial_\theta^{k+1}\tilde\Psi\partial_\theta^{k+2}\tilde\Psi \sin(2\theta)^{2k+1-\gamma}-(k+2)\int\sin^2(\theta)\big(\partial_\theta^{k+1}\tilde\Psi\big)^2\sin(2\theta)^{2k-\gamma}\] \[+(k+1)\int \cos^2(\theta)\partial_\theta^k\tilde\Psi \partial_\theta^{k+2}\tilde\Psi \sin(2\theta)^{2k-\gamma}\]
\[=\frac{2k+1-\gamma}{4}\int (\partial_\theta^{k+1}\tilde\Psi)^2 \cos(2\theta) \sin(2\theta)^{2k-\gamma}-(k+2)\int\sin^2(\theta)\big(\partial_\theta^{k+1}\tilde\Psi\big)^2\sin(2\theta)^{2k-\gamma}\] \[-(k+1)\int \cos^2(\theta)\big(\partial_\theta^{k+1}\tilde\Psi \big)^2sin(2\theta)^{2k-\gamma}+E,\] where $E$ is lower order and satisfies \[|Ew^2|_{L^1_R}\leq C|F|_{\mathcal{H}^k}.\]
On the other hand, we have:
\[\frac{2k+1-\gamma}{4}\int (\partial_\theta^{k+1}\tilde\Psi)^2 \cos(2\theta) \sin(2\theta)^{2k-\gamma}-(k+2)\int\sin^2(\theta)\big(\partial_\theta^{k+1}\tilde\Psi\big)^2\sin(2\theta)^{2k-\gamma}\] \[-(k+1)\int \cos^2(\theta)\big(\partial_\theta^{k+1}\tilde\Psi \big)^2sin(2\theta)^{2k-\gamma}\leq 0.\] 

\end{proof}

\subsection{The singular term}\label{SingularTerm}
In Propositions \ref{prop:L2}, \ref{prop:H2}, and \ref{prop:Hk}, one of the main conditions on $F$ is the condition \[\int_{0}^{\pi/2}F(R,\theta)\cos^2(\theta)\sin(\theta)d\theta\equiv 0.\] In fact, when $\alpha=0$ this is precisely the condition necessary to solve \eqref{PolarBSL}. Now we show how to solve the problem when \[F_\star(R):=\int_0^{\pi/2}F(R,\theta)\cos^2(\theta)\sin(\theta)d\theta\not\equiv 0.\] Note first that when $\alpha=0$, $\sin(2\theta)$ is in the kernel of $L$ (in \eqref{PolarBSL}). Consequently, if $G$ is some function of $R$ only we have:
\[L(G(R)\sin(2\theta))=(\alpha^2 R^2\partial_{RR}G+\alpha(5+\alpha)R\partial_R G) \sin(2\theta).\] Thus, if $\Psi$ solves \eqref{PolarBSL} and we define $\hat\Psi=\Psi+G\sin(2\theta)$ then we have:
\[L(\hat\Psi)=F+(\alpha^2 R^2\partial_{RR}G+\alpha(5+\alpha)R\partial_R G) \sin(2\theta).\] Now noting that $\int_{0}^{\pi/2}\sin(2\theta)\cos^2(\theta)\sin(\theta)d\theta=\frac{4}{15}$ we see that if we define $G$ by: \[\alpha^2 R^2\partial_{RR}G+\alpha(5+\alpha)R\partial_R G=-\frac{15}{4}F_\star.\] where $G$ vanishes as $R\rightarrow\infty$, then $\hat\Psi$ will enjoy all the bounds given in Propositions \ref{prop:L2} and \ref{prop:Hk}. For example, we will have that
\[|\partial_{\theta\theta}\hat\Psi|_{\mathcal{H}^k}\leq C|F|_{\mathcal{H}^k},\] which will follow from a bound on $\alpha G$ which we will observe below.

Now we just need to solve for $G$. 
Notice that \[\partial_{RR}G+\frac{5+\alpha}{\alpha}\frac{1}{R}\partial_R G=-\frac{15}{4\alpha^2 R^2}F_\star.\] 
Thus, 
\[\partial_R\Big(R^{\frac{5+\alpha}{\alpha}}\partial_RG\Big)=-\frac{15}{4\alpha^2}R^{\frac{5-\alpha}{\alpha}}F_\star\]
and so 
\[R^{\frac{5+\alpha}{\alpha}}\partial_R G=-\frac{15}{4\alpha^2}\int_0^Rs^{\frac{5-\alpha}{\alpha}}F_\star(s)ds.\]
Therefore,
\[G=-\frac{15}{4\alpha^2}\int_{R}^\infty \rho^{-\frac{5+\alpha}{\alpha}}\int_0^\rho s^{\frac{5-\alpha}{\alpha}}F_\star(s)dsd\rho.\]
Next, by integrating by parts we see:
\[G=\frac{3}{4\alpha}\int_R^\infty \partial_\rho(\rho^{-\frac{5}{\alpha}})\int_0^\rho s^{\frac{5-\alpha}{\alpha}}F_\star(s)dsd\rho=-\frac{3}{4\alpha}\int_{R}^\infty \frac{F_\star(\rho)}{\rho}d\rho-\frac{3}{4\alpha}R^{-\frac{3}{\alpha}}\int_0^R \rho^{\frac{3-\alpha}{\alpha}}F_\star(\rho)d\rho.\]
Thus, 
\[G=-\frac{1}{4\alpha}L_{12}(F)-\frac{3}{4\alpha}R^{-\frac{3}{\alpha}}\int_0^R \rho^{\frac{3-\alpha}{\alpha}}F_\star(\rho)d\rho:=G_\star+\bar{G}.\]
Next, observe that while $\bar G$ is preceded by $\frac{1}{\alpha}$, we still have a good bound for it. 
\[|\bar G|_{L^2}\leq C|F|_{L^2}\] with $C$ a constant \emph{independent} of $\alpha$. This is a consequence of the following Hardy-type inequality established in Lemma A.7 of \cite{EJDG}:
\begin{lemma} \label{HardyInequality4}
Let $\alpha>0$. For all $f\in \mathcal{H}^2([0,\infty))$ we have \[\Big|\sin(2\theta)R^{-\frac{5}{\alpha}}\int_0^R \rho^{\frac{5-\alpha}{\alpha}}f(\rho)d\rho\Big|_{\mathcal{H}^2}\leq 100\alpha|f|_{\mathcal{H}^2}.\]
\end{lemma}
We have proved the following Theorem. 
\begin{theorem}\label{RemovingL12}
Let $\alpha>0$ and $F\in\mathcal{H}^k$ given. Let $\Psi$ be the unique $C^2$ solution to \eqref{PolarBSL} which vanishes on $\theta=0$, $\theta=\frac{\pi}{2}$ and as $R\rightarrow\infty.$ Then, 
\[\alpha^2|R^2\partial_{RR}\Psi|_{\mathcal{H}^k}+|\partial_{\theta\theta}\big(\Psi-\frac{1}{4\alpha}\sin(2\theta)L_{12}(F)\big)|_{\mathcal{H}^k}\leq C|F|_{\mathcal{H}^k},\] with $C$ a universal constant independent of $\alpha, \gamma,$ and $k$ in the definition of $\mathcal{H}^k$.
\end{theorem}

\section{Some useful facts about $\mathcal{H}^k$ and $\mathcal{W}^{l,\infty}$}\label{H2}
In this section we collect a few facts about the spaces $\mathcal{H}^k$ and $\mathcal{W}^{l,\infty}$ that we will find useful. For the sake of concreteness, we will fix $k=4$ and $l=5$ in the following, but everything we will do will be applicable for any $k\geq 4$ and $l\geq 5$. We should remark that in the original version of this paper on the ArXiv the non-linear estimates were done in $\mathcal{H}^2$. In $\mathcal{H}^2$ the problem becomes critical in a certain sense and it is not clear whether the non-linear estimates can be closed in an easy way. The author thanks Jiajie Chen for pointing out this oversight. In $\mathcal{H}^4$ the estimates are significantly simpler because we have better embedding theorems. 

\subsection{Product Rules in $\mathcal{H}^4$}

We begin by proving that $\mathcal{H}^4$ embeds in a space of H\"older continuous functions.

\begin{lemma}\label{infinity_separated}Let $g\in C_c^\infty((0,\infty)\times (0,\pi/2))$. We have that
\[\sup_\theta |g(z,\theta)|^2\leq \frac{C}{\gamma-1} \int_0^{\pi/2} |\partial_\theta g(z,\theta)|^2 \sin(2\theta)^{2-\gamma}d\theta,\] and
\[\sup_z |g(z,\theta)|^2\leq C\int_0^\infty |D_z g(z,\theta)|^2 \frac{(1+z)^4}{z^4}dz,\] with $C$ a universal constant.
\end{lemma} 
\begin{proof}
The proofs of both statements is essentially the same so we only prove the first one. 
Observe that \[|g(z,\theta)|^2=|g(z,\theta)-g(z,0)|^2\leq \Big(\int_{0}^{\pi/2} |\partial_\theta g(z,\theta)|d\theta\Big)^2\leq \int_{0}^{\pi/2} |\partial_\theta g(z,\theta)|^2\sin(2\theta)^{2-\gamma}d\theta \int_{0}^{\pi/2} \sin(2\theta)^{\gamma-2}d\theta \]
\[\leq\frac{C}{\gamma-1}\int_{0}^{\pi/2} |\partial_\theta g(z,\theta)|^2\sin(2\theta)^{2-\gamma}d\theta.\]
\end{proof}
\begin{corollary}\label{Linfty} Assuming that $g\in \mathcal{H}^2$ we have that
\[|g|_{L^\infty}\leq \frac{C}{\sqrt{\gamma-1}}|g|_{\mathcal{H}^2},\] for $C$ a universal constant. In fact, for any $\beta<\gamma-1$, we have that
\[|g|_{C^\beta}\leq C_{\beta}|g|_{\mathcal{H}^2}.\]
\end{corollary}

\begin{remark}
We should remark that it is not difficult to see that the class $C_c^\infty((0,\infty)\times (0,\pi/2))$ is dense in $\mathcal{H}^k$ for any $k\in\mathbb{N}$. Using this along with the preceding corollary, it is not difficult to see that each $\mathcal{H}^2$ embeds continuously in $C^\beta$ for any $\beta<\gamma-1$. 
\end{remark}

\begin{corollary}\label{SecondDerivativesBounded} If $g\in\mathcal{H}^4$ and if $D$ is any first or second derivative of $D_z$ and $D_\theta$, then we have that $|g|_{L^\infty}+|Dg|_{L^\infty}\leq \frac{C}{\sqrt{\gamma-1}}|g|_{\mathcal{H}^4}$. 
\end{corollary}
Now we have the main product lemma in $\mathcal{H}^4$.  
\begin{lemma}\label{HProductRule}
Let $f,g\in\mathcal{H}^4$. Then, 
\[|fg|_{\mathcal{H}^4}\leq \frac{C}{\sqrt{\gamma-1}}|f|_{\mathcal{H}^4}|g|_{\mathcal{H}^4}.\]
\end{lemma}
\begin{proof}
The proof follows basically directly from Corollary \ref{SecondDerivativesBounded} and is elementary. For this reason we just consider two model terms: $D_\theta^{4} (fg)$ and $D_\theta D_z^3(fg)$, the second one being the only real difficulty. First,
\[D_\theta^4 (fg)=\sum_{i=0}^2 a_iD_\theta^i f D_\theta^{4-i} g+ \sum_{i=3}^4 a_i D_\theta^{i} f D_\theta^{4-i} g,\] where $a_i$ are binomial coefficients. 
Thus, 
\[|D_\theta^4(fg)W|_{L^2}\leq C\Big(\sum_{i=0}^2 |D_\theta^i f D_\theta^{4-i} g W|_{L^2}+ \sum_{i=3}^4 |D_\theta^{i} f D_\theta^{4-i} gW|_{L^2}\Big)\]
\[\leq C(|f|_{L^\infty}+|D_\theta f|_{L^\infty}+|D_\theta^2 f|_{L^\infty})|g|_{\mathcal{H}^4}+C(|D_\theta g|_{L^\infty}+|g|_{L^\infty})|f|_{\mathcal{H}^4}.\]
Next, consider 
$(D_\theta D_z^3)(fg).$ The problem with this term is that since, in the definition of $\mathcal{H}^4$, terms without a $\theta$ derivative have a slightly weaker weight than terms with a $\theta$ derivative, when we encounter the term $|D_z^3 f D_\theta g W|_{L^2},$ we cannot simply pull the $|D_\theta g|_{L^\infty}$ out since we do not have an estimate on $|D_z^3 f W|_{L^2}$. This problem \emph{only} occurs when we pair one derivative in $\theta$ with three derivatives in $z$ and all other terms in $|fg|_{\mathcal{H}^4}$ can be treated as the first term $D_\theta^4(fg)$ was. To overcome this, we use the full power of Lemma \ref{infinity_separated} to pull out $D_\theta g w_\theta$ in $L^\infty_z$ and $|D_z^3 f w|_{L^\infty_\theta}$. Indeed, 
\[|D_\theta g D_z^3 f W|_{L^2}^2=\int_0^\infty \int_0^{\pi/2} (D_z^3 f)^2 (\partial_\theta g)^2 \sin(2\theta)^{2-\gamma}\frac{(1+z)^4}{z^4}dzd\theta\]
\[\leq  \Big(\int_0^{\pi/2} \sup_z |\partial_\theta g|^2\sin(2\theta)^{2-\gamma} d\theta\Big)\Big(\int_0^\infty \sup_\theta |D_z^3 f|^2 \frac{(1+z)^4}{z^4}dz\Big)\]
\[\leq \frac{C}{\gamma-1} |g|_{\mathcal{H}^2}|f|_{\mathcal{H}^4},\] by Lemma \ref{infinity_separated}.

\end{proof}
We now state a corresponding statement regarding the product of $\mathcal{W}^{4,\infty}$ and $\mathcal{H}^4$ functions whose proof is also elementary\footnote{A small difference between the preceding case and this one is that the $\mathcal{W}^{4,\infty}$ part can accept four $D_\theta$ or $D_z$ derivatives in $L^\infty$ which actually makes the proof even easier.}:
\begin{lemma}\label{WProductRule}
Let $f\in\mathcal{W}^{4,\infty}$ and $g\in\mathcal{H}^4$. Then, $fg\in\mathcal{H}^4$ and
\[|fg|_{\mathcal{H}^4}\leq \frac{C}{\sqrt{\gamma-1}}|f|_{\mathcal{W}^{4,\infty}}|g|_{\mathcal{H}^4}.\]
\end{lemma}

\subsection{Transport Estimates}

We now move to state and prove some transport estimates which are similar in nature to the product rules in the preceding subsection. 
We will encounter a number of different types of transport terms, the most dangerous of which (in terms of regularity considerations) are of the form:
\[\alpha z\partial_z \Psi_g \partial_\theta g,\qquad U(\Phi_g-\frac{1}{4\alpha}\sin(2\theta)L_{12}(g))\partial_\theta g,\qquad V(\Phi_g-\frac{1}{4\alpha}\sin(2\theta)L_{12}(g))z\partial_z g,\]
where the operator $U$ can be thought of as $Id+\alpha z\partial_z$ and $V$ can be thought of as $\partial_\theta$. In view of the elliptic estimates from Theorem \ref{RemovingL12}, we can think of the above transport terms as $g$ being transported by a velocity which is one derivative more regular than $g$ (as is classically the case with the transport term in the vorticity equation). 
We can thus formulate the following simple lemmas.

\begin{lemma}\label{Transport1}
Assume $f\in \mathcal{H}^4$ and $g\in \mathcal{H}^4$. Then, 
\[|(f D_\theta g, g)_{\mathcal{H}^4}|\leq \frac{C}{\sqrt{\gamma-1}}|f|_{\mathcal{H}^4}|g|_{\mathcal{H}^4}^2.\]
\end{lemma}
\begin{proof}
The proof of the lemma is again elementary and we only consider two model cases: $(D_\theta^4 (f D_\theta g), D_\theta^4 g W^2)_{L^2}$ and $(D_z^3 D_\theta(fD_\theta g), D_z^3D_\theta g W^2)_{L^2},$ the delicate one being the second one.  
For the first one, observe that
\[(D_\theta^4(fD_\theta g), D_\theta^4 g W^2)_{L^2}=\sum_{i=1}^2 a_i(D_\theta^i f D_\theta^{5-i}g, D_\theta^4g W^2)_{L^2}+\sum_{i=3}^4 a_i(D_\theta^i f D_\theta^{5-i}g, D_\theta^4 W^2)_{L^2}+\frac{1}{2} ( D_\theta((D_\theta^4 g)^2), fW^2)_{L^2}\]
\[\leq C (|D_\theta f|_{L^\infty}+|D_\theta^2 f|_{L^\infty})|g|_{\mathcal{H}^4}^2+C(|D_\theta g|_{L^\infty}+|D^2_\theta g|_{L^\infty})|f|_{\mathcal{H}^k}|g|_{\mathcal{H}^k}. \]
\[\leq \frac{C}{\sqrt{\gamma-1}}|f|_{\mathcal{H}^k}|g|_{\mathcal{H}^k}^2.\]
Now we move to investigate the term $(D_z^3 D_\theta(fD_\theta g), D_z^3D_\theta g W^2)_{L^2}.$ Upon expanding the derivatives in this expression using Leibniz's rule as above, we see that there is again only one term that is not treated as above: $(D_z^3 f D_\theta^2 g, D_z^3 D_\theta g W^2)_{L^2}.$ As in the proof of Lemma \ref{HProductRule}, we cannot simply pull out $|D_\theta^2 g|_{L^\infty}.$ Thus we first apply Cauchy-Schwarz and then we observe that:   
\[\sqrt{\int_0^\infty \int_0^{\pi/2} (D_z^3 f)^2 (\partial_\theta D_\theta g)^2 \sin(2\theta)^{2-\gamma}\frac{(1+z)^4}{z^4}dzd\theta}\leq \frac{C}{\sqrt{\gamma-1}} |f|_{\mathcal{H}^4}|D_\theta g|_{\mathcal{H}^2}\leq \frac{C}{\sqrt{\gamma-1}}|f|_{\mathcal{H}^4}|g|_{\mathcal{H}^4}, \] where we applied Lemma \ref{infinity_separated}
just as in the second part of the proof of Lemma \ref{HProductRule}. 
\end{proof}
We next state the corresponding result for $D_z$ derivatives, whose proof is identical to the proof of Lemma \ref{Transport1}. 
\begin{lemma}\label{Transport2}
If $f, g\in \mathcal{H}^4$, then
\[|(f D_z g, g)_{\mathcal{H}^4}|\leq \frac{C}{\sqrt{\gamma-1}}|f|_{\mathcal{H}^4}|g|_{\mathcal{H}^4}^2.\]
\end{lemma}
We further have two more similar transport estimates with $\mathcal{W}^{4,\infty}$ velocity. 
\begin{lemma}\label{Transport3}
If $f\in \mathcal{W}^{4,\infty}$ and $g\in\mathcal{H}^4$, then
\[|(f D_z g, g)_{\mathcal{H}^4}|+|(f D_\theta g, g)_{\mathcal{H}^4|}\leq \frac{C}{\sqrt{\gamma-1}}|f|_{\mathcal{W}^{4,\infty}}|g|_{\mathcal{H}^4}^2.\]
\end{lemma}
The proof of Lemma \ref{Transport3} is much simpler than the proof of Lemma \ref{Transport1} since when any derivative hits $f$ we can simply estimate that term pointwise using $|f|_{\mathcal{W}^{4,\infty}}$ while the other term can be estimated using integration by parts. 

Let us end by mentioning a slightly more non-trivial estimate that we will use. 
\begin{proposition}\label{SpecialTransport}
Let $g\in\mathcal{H}^4$ and assume that $f$ is such that $\partial_\theta f\in\mathcal{W}^{4,\infty}$ and $f(z,0)=f(z,\pi/2)=0$, for all $z\in [0,\infty)$.  Then, 
\[(f\partial_\theta g,g)_{\mathcal{H}^4}\leq \frac{C}{\sqrt{\gamma-1}}|\partial_\theta f|_{\mathcal{W}^{4,\infty}}|g|_{\mathcal{H}^4}^2.\]
\end{proposition}
\begin{proof}
The proof is slightly more delicate here so we will go into more detail. We will deal with $D_\theta$ and $D_z$ derivatives slightly differently. First observe that for any $0\leq j\leq 4$ we have that \[|D_z^j\frac{f}{\sin(2\theta)}|_{L^\infty}\leq C|D_z^j\partial_\theta f|_{L^\infty}\leq C|\partial_\theta f|_{\mathcal{W}^{4,\infty}}.\] 
Now, $(f\partial_\theta g,g)_{\mathcal{H}^4}$ can be written as a sum of terms of the form:
\[(D_\theta^{i_1}D_z^{i_2}f D_{\theta}^{j_1}D_z^{j_2}\partial_\theta g, D_{\theta}^{i_1+j_1}D_{z}^{i_2+j_2}g W_{i,j})_{L^2},\] where $0\leq i_1+i_2+j_1+j_2\leq 4$ and $W_{i,j}$ is $\frac{w}{\sin(2\theta)^\eta}$ if $i_1=j_1=0$ and $\frac{w}{\sin(2\theta)^\gamma}$ otherwise. We only need to consider a couple of cases. 

{\bf Case 1:} $i_1=i_2=0$ (all derivatives fall on $g$). Here we will want to integrate by parts except in the commutator terms. The information we will use about $f$ is that
\[|\frac{f}{\sin(2\theta)}|_{L^\infty}+|\partial_\theta f|_{L^\infty}\leq C|f|_{\mathcal{W}^{4,\infty}}\] since $f(z,0)=f(z,\pi/2)=0$ for all $z$. 

We have:
\[(f D_\theta^{j_1} \partial_\theta D_z^{j_2} g, D_{\theta}^{j_1}D_z^{j_2}g W_{i,j})_{L^2}=(f \partial_\theta(D_{\theta}^{j_1}D_z^{j_2}g ), D_{\theta}^{j_1}D_z^{j_2}g W_{i,j})_{L^2}\]\[+(f[D_\theta^{j_1}, \partial_\theta]D_z^{j_2}g, D_{\theta}^{j_1}D_z^{j_2}g W_{i,j})_{L^2}:=I+II.\] 
Now, \[|I|=|(f \partial_\theta(D_{\theta}^{j_1}D_z^{j_2}g ), D_{\theta}^{j_1}D_z^{j_2}g W_{i,j})_{L^2}|=\frac{1}{2}|(\partial_{\theta}(fW_{i,j}), (D_{\theta}^{j_1}D_z^{j_2}g )^2)_{L^2}|\leq C|f|_{\mathcal{W}^{4,\infty}}|g|_{\mathcal{H}^4}^2,\] since \[|\partial_\theta(f W_{i,j})|\leq C|f|_{\mathcal{W}^{4,\infty}}W_{i,j},\] for $C$ a universal constant where we just used boundedness of $\partial_\theta f$ and $\frac{f}{\sin(2\theta)}$. Next, for the commutator term:
\[II=(f[D_\theta^{j_1}, \partial_\theta]D_z^{j_2}g, D_{\theta}^{j_1}D_z^{j_2}g W_{i,j})_{L^2}.\]
It is easy to see that \[|\sin(2\theta)[D_\theta^{j_1},\partial_\theta] D_{z}^{j_2}g|\leq C\sum_{i=1}^{j_1} |D_{\theta}^i D_{z}^{j_2}g|.\]
Note that the sum (and the commutator) is 0 when $j_1=0$. Thus, 
\[|II|\leq |\frac{f}{\sin(2\theta)}|_{L^\infty}|g|_{\mathcal{H}^4}^2.\]
This concludes Case 1.

{\bf Case 2:} $i_1+i_2\geq 1$. This case is very similar to the above and we leave it to the reader. 
\end{proof}

\subsection{Functions which belong to $\mathcal{W}^{5,\infty}$}
We now give the main example of a function belonging to $\mathcal{W}^{5,\infty}$ which we will use. We remind the reader that \[\gamma=1+\frac{\alpha}{10}.\] We let \[\Gamma(\theta)=(\sin(\theta)\cos^2(\theta))^{\alpha/3}.\]
\begin{proposition}\label{GammaRegularity}
$\Gamma\in \mathcal{W}^{5,\infty}$ with norm independent of $\alpha$. 
\end{proposition}
\begin{proof}
We compute:
\[\partial_\theta\Gamma=\frac{\alpha}{3} \Gamma\frac{\cos^2(\theta)-2\sin^2(\theta)}{\sin(\theta)\cos(\theta)}.\]
This shows that 
\[\sin(2\theta)\partial_\theta \Gamma=\frac{2\alpha}{3} \Gamma (\cos^2(\theta)-2\sin^2(\theta)).\]
Thus, \[|D_\theta \Gamma\sin(2\theta)^{1-\gamma}|\leq \frac{2\alpha}{3}.\]
Moreover,
\[|D_\theta^2\Gamma\sin(2\theta)^{1-\gamma}|\leq \frac{4\alpha^2}{3}+\frac{20\alpha}{3}\sin(2\theta).\]
Similar calculations for the third and fourth derivatives give the proof. 
\end{proof}

We also have the following clear proposition:

\begin{proposition}\label{ProductRuleInW}
There exists a universal constant $C>0$ so that if $f,g\in \mathcal{W}^{4,\infty},$ then \[|fg|_{\mathcal{W}^{4,\infty}}\leq C|f|_{\mathcal{W}^{4,\infty}}|g|_{\mathcal{W}^{4,\infty}}.\]
\end{proposition}
Observe now that if we let  $F_*(z,\theta)=\frac{2\alpha \Gamma(\theta)}{c}\frac{z}{(z+1)^2}$, then we have that $L_{12}(F_*)=\frac{2\alpha}{z+1}.$ Thus, $L_{12}(F_*)\in\mathcal{W}^{5,\infty}$. Moreover, we define $\Phi_*$ by \[-\alpha^2z^2\partial_{zz}\Phi_*-\alpha(5+\alpha)z\partial_z\Phi_*-\partial_{\theta\theta}\Phi_*+\partial_\theta\big(\tan(\theta)\Phi_*\big)-6\Phi=F_*.\] 
\begin{proposition}\label{PhiStarRegularity}
We have \[|\partial_{zz}(\Phi_*-\frac{1}{4\alpha}\sin(2\theta)L_{12}(F_*))|_{\mathcal{W}^{5,\infty}}+|\frac{z+1}{z}\partial_{\theta\theta}(\Phi_*-\frac{1}{4\alpha}\sin(2\theta)L_{12}(F_*))|_{\mathcal{W}^{5,\infty}}\leq C\alpha.\]
\end{proposition}
\begin{proof}
As in Section \ref{SingularTerm}, we first note that $\int_{0}^{\pi/2}F_*(\theta,z)\sin(\theta)\cos^2(\theta)dz=\frac{2\alpha z}{(1+z)^2}.$ Now define $G$ by 
\[\alpha^2z^2\partial_{zz}G+\alpha(5+\alpha)z\partial_zG=-\frac{15}{4}\frac{2\alpha z}{(1+z)^2}.\]
\end{proof}
Now we observe, as in Section \ref{SingularTerm}, that 
\[G=-\frac{1}{4\alpha} L_{12}(F_*)-\frac{3}{2}z^{-5/\alpha}\int_0^z\rho^{\frac{5-\alpha}{\alpha}}\frac{\rho}{(1+\rho)^2}d\rho\]
\[=-\frac{1}{4\alpha}L_{12}(F_*)-\frac{3}{2}z^{-5/\alpha}\int_0^z \frac{\rho^{5/\alpha}}{(1+\rho)^2}d\rho.\]
\[=-\frac{1}{4\alpha}L_{12}(F_*)-\frac{3\alpha}{2}z^{-5/\alpha} \frac{\rho^{5/\alpha+1}}{5+\alpha}\frac{1}{(1+\rho)^2}\Big|_{\rho=0}^z+\frac{3\alpha}{5+\alpha}z^{-5/\alpha}\int_0^z \frac{\rho^{5/\alpha+1}}{(1+\rho)^3}d\rho.\]
\[=-\frac{1}{4\alpha}L_{12}(F_*)-\frac{3\alpha}{2(5+\alpha)}\frac{z}{(1+z)^2}+\frac{3\alpha}{5+\alpha}z^{-5/\alpha}\int_0^z \frac{\rho^{5/\alpha+1}}{(1+\rho)^3}d\rho.\]
In particular, it is easy to see that \[|G+\frac{1}{4\alpha}L_{12}(F_*)|_{W^{4,\infty}}\leq C\alpha.\]
Now, from Proposition \ref{prop:L2}, we get that \[\alpha^2 |z^2\partial_{zz}(\Phi_*+\sin(2\theta)G)|_{L^2}+ |\partial_{\theta\theta}(\Phi_*+\sin(2\theta)G)|_{L^2}\leq C|F_*|_{L^2}\leq C\alpha.\]
Since $F_*$ is infinitely smooth in $z$ we also have that
\[\alpha^2 |\partial_z^k\Big(z^2\partial_{zz}(\Phi_*+\sin(2\theta)G)\Big)|_{L^2}+ |\partial_z^k\partial_{\theta\theta}(\Phi_*+\sin(2\theta)G)|_{L^2}\leq C_k\alpha\] for any integer $k$. 
Consequently, we define \[
\bar\Phi=\Phi_*+\sin(2\theta)G.
\] and $\bar F_*=F_*-\frac{15}{2}\sin(2\theta)\frac{\alpha z}{(1+z)^2}$ (so that $\int_0^{\pi/2}\bar F_* K=0$) and we see:
\[-\partial_{\theta\theta}\bar\Phi+\partial_\theta\big(\tan(\theta)\bar\Phi\big)=\bar{F_*}+6\bar\Phi+\alpha^2z^2\partial_{zz}\bar\Phi+\alpha(5+\alpha)z\partial_z\bar\Phi.\] It is now easy to see (since we have arbitrary smoothness in $z$) that most of the terms on the right hand side can be neglected and that to establish the proposition it suffices to show that the solution $\tilde\Phi$ to 
\[\partial_{\theta\theta}\tilde\Phi-\partial_\theta (\tan(\theta)\tilde\Phi) =\Gamma(\theta)\] 
satisfies: 
\[|\partial_{\theta\theta}\tilde\Phi|_{\mathcal{W}^{4,\infty}}\leq C.\] We do this by directly solving the above equation. 
We see:
\[\partial_\theta\tilde\Phi-\tan(\theta)\tilde\Phi=\int_0^\theta\Gamma(\theta')d\theta'+C_1\]
\[\partial_\theta (\tilde\Phi\cos(\theta))=\cos(\theta)\int_0^\theta \Gamma(\theta')d\theta'+C_1\cos(\theta)\]
Thus, 
\[\tilde\Phi =\frac{1}{\cos(\theta)}\int_0^\theta\cos(\beta)\int_0^\beta\Gamma(\theta')d\theta'+C_1\frac{\sin(\theta)}{\cos(\theta)}\]
\[=\frac{1}{\cos(\theta)}\Big(\sin(\theta)\int_0^\theta\Gamma(\beta)d\beta+C_1\sin(\theta)-\int_0^{\theta}\sin(\beta)\Gamma(\beta)d\beta\Big).\] $C_1$ is now chosen to keep the boundary condition $\tilde\Phi(\pi/2)=0$. 
That is, we want:
\[\int_0^{\pi/2}(1-\sin(\beta))\Gamma(\beta)d\beta=-C_1.\]
Now we want to sketch why $\partial_{\theta\theta}\tilde\Phi\in \mathcal{W}^{4,\infty}.$ Clearly $\tilde\Phi$ is infinitely differentiable away from $\theta=0$ and $\theta=\pi/2$. 
Let us thus focus our attention near $\theta=\pi/2$ (near $\theta=0$ the situation is the same and even easier). Near $\pi/2$ we may replace the $\sin(\theta)$ by $1$ up to a vanishing perturbation which leads to a smoother term. 
Thus we get, taking into account the definition of $C_1$:
\[\frac{1}{\cos(\theta)}\int_\theta^{\pi/2}(1-\sin(\beta))\Gamma(\beta)d\beta\approx \frac{1}{\pi/2-\theta} \int_\theta^{\pi/2}(\pi/2-\beta)^{2+2\alpha/3}d\beta\approx (\pi/2-\beta)^{2+2\alpha/3}.\]
Thus, near $\pi/2$
\[\partial_{\theta\theta}\tilde\Phi\approx (\pi/2-\beta)^{2\alpha/3}.\] 
In the above, as before, we have to use the inequality:
\[|\sin(\theta)^\alpha-\theta^\alpha|\leq C |\theta|^2\] for $C$ a universal constant independent\footnote{In fact, it is $C\alpha$ but we do not need this extra smallness.} of $\alpha$. We leave further details to the reader. 

Next, we have the following Proposition.

\begin{proposition}
Assume that $f$ is smooth and that $f(0,\theta)=0$ for all $\theta\in [0,\pi/2]$. Then,
\[|\frac{z+1}{z}f|_{\mathcal{W}^{4,\infty}}\leq C|f|_{\mathcal{W}^{5,\infty}}.\]
\end{proposition}

\begin{proof}
The proof simply consists of examining the Taylor expansion of $f$ in $z$.  
\end{proof}

A corollary is the following
\begin{corollary}\label{PhiStarRegularity2} We have that $\frac{1}{z}F_*, \frac{1}{z}(\Phi_*-\frac{1}{4\alpha}\sin(2\theta)L_{12}(F_*))\in\mathcal{W}^{4,\infty}$ and \[|\frac{(z+1)^2}{z}F_*|_{\mathcal{W}^{4,\infty}}+ |\frac{z+1}{z}\partial_{\theta\theta}(\Phi_*-\frac{1}{4\alpha}\sin(2\theta)L_{12}(F_*))|_{\mathcal{W}^{4,\infty}}\leq C\alpha.\]
\end{corollary}
This follows from the preceding Proposition and Proposition \ref{PhiStarRegularity}.
\begin{remark}
It is true that $F_*$ is not smooth but that it is only smooth in $z$ (though it has some regularity in $\theta$). Nevertheless, it is easy to see that $F_*$ can be approximated by smooth functions in the $\mathcal{W}^{l,\infty}$ norm for any $l$ so that it easy to derive the corollary from the Proposition.  
\end{remark}

\subsection{From $\mathcal{W}^{4,\infty}$ to $\mathcal{H}^4$.}
We will also find the following proposition useful. 

\begin{proposition}\label{WtoH}
There exists a universal constant $C>0$ (independent of $\alpha$) so that if $\frac{(z+1)^3}{z^2}f\in \mathcal{W}^{4,\infty}$, then $f\in\mathcal{H}^4$ and \[|f|_{\mathcal{H}^4}\leq C\Big|\frac{(z+1)^3}{z^2} f\Big|_{\mathcal{W}^{4,\infty}}.\]

\end{proposition}

\begin{proof}
The main term to consider in $|f|_{\mathcal{H}^4}$ is $|D_z^4 f \frac{w}{\sqrt{\sin(2\theta)^\eta}}|_{L^2}.$ We see:
\[|D_z^4 f \frac{w}{\sqrt{\sin(2\theta)^\eta}}|_{L^2}^2\leq C\sum_{i=1}^4 |\partial_{z}^i f\frac{z^iw}{\sqrt{\sin(2\theta)^\eta}}|_{L^2}^2  \leq C\sum_{i=1}^{4}|(z+1)\partial_{z}^if z^i w|_{L^\infty}^2, \] where all we used is that $\frac{1}{\sin(2\theta)^\eta}$ is integrable. 
On the other hand, by definition,
\[|\frac{(z+1)^3}{z^2}f|_{\mathcal{W}^{4,\infty}}\geq \sum_{i=0}^4\Big|(z+1)^i \partial_z^i\Big(\frac{(z+1)^3}{z^2}f\Big)\Big|_{L^\infty}\]
Now let $f=z^2 k$. Then we see:
\[|\frac{(z+1)^3}{z^2}f|_{\mathcal{W}^{4,\infty}}\geq \sum_{i=0}^{4}|(z+1)^i\partial_{z}^i\Big((z+1)^3k\Big)|_{L^\infty}\geq c\sum_{i=0}^4|(z+1)^{i+3} \partial_{z}^ik|_{L^\infty},\] for some universal constant $c>0$. 
Now, we need to bound \[\sum_{i=1}^{4}|(z+1)\partial_{z}^if z^i w|_{L^\infty}^2=\sum_{i=1}^{4}|(z+1)z^i w\partial_{z}^i (z^2k) |_{L^\infty}^2\leq C\sum_{i=0}^4|(z+1)z^{i+2}w \partial_z^i k|_{L^\infty}^2=C\sum_{i=0}^4|(z+1)^3z^i\partial_z^i k|_{L^\infty}^2\]
\[\leq C|\frac{(z+1)^3}{z^2}f|_{\mathcal{W}^{4,\infty}}^2.\] 
The rest of the terms are treated similarly or in a simpler way. 
\end{proof}

We also have the following useful Lemma. 
\begin{lemma} \label{WtoW}
Assume that $\partial_\theta f\in\mathcal{W}^{4,\infty}$, $ g\in\mathcal{W}^{5,\infty},$ and that $f(z,0)=f(z,\pi/2)=0$ for all $z\in [0,\infty)$. Then, 
\[|f\partial_\theta g|_{\mathcal{W}^{4,\infty}}\leq C|\partial_\theta f|_{\mathcal{W}^{4,\infty}}|g|_{\mathcal{W}^{5,\infty}},\] for $C$ a universal constant. 
\end{lemma}

\begin{remark}
If we knew that $\partial_\theta f\in\mathcal{W}^{4,\infty}$ implies that $\frac{f}{\sin(2\theta)}\in \mathcal{W}^{4,\infty}$, this result would follow directly from Proposition \ref{ProductRuleInW}. 
\end{remark}

\begin{proof}
Let us just consider the cases when the derivatives involved are $\partial_z^4$ and $D_{\theta}^4$. 
Observe that
\[|(z+1)^4\partial_z^4 (f\partial_\theta g)|_{L^\infty}\leq \sum_{j=0}^4|(z+1)^j \frac{\partial_z^jf}{\sin(2\theta)} (z+1)^{4-j}D_\theta \partial_{z}^{4-j}g|_{L^\infty}\]
\[\leq |g|_{\mathcal{W}^{5,\infty}}\sum_{j=0}^4|(z+1)^j \frac{\partial_z^jf}{\sin(2\theta)}|_{L^\infty}\leq C|g|_{\mathcal{W}^{5,\infty}}\sum_{j=0}^4 \sup_{z}\sup_{\theta} |(z+1)^j \partial_{z}^j\partial_\theta f(z,\theta)|\leq |\partial_\theta f|_{\mathcal{W}^{4,\infty}}|g|_{\mathcal{W}^{5,\infty}}.\]
Next, we observe that
\[|D_\theta^4(f\partial_\theta g) \frac{\sin(2\theta)^{-\alpha/5}}{\alpha+\sin(2\theta)}|_{L^\infty}\leq C\sum_{j=0}^4 |D_\theta^j\partial_\theta f|_{L^\infty} |g|_{\mathcal{W}^{5,\infty}}\leq C|\partial_\theta f|_{\mathcal{W}^{4,\infty}}|g|_{\mathcal{W}^{5,\infty}}.\]
\end{proof}

\subsection{Other Useful Facts}
In this subsection, we establish two more elementary facts. 
First we have the following
\begin{proposition}\label{L12H}
There exists a universal constant $C>0$ so that for all $g\in\mathcal{H}^4$ with $L_{12}(g)(0)=0$ we have that $L_{12}(g)\in\mathcal{H}^4$ and
\[|L_{12}(g)|_{\mathcal{H}^4}\leq C|g|_{\mathcal{H}^4}.\]
\end{proposition}

\begin{proof}
Let $g\in\mathcal{H}^4$ be such that $L_{12}(g)(0)=0$. Observe that $L_{12}(g)$ is independent of $\theta$ so all we only need to show that \[|D_z^k L_{12}(g)w|_{L^2_z}\leq C|g|_{\mathcal{H}^4},\]  for $0\leq k\leq 4$. Since \[D_R L_{12}(g)=(g, K)_{L^2_\theta},\] the only non-trivial case is $k=0$. In the case $k=0$ observe that:
\[\int_{0}^1 \frac{(z+1)^4}{z^2}(L_{12}(g)(z))^2dz\leq C\int_{0}^1\frac{1}{z^4} \Big(\int_{0}^{\pi/2}\int_{0}^z\frac{g(r,\theta)}{r}dr d\theta\Big)^2 dz:=I\] where we used that $L_{12}(g)(0)=0$.
Now observe that, by integration by parts, \[I=\frac{3}{2}\int_{0}^1 \Big(\int_{0}^{\pi/2}\frac{g(z,\theta)}{z^2}\Big)\frac{1}{z^2}\Big(\int_{0}^{\pi/2}\int_{0}^z\frac{g(r,\theta)}{r}dr d\theta\Big)dz\leq \frac{3}{2}\sqrt{I} \int_{0}^1 \Big(\int_{0}^{\pi/2} \frac{g(z,\theta)}{z^2}d\theta\Big)^2\leq\] \[C\sqrt{I}\int_{0}^1\int_{0}^{\pi/2} \frac{g(z,\theta)^2}{z^4}dzd\theta\leq C\sqrt{I}|g|_{\mathcal{H}^4}.\] Thus, 
$I\leq C|g|_{\mathcal{H}^4}^2$ and
\[\int_{0}^1 \frac{(z+1)^4}{z^2}(L_{12}(g)(z))^2dz\leq C|g|_{\mathcal{H}^4}^2.\] 
Next, observe that
\[\int_{1}^\infty \frac{(z+1)^4}{z^4}(L_{12}(g)(z))^2\leq |g|_{\mathcal{H}^4}^2\] again using integration by parts as above and in the proof of Lemma \ref{HardyWeight}.

\end{proof}

We also have the following useful Proposition
\begin{proposition}\label{DivisionBySin}
Let $\Phi: [0,\infty)\times [0,\pi/2]\rightarrow\mathbb{R}$ be smooth and rapidly decaying in $z$. Assume also that $\Phi(z,0)=\Phi(z,\pi/2)=0$ for all $z\in [0,\infty)$. Then, $\frac{1}{\sin(2\theta)}\Phi\in \mathcal{H}^4$ and
\[|\frac{1}{\sin(2\theta)}\Phi|_{\mathcal{H}^4}\leq C|\partial_\theta\Phi|_{\mathcal{H}^4}.\]
\end{proposition}
\begin{remark}
By density, this of course extends to all $\Phi$ with $\partial_\theta\Phi\in\mathcal{H}^4$ and $\Phi(z,0)=\Phi(z,\pi/2)=0$ for all $z$. 
\end{remark}
The proof directly follows from two lemmas. The first,  Lemma \ref{SharpHardyInequality1}, has already been established and the second is a variant on Lemma \ref{HardyInequality2}: 
\begin{lemma}Assume that $f\in C_c^\infty((0,\infty))$ and $k\in\mathbb{N}$. Then,
\[\int_{0}^{\infty}x^{2k-\gamma}\Big(\partial_x^{k} \big(\frac{f(x)}{x}\big)\Big)^2dx\leq \int_{0}^\infty x^{2k-\gamma}(\partial_{x}^{k+1}f)^2.\] 
\end{lemma}

\begin{proof}
Define $h=\frac{f(x)}{x}$. Then, $\partial_{x}^{k+1}f=\partial_{x}^{k+1}hx+(k+1)\partial_x^kh$. 
Now observe that
\[\int_{0}^\infty x^{2k-\gamma}\partial_{x}^k h \partial_{x}^{k+1}f=\int_{0}^\infty x^{2k-\gamma} ( (k+1)(\partial_{x}^k h)^2+x\partial_{x}^k h\partial_{x}^{k+1}h)=[(k+1)-\frac{1}{2}(2k+1-\gamma)] \int_{0}^{\infty}x^{2k-\gamma} (\partial_{x}^k h)^2\]
\[=(1+\frac{\gamma}{2})\int_{0}^{\infty}x^{2k-\gamma}(\partial_{x}^k h)^2.\] The result now follows from the Cauchy-Schwarz inequality. 

\end{proof}

\section{Self-similar variables and modulation}\label{Modulation}
Recall from Section \ref{Setup} the following equations for $\Omega$ and $\Psi$:
\begin{equation}\label{OmegaEvolutionFinal}\frac{1}{2}\partial_t\Omega+U(\Psi)\partial_\theta\Omega+V(\Psi)\alpha R\partial_R\Omega=\mathcal{R}(\Psi) \Omega,\end{equation}

\begin{equation}\label{Relations}
U(\Psi):=-3\Psi-\alpha R\partial_R\Psi\quad V(\Psi):=\partial_\theta\Psi-\tan(\theta)\Psi, \quad \R(\Psi):=\frac{1}{\cos(\theta)}\Big(2\sin(\theta)\Psi+\alpha\sin(\theta)R\partial_R\Psi+\cos(\theta)\partial_\theta\Psi\Big),
\end{equation}

\begin{equation}\label{PolarBSLFinal1}-\alpha^2R^2\partial_{RR}\Psi-\alpha(5+\alpha)R\partial_R\Psi-\partial_{\theta\theta}\Psi+\partial_\theta\big(\tan(\theta)\Psi\big)-6\Psi=\Omega.\end{equation} 
Note that the $\frac{1}{2}$ preceding the $\partial_t\Omega$ is there for convenience and can be viewed as a time-dilation. 
Let us search for a solution of the form \[\Omega=\frac{1}{1-(1+\mu)t}F\Big(\frac{R}{(1-(1+\mu)t)^{1+\lambda}},\theta\Big)\] where $\mu$ and $\lambda$ are small real numbers.  We introduce the self-similar variable \[z=\frac{R}{(1-(1+\mu)t)^{1+\lambda}}.\] It is easy to see that if $\Omega$ has the above form, then $\Psi$ should have the form: \[\Psi=\frac{1}{1-(1+\mu)t}\Phi(z,\theta).\] Now we write the equations for $F$ and $\Phi$:

\[(1+\mu)F+(1+\mu)(1+\lambda) z\partial_z F +2U(\Phi)\partial_\theta F+2\alpha V(\Phi) R\partial_RF=2\mathcal{R}(\Phi) F\]
\[U(\Phi):=-3\Phi-\alpha R\partial_R\Phi\quad V(\Phi):=\partial_\theta\Phi-\tan(\theta)\Phi, \quad \R(\Phi):=\frac{1}{\cos(\theta)}\Big(2\sin(\theta)\Phi+\alpha\sin(\theta)R\partial_R\Phi+\cos(\theta)\partial_\theta\Phi\Big),\]
\[-\alpha^2R^2\partial_{RR}\Phi-\alpha(5+\alpha)R\partial_R\Phi-\partial_{\theta\theta}\Phi+\partial_\theta\big(\tan(\theta)\Phi\big)-6\Phi=F.\]
Now recall from the elliptic estimates that $\Phi-\frac{1}{4\alpha}\sin(2\theta)L_{12}(F)$ satisfies much better estimates than $\Phi$ itself. Thus we write:

\[(1+\mu)F+(1+\mu)(1+\lambda) z\partial_z F +\frac{1}{2\alpha}U(\sin(2\theta)L_{12}(F))\partial_\theta F+\frac{1}{2}V(\sin(2\theta)L_{12}(F)) R\partial_RF-\frac{1}{2\alpha}\mathcal{R}(\sin(2\theta)L_{12}(F)) F\] \[=2\mathcal{R}(\Phi-\frac{1}{4\alpha}\sin(2\theta)L_{12}(F)) F-2U(\Phi-\frac{1}{4\alpha}\sin(2\theta)L_{12}(F))\partial_\theta F-2\alpha V(\Phi-\frac{1}{4\alpha}\sin(2\theta)L_{12}(F)) R\partial_RF.\]
Now let's compute:
\[U(\sin(2\theta)L_{12}(F))=-3\sin(2\theta)L_{12}(F)+\alpha\sin(2\theta)(F,K)_{L^2_\theta}.\]
\[V(\sin(2\theta)L_{12}(F))=2(\cos(2\theta)-\sin^2(\theta))L_{12}(F),\]
\[\mathcal{R}(\sin(2\theta)L_{12}(F))=2L_{12}(F)-2\alpha\sin^2(\theta) (F,K)_{L^2_\theta}.\]
Thus, 
\[(1+\mu)F+(1+\mu)(1+\lambda) z\partial_z F-\frac{1}{\alpha}L_{12}(F) F -\frac{3}{2\alpha}\sin(2\theta)L_{12}(F)\partial_\theta F+(\cos(2\theta)-\sin^2(\theta))L_{12}(F) z\partial_zF=\mathcal{N},\] where \[\mathcal{N}=2\mathcal{R}(\Phi-\frac{1}{4\alpha}\sin(2\theta)L_{12}(F)) F-2U(\Phi-\frac{1}{4\alpha}\sin(2\theta)L_{12}(F))\partial_\theta F\] \[-2\alpha V(\Phi-\frac{1}{4\alpha}\sin(2\theta)L_{12}(F)) R\partial_RF-\alpha\sin(2\theta)(F,K)_{L^2_\theta}\partial_\theta F-2\alpha\sin^2(\theta)(F,K)_{L^2_\theta}F.\]
Rewriting this once more we get:
\[F+z\partial_zF -\frac{1}{\alpha}L_{12}(F) F -\frac{3}{2\alpha}\sin(2\theta)L_{12}(F)\partial_\theta F+(\cos(2\theta)-\sin^2(\theta))L_{12}(F) z\partial_zF=-\mu F-(\mu+\lambda+\mu\lambda)z\partial_z F +\mathcal{N},\] 

Next, we write:
\[F=F_*+g,\] where 
\[F_*=\alpha\frac{\Gamma(\theta)}{c}\frac{2z}{(1+z)^2},\] with \begin{equation}\label{Gamma}\Gamma(\theta)=(\sin(\theta)\cos^2(\theta))^{\alpha/3}\end{equation} and $c=\int_{0}^{\pi/2}\Gamma(\theta)K(\theta)d\theta$. $\mu$ and $\lambda$ will be chosen to ensure that there exists a (small) $g\in\mathcal{H}^2$ with $L_{12}(g)(0)=0$ so that $F=F_*+g$ solves the above. Now we write the equation for $g$, noting that \[F_*+z\partial_z F_*-\frac{1}{\alpha}L_{12}(F_*)F_*=0,\]
we get:

\begin{equation}\label{intermediategeqn}\mathcal{L}_\Gamma(g)-\frac{3}{2\alpha}\sin(2\theta)L_{12}(F_*)\partial_\theta g=-\mu F_*-(\mu+\lambda+\mu \lambda)z\partial_z F_*+\mathcal{N}_0+\mathcal{N}+\mathcal{N}_*,\end{equation} where
\[\mathcal{N}_0=\frac{3}{2\alpha}\sin(2\theta)L_{12}(F_*)\partial_\theta F_*-(\cos(2\theta)-\sin^2(\theta))L_{12}(F_*)z\partial_z F_*+\frac{1}{\alpha}L_{12}(g) g+\frac{3}{2\alpha}\sin(2\theta)L_{12}(g)\partial_\theta F\]\[-(\cos(2\theta)-\sin^2(\theta))L_{12}(g)z\partial_z F-(\cos(2\theta)-\sin^2(\theta))L_{12}(F_*)z\partial_z g,\]
\[\mathcal{N}_*=-\mu g-(\mu+\lambda+\mu\lambda)g.\]
We now re-write \eqref{intermediategeqn} as:
\[\mathcal{L}_\Gamma^T(g)=-\frac{\Gamma(\theta)}{c}\frac{2z^2}{(1+z)^3}L_{12}(\frac{3}{1+z}\sin(2\theta)\partial_\theta g)(0)-\mu F_*-(\mu+\lambda+\mu\lambda)z\partial_z F_*+\mathcal{N}_0+\mathcal{N}+\mathcal{N}_*.\]
We now choose $\lambda$ so that $\mu+(\mu+\lambda+\mu\lambda)=0$. This will cancel all terms which vanish only linearly at $z=0$ in the above equation. 
That is, we take:
\[\lambda=-\frac{2\mu}{\mu+1}.\]
Thus we get:
\[\mathcal{L}_\Gamma^T(g)=-\frac{\Gamma(\theta)}{c}\frac{2z^2}{(1+z)^3}L_{12}(\frac{3}{1+z}\sin(2\theta)\partial_\theta g)(0)-\mu F_*+\mu z\partial_z F_*+\mathcal{N}_0+\mathcal{N}+\mathcal{N}_*,\] which becomes:
\[\mathcal{L}_\Gamma^T(g)=-\frac{\Gamma(\theta)}{c}\frac{2z^2}{(1+z)^3}\Big(L_{12}(\frac{3}{1+z}\sin(2\theta)\partial_\theta g)(0)+2\alpha\mu\Big)+\mathcal{N}_0+\mathcal{N}+\mathcal{N}_*.\]
Now call
\[\bar\mu=L_{12}(\frac{3}{1+z}\sin(2\theta)\partial_\theta g)(0)+2\alpha\mu.\]
Thus we arrive at:
\[\mathcal{L}_\Gamma^T(g)=-\bar\mu\frac{\Gamma(\theta)}{c}\frac{2z^2}{(1+z)^3}+\mathcal{N}_0+\mathcal{N}+\mathcal{N}_*.\]
Now we choose $\bar\mu$ so that $L_{12}(g)$ remains $0$. 
That is, we take:
\[\bar\mu:=L_{12}(\mathcal{N}_0)(0)+L_{12}(\mathcal{N})(0),\] where we note that $L_{12}(\mathcal{N}_0)(0)=0$ so long as $L_{12}(g)(0)=0$. 
Therefore, we have to solve:
\begin{equation} \label{gEquationFinal}\mathcal{L}_\Gamma^T(g)=-\Big(L_{12}(\mathcal{N}_0)(0)+L_{12}(\mathcal{N})(0)\Big)\frac{\Gamma(\theta)}{c}\frac{2z^2}{(1+z)^3}+\mathcal{N}_0+\mathcal{N}+\mathcal{N}_*.\end{equation}

\begin{equation}\label{N0}\mathcal{N}_0=\frac{3}{2\alpha}\sin(2\theta)L_{12}(F_*)\partial_\theta F_*-(\cos(2\theta)-\sin^2(\theta))L_{12}(F_*)z\partial_z F_*+\frac{1}{\alpha}L_{12}(g) g+\frac{3}{2\alpha}\sin(2\theta)L_{12}(g)\partial_\theta F \end{equation}\[-(\cos(2\theta)-\sin^2(\theta))L_{12}(g)z\partial_z F-(\cos(2\theta)-\sin^2(\theta))L_{12}(F_*)z\partial_z g\]

\begin{equation}\label{N} \mathcal{N}=2\mathcal{R}(\Phi-\frac{1}{4\alpha}\sin(2\theta)L_{12}(F)) F-2U(\Phi-\frac{1}{4\alpha}\sin(2\theta)L_{12}(F))\partial_\theta F\end{equation} \[-2\alpha V(\Phi-\frac{1}{4\alpha}\sin(2\theta)L_{12}(F)) z\partial_zF-\alpha\sin(2\theta)(F,K)_{L^2_\theta}\partial_\theta F-2\alpha\sin^2(\theta)(F,K)_{L^2_\theta}F.\]

\begin{equation}\label{N*}
\mathcal{N}_*=-\mu g-(\mu+\lambda+\mu\lambda)g
\end{equation}

\begin{equation} 
\mu=\frac{1}{2\alpha}(L_{12}(\mathcal{N}_0)(0)+L_{12}(\mathcal{N})(0))-\frac{1}{2\alpha}L_{12}(\frac{3}{1+z}D_\theta g)(0), \qquad \lambda=-\frac{2\mu}{\mu+1}.
\end{equation}

\begin{equation}\label{PhitoU} U(\Phi):=-3\Phi-\alpha z\partial_z\Phi\quad V(\Phi):=\partial_\theta\Phi-\tan(\theta)\Phi, \quad \R(\Phi):=\frac{1}{\cos(\theta)}\Big(2\sin(\theta)\Phi+\alpha\sin(\theta)z\partial_z\Phi+\cos(\theta)\partial_\theta\Phi\Big),\end{equation}

\begin{equation}\label{PolarBSLFinal}-\alpha^2z^2\partial_{zz}\Phi-\alpha(5+\alpha)z\partial_z\Phi-\partial_{\theta\theta}\Phi+\partial_\theta\big(\tan(\theta)\Phi\big)-6\Phi=F.\end{equation}

The remaining portion of the paper will be devoted to showing that the system \eqref{gEquationFinal}-\eqref{PolarBSLFinal} possesses an $\mathcal{H}^2$ solution of size at most $O(\alpha^2)$ if $\alpha$ is small enough. Toward this end, we will study the $\mathcal{H}^4$ inner product of \eqref{gEquationFinal} with $g$.

\subsection{Terms in $\mathcal{N}_0$}
The goal of this subsection is to estimate $(g,\mathcal{N}_0)_{\mathcal{H}^4}$ for $\mathcal{N}_0$ as in \eqref{N0}:
\[\mathcal{N}_0=\frac{3}{2\alpha}\sin(2\theta)L_{12}(F_*)\partial_\theta F_*-(\cos(2\theta)-\sin^2(\theta))L_{12}(F_*)z\partial_z F_*+\frac{1}{\alpha}L_{12}(g) g+\frac{3}{2\alpha}\sin(2\theta)L_{12}(g)\partial_\theta F \]\[-(\cos(2\theta)-\sin^2(\theta))L_{12}(g)z\partial_z F-(\cos(2\theta)-\sin^2(\theta))L_{12}(F_*)z\partial_z g:=\sum_{i=1}^6I^{\mathcal{N}_0}_i.\] The result of this subsection is Proposition \ref{N0Estimate}.

\subsubsection{$I^{\mathcal{N}_0}_1+I^{\mathcal{N}_0}_2$}

In these terms we see the importance of the exact form of $\Gamma$ as $\Gamma(\theta)=(\sin(\theta)\cos^2(\theta))^{\alpha/3}$. Indeed, each term of \[\frac{3}{2\alpha}\sin(2\theta)L_{12}(F_*)\partial_\theta F_*-(\cos(2\theta)-\sin^2(\theta))L_{12}(F_*)z\partial_z F_*\] only vanishes linearly at $z=0$ and thus does not belong to $\mathcal{H}^4$. However, because of the exact form of $\Gamma$ we see:
\[\frac{3}{2\alpha}\sin(2\theta)L_{12}(F_*)\partial_\theta F_*-(\cos(2\theta)-\sin^2(\theta))L_{12}(F_*)z\partial_z F_*\]\[=\frac{2\alpha}{c(1+z)}\Big(3\sin(2\theta)\frac{z}{(1+z)^2}\partial_\theta\Gamma-2\alpha(\cos(2\theta)-\sin^2(\theta))z\partial_z \frac{z}{(1+z)^2}\Gamma\Big).\]
\[=\frac{8\alpha^2}{c(1+z)}(\cos(2\theta)-\sin^2(\theta))\frac{z^2}{(1+z)^3}\Gamma.\]
A direct calculation then gives 
\[\Big|\frac{3}{2\alpha}\sin(2\theta)L_{12}(F_*)\partial_\theta F_*-(\cos(2\theta)-\sin^2(\theta))L_{12}(F_*)z\partial_z F_*\Big|_{\mathcal{H}^4}\leq C\alpha^2.\]

\subsubsection{$I^{\mathcal{N}_0}_3$}

By the product estimate in Proposition \ref{HProductRule}, we have:
\[|II^{\mathcal{N}_0}_3|_{\mathcal{H}^4}=\Big|\frac{1}{\alpha}gL_{12}(g)\Big|_{\mathcal{H}^4}\leq \frac{C}{\alpha^{3/2}}|g|_{\mathcal{H}^4}^2.\]

\subsubsection{$I^{\mathcal{N}_0}_{4}, I^{\mathcal{N}_0}_{5}, I^{\mathcal{N}_0}_{6}$.}
We now consider $I^{\mathcal{N}_0}_{4}, I^{\mathcal{N}_0}_{5},$ and $ I^{\mathcal{N}_0}_{6}$. First, \[I^{\mathcal{N}_0}_4=\frac{3}{2\alpha}L_{12}(g)\sin(2\theta)\partial_\theta F_*+\frac{3}{2\alpha}L_{12}(g)\sin(2\theta)\partial_\theta g.\]
It is easy to see that \[|\sin(2\theta)\partial_\theta F_*|_{\mathcal{W}^{4,\infty}}\leq C\alpha^2.\]
Thus, by Lemma \ref{WProductRule} and Proposition \ref{L12H}, we have that
\[|\frac{3}{2\alpha}L_{12}(g)\sin(2\theta)\partial_\theta F_*|_{\mathcal{H}^4}\leq C\sqrt{\alpha}|g|_{\mathcal{H}^4}.\]
Next, by Propositions \ref{Transport1} and \ref{L12H}, we have that 
\[|(\frac{3}{2\alpha}L_{12}(g)\sin(2\theta)\partial_\theta g, g)_{\mathcal{H}^4}|\leq \frac{C}{\alpha^{3/2}}|g|_{\mathcal{H}^4}^3.\]

Next, consider \[I^{\mathcal{N}_0}_{5}=-(\cos(2\theta)-\sin^2(\theta))L_{12}(g)z\partial_z F_*-(\cos(2\theta)-\sin^2(\theta))L_{12}(g)z\partial_z g.\] 
As above, it is easy to see that 
\[|z\partial_z F_*|_{\mathcal{W}^{4,\infty}}\leq C\alpha.\]
Hence, by Lemma \ref{WProductRule} and Proposition \ref{L12H}, we have that:
\[|(\cos(2\theta)-\sin^2(\theta))L_{12}(g)z\partial_z F_*|_{\mathcal{H}^4}\leq C\sqrt{\alpha}|g|_{\mathcal{H}^4}^2.\] Moreover, by Propositions \ref{Transport2} and \ref{L12H}, we have that 
\[|((\cos(2\theta)-\sin^2(\theta))L_{12}(g)z\partial_z g,g)_{\mathcal{H}^4}|\leq \frac{C}{\sqrt{\alpha}}|g|_{\mathcal{H}^4}^3.\]

Finally, we study
\[I^{\mathcal{N}_0}_6=-(\cos(2\theta)-\sin^2(\theta))L_{12}(F_*)z\partial_z g.\]
Observe that \[|(\cos(2\theta)-\sin^2(\theta))L_{12}(F_*)|_{\mathcal{W}^{4,\infty}}\leq C\alpha.\] This is clear since $L_{12}(F_*)=\frac{2\alpha}{1+z}$. Thus, using Proposition \ref{Transport3}, we have that 
\[|(I^{\mathcal{N}_0}_6, g)_{\mathcal{H}^4}|\leq C\sqrt{\alpha}|g|_{\mathcal{H}^4}^2.\]

\subsubsection{Estimate of $(\mathcal{N}_0, g)_{\mathcal{H}^4}$.}

We have established the following Proposition
\begin{proposition}\label{N0Estimate}
Assume that $g\in\mathcal{H}^4$ and $L_{12}(g)(0)=0$. Then, if $\mathcal{N}_0$ is defined as in \eqref{N0}, we have:
\[|(\mathcal{N}_0, g)|_{\mathcal{H}^4}\leq C(\alpha^2|g|_{\mathcal{H}^4}+\sqrt{\alpha}|g|_{\mathcal{H}^4}^2+\frac{1}{\alpha^{3/2}}|g|_{\mathcal{H}^4}^3).\]
\end{proposition}

\subsection{Terms in $\mathcal{N}$}

The goal of this subsection is to estimate $(g, \mathcal{N})_{\mathcal{H}^4}$, for $\mathcal{N}$ as in \eqref{N}:

\[\mathcal{N}=2\mathcal{R}(\Phi-\frac{1}{4\alpha}\sin(2\theta)L_{12}(F)) F-2U(\Phi-\frac{1}{4\alpha}\sin(2\theta)L_{12}(F))\partial_\theta F\] \[-2\alpha V(\Phi-\frac{1}{4\alpha}\sin(2\theta)L_{12}(F)) z\partial_zF-\alpha\sin(2\theta)(F,K)_{L^2_\theta}\partial_\theta F-2\alpha\sin^2(\theta)(F,K)_{L^2_\theta}F:=\sum_{i=1}^5I^{\mathcal{N}}_i.\]
The result of this subsection is Proposition \ref{NEstimate}. 

\subsubsection{${I}^{\mathcal{N}}_1$}
We begin by studying
\[{I}^{\mathcal{N}}_1=2\mathcal{R}(\Phi-\frac{1}{4\alpha}\sin(2\theta)L_{12}(F)) F\]
\[=2\mathcal{R}(\Phi_g-\frac{1}{4\alpha}\sin(2\theta)L_{12}(g)) g+2\mathcal{R}(\Phi_g-\frac{1}{4\alpha}\sin(2\theta)L_{12}(g)) F_*\] \[+2\mathcal{R}(\Phi_{F_*}-\frac{1}{4\alpha}\sin(2\theta)L_{12}(F_*)) g+2\mathcal{R}(\Phi_{F_*}-\frac{1}{4\alpha}\sin(2\theta)L_{12}(F_*)) F_*.\] 
Let us study the $\mathcal{H}^4$ norm of each of these terms in order.
First, by Theorem \ref{RemovingL12} we have that
\[|\mathcal{R}(\Phi_g-\frac{1}{4\alpha}\sin(2\theta)L_{12}(g))|_{\mathcal{H}^4}\leq C|g|_{\mathcal{H}^4}.\] Thus, by Proposition \ref{HProductRule}, 

\[|\mathcal{R}(\Phi_g-\frac{1}{4\alpha}\sin(2\theta)L_{12}(g)) g|_{\mathcal{H}^4}\leq \frac{C}{\sqrt{\alpha}}|g|_{\mathcal{H}^4}^2.\] Next, using that $|F_*|_{\mathcal{W}^{4,\infty}}\leq C\alpha$ and Lemma \ref{WProductRule}, we have that
\[|\mathcal{R}(\Phi_g-\frac{1}{4\alpha}\sin(2\theta)L_{12}(g)) F_*|_{\mathcal{H}^4}\leq C\sqrt{\alpha}|g|_{\mathcal{H}^4}.\] Similarly, using Proposition \ref{PhiStarRegularity} and Lemma \ref{WProductRule}, we have that \[|\mathcal{R}(\Phi_{F_*}-\frac{1}{4\alpha}\sin(2\theta)L_{12}(F_*)) g|_{\mathcal{H}^4}\leq C\sqrt{\alpha}|g|_{\mathcal{H}^4}.\]
Finally, we come to \[\mathcal{R}(\Phi_{F_*}-\frac{1}{4\alpha}\sin(2\theta)L_{12}(F_*)) F_*.\] For this one, we make use of Corollary \ref{PhiStarRegularity2} and Proposition \ref{WtoH} to see that
\[|\mathcal{R}(\Phi_{F_*}-\frac{1}{4\alpha}\sin(2\theta)L_{12}(F_*))F_*|_{\mathcal{H}^4}\leq|\frac{z+1}{z}\mathcal{R}(\Phi_{F_*}-\frac{1}{4\alpha}\sin(2\theta)L_{12}(F_*))|_{\mathcal{W}^{4,\infty}} |\frac{(z+1)^2}{z}F_*|_{\mathcal{W}^{4,\infty}} \leq C\alpha^2.\] 

In conclusion, we see that
\[|I^{\mathcal{N}}_1|_{\mathcal{H}^4}\leq C(\alpha^2+\sqrt{\alpha}|g|_{\mathcal{H}^4}+\frac{1}{\sqrt{\alpha}}|g|_{\mathcal{H}^4}^2).\]

\subsubsection{$I^{\mathcal{N}}_i$ for $2\leq i\leq 5$.}
In this subsection, we will study the remaining four terms of $\mathcal{N}$:

\[I^{\mathcal{N}}_2=-2U(\Phi-\frac{1}{4\alpha}\sin(2\theta)L_{12}(F))\partial_\theta F,\]
\[I^{\mathcal{N}}_3=-2\alpha V(\Phi-\frac{1}{4\alpha}\sin(2\theta)L_{12}(F)) z\partial_zF,\]
\[I^{\mathcal{N}}_4=-\alpha\sin(2\theta)(F,K)_{L^2_\theta}\partial_\theta F,\]
\[I^{\mathcal{N}}_5=-2\alpha\sin^2(\theta)(F,K)_{L^2_\theta}F.\]
All of these terms are dealt in basically the same way: either with a transport estimate (like Proposition \ref{Transport1}), a product rule (like Proposition \ref{HProductRule}), or Proposition \ref{WtoH} when $F_*$ interacts with itself in conjunction with an elliptic estimate like Theorem \ref{RemovingL12} and Proposition \ref{PhiStarRegularity}. For the sake of avoiding repetition, we will only give the details for $I^{\mathcal{N}}_2$ and leave the rest to the reader. 

Now, to study $I^{\mathcal{N}}_2$, we expand $F=F_*+g$ and study each of the four terms individually as above: \[U(\Phi-\frac{1}{4\alpha}\sin(2\theta)L_{12}(F))\partial_\theta F=U(\Phi_g-\frac{1}{4\alpha}\sin(2\theta)L_{12}(g))\partial_\theta g+U(\Phi_g-\frac{1}{4\alpha}\sin(2\theta)L_{12}(g))\partial_\theta F_*\] \[+U(\Phi_{F_*}-\frac{1}{4\alpha}\sin(2\theta)L_{12}(F_*))\partial_\theta g+U(\Phi_{F_*}-\frac{1}{4\alpha}\sin(2\theta)L_{12}(F_*))\partial_\theta F_*.\]
First, by Proposition \ref{DivisionBySin}, we have that 
\[|\frac{1}{\sin(2\theta)}U(\Phi_g-\frac{1}{4\alpha}\sin(2\theta)L_{12}(g))|_{\mathcal{H}^4}\leq C|g|_{\mathcal{H}^4}.\] Thus, by Proposition \ref{Transport1}, we have that 
\[|(U(\Phi_g-\frac{1}{4\alpha}\sin(2\theta)L_{12}(g))\partial_\theta g,g)_{\mathcal{H}^4}|=|(\frac{1}{\sin(2\theta)}U(\Phi_g-\frac{1}{4\alpha}\sin(2\theta)L_{12}(g))D_\theta g,g)_{\mathcal{H}^4}|\leq \frac{C}{\sqrt{\alpha}}|g|_{\mathcal{H}^4}^3.\]
Next, \[|U(\Phi_g-\frac{1}{4\alpha}\sin(2\theta)L_{12}(g))\partial_\theta F_*|_{\mathcal{H}^4}\leq\frac{C}{\sqrt{\alpha}}|\frac{1}{\sin(2\theta)}U(\Phi_g-\frac{1}{4\alpha}\sin(2\theta)L_{12}(g))|_{\mathcal{H}^4}|D_\theta F_*|_{\mathcal{W}^{4,\infty}}\leq C\sqrt{\alpha}|g|_{\mathcal{H}^4},\] where we have used Proposition \ref{DivisionBySin} and Lemma \ref{WProductRule}.

Now we turn to \[(|U(\Phi_{F_*}-\frac{1}{4\alpha}\sin(2\theta)L_{12}(F_*))\partial_\theta g,g)_{\mathcal{H}^4}|\leq C|(\partial_\theta U(\Phi_{F_*}-\frac{1}{4\alpha}\sin(2\theta)L_{12}(F_*)))|_{\mathcal{W}^{4,\infty}} |g|_{\mathcal{H}^4}^2\leq C\alpha |g|_{\mathcal{H}^4}^2,\] by Proposition \ref{SpecialTransport} and Proposition \ref{PhiStarRegularity}. 

Finally, we observe that 
\[|U(\Phi_{F_*}-\frac{1}{4\alpha}\sin(2\theta)L_{12}(F_*))\partial_\theta F_*|_{\mathcal{H}^4}\leq C\alpha^2,\] which follows from Lemma \ref{WtoW}, Corollary \ref{PhiStarRegularity2}, and Proposition \ref{WtoH}. 

The estimates on $I_{3}^{\mathcal{N}}$ is similar and the estimates of $I^{\mathcal{N}}_4$ and $I^{\mathcal{N}}_5$ are even easier. 

In conclusion, we have established that 
\[\sum_{i=2}^5|(I^{\mathcal{N}}_i,g)_{\mathcal{H}^4}|\leq C( \frac{1}{\sqrt{\alpha}}|g|_{\mathcal{H}^4}^3+\sqrt{\alpha}|g|_{\mathcal{H}^4}^2+\alpha^2|g|_{\mathcal{H}^4}).\]

\subsubsection{Estimate of $(\mathcal{N}, g)_{\mathcal{H}^4}$.}

\begin{proposition}
\label{NEstimate}
Let $g\in\mathcal{H}^4$ satisfy $L_{12}(g)(0)=0$. Then we have that
\[|(g,\mathcal{N})_{\mathcal{H}^4}|\leq C(\alpha^2|g|_{\mathcal{H}^4}+\sqrt{\alpha}|g|_{\mathcal{H}^4}^2+\frac{1}{\alpha^{3/2}}|g|_{\mathcal{H}^4}^3).\]
\end{proposition}

\subsection{Terms in $\mathcal{N}_*$}

Finally, we move to $\mathcal{N}_*$:
\[\mathcal{N}_*=-\mu g-(\mu+\lambda+\mu\lambda)g,\] where 
\[\mu=\frac{1}{2\alpha}(L_{12}(\mathcal{N}_0)(0)+L_{12}(\mathcal{N})(0))-\frac{1}{2\alpha}L_{12}(\frac{3}{1+z}D_\theta g)(0), \qquad \lambda=-\frac{2\mu}{\mu+1}.
\]
It is easy to see that 
\[|(g,g)_{\mathcal{H}^4}|+ |(g,z\partial_zg)|_{\mathcal{H}^4}\leq C|g|_{\mathcal{H}^4}^2.\]
Thus, 
\[|(\mathcal{N}_*,g)_{\mathcal{H}^4}|\leq C|\mu| |g|_{\mathcal{H}^4}^2,\] so long as $|\mu|\leq \frac{1}{2}$ (it will, in fact, be of order $\alpha$ in the end).  

\subsubsection{Estimate on $\mu$}

We now establish the following proposition.
\begin{proposition}\label{muestimate} \begin{equation} |\mu|\leq C(\alpha+\frac{1}{\alpha^{3/2}}|g|_{\mathcal{H}^4}+\frac{1}{\alpha^{5/2}}|g|_{\mathcal{H}^4}^2).\end{equation}\end{proposition}

\begin{remark}
This is actually an over-estimate in the second two terms where there is an extra factor of $\alpha^{-1/2}$ than what needs to be there. We do not care for this precision. 
\end{remark}

\begin{proof}
The proof is just based on the trivial observation that 
\[|L_{12}(f)(0)|\leq C|\frac{z+1}{z}f|_{L^2},\] whenever $f$ is such that the right hand side is finite. It then easy to see that the estimates we have already done above on $\mathcal{N}$ and $\mathcal{N}_0$ give us the result. Toward some completeness, we detail the argument for $I^{\mathcal{N}}_2$ that was studied above. 
In particular, we study:
\[L_{12}(I^{\mathcal{N}}_2)(0)=-2L_{12}\Big(U(\Phi-\frac{1}{4\alpha}\sin(2\theta)L_{12}(F))\partial_{\theta}F\Big)(0).\] Recall that $F=F_*+g$, but we do not treat $g$ and $F_*$ much differently below so we just keep them together. Observe that
\[L_{12}\Big(U(\Phi-\frac{1}{4\alpha}\sin(2\theta)L_{12}(F))\partial_{\theta}F\Big)(0)=(U(\Phi-\frac{1}{4\alpha}\sin(2\theta)L_{12}(F))\partial_{\theta}F, \frac{K(\theta)}{z})_{L^2}\]
\[=-(\partial_\theta U(\Phi-\frac{1}{4\alpha}\sin(2\theta)L_{12}(F))F, \frac{K(\theta)}{z})_{L^2}-(U(\Phi-\frac{1}{4\alpha}\sin(2\theta)L_{12}(F))F, \frac{K'(\theta)}{z})_{L^2},\] where we have just integrated by parts in $\theta$. Now, note that $K,K'$ are uniformly bounded and that \[|\partial_\theta U(\Phi-\frac{1}{4\alpha}\sin(2\theta)L_{12}(F))|_{L^\infty}+|U(\Phi-\frac{1}{4\alpha}\sin(2\theta)L_{12}(F))|_{L^\infty}\leq C(\alpha + \frac{1}{\sqrt{\alpha}}|g|_{\mathcal{H}^4}),\] using Theorem \ref{RemovingL12}, Proposition \ref{PhiStarRegularity}, and Corollary \ref{Linfty} (the $\alpha$ is coming from the $F_*$ and the $\frac{1}{\sqrt{\alpha}}|g|_{\mathcal{H}^4}$ is coming from the $g$ in $F=F_*+g$). We then see that:
\[|L_{12}(I^{\mathcal{N}}_2)(0)|\leq C(\alpha+\frac{1}{\sqrt{\alpha}}|g|_{\mathcal{H}^4}|(F,\frac{1}{z})_{L^2}|\leq C(\alpha+\frac{1}{\sqrt{\alpha}}|g|_{\mathcal{H}^4}) (\alpha+ |g|_{\mathcal{H}^4}). \] The rest of the estimates are of a similar nature and are much easier than what we have already done. We leave them to the interested reader. 

\end{proof}

\subsubsection{Estimate on $(\mathcal{N}_*,g)_{\mathcal{H}^4}$.}

From Proposition \ref{muestimate} and the preceding calculation, we see that we have the following
\begin{proposition}\label{N*Estimate}
Let $g\in\mathcal{H}^4$ satisfy $L_{12}(g)(0)=0$. Then we have that
\[|(g,\mathcal{N}_*)_{\mathcal{H}^4}|\leq C(\alpha |g|_{\mathcal{H}^4}^2+\frac{1}{\alpha^{3/2}}|g|_{\mathcal{H}^4}^3+\frac{1}{\alpha^{5/2}}|g|_{\mathcal{H}^4}^4).\]
\end{proposition}

\subsection{Final a-priori estimate on $g$}
By combining the estimates \eqref{N0Estimate}, \eqref{NEstimate},\eqref{N*Estimate}, we have shown the following proposition. 

\begin{proposition}\label{innerproductestimate}
There exists a universal constant $C>0$ so that the following holds for all given $\alpha>0$ and $g\in\mathcal{H}^4$ with $L_{12}(g)(0)=0$:  \[|(\mathcal{N}_0, g)_{\mathcal{H}^4}|+|(\mathcal{N}, g)_{\mathcal{H}^4}|+|(\mathcal{N}_*,g)_{\mathcal{H}^4}|+\Big|\Big(L_{12}(\mathcal{N}_0)(0)+L_{12}(\mathcal{N})(0)\Big)\frac{\Gamma(\theta)}{c}\frac{2z^2}{(1+z)^3}\Big|_{\mathcal{H}^4}|g|_{\mathcal{H}^4}\] \[\leq C\Big(\alpha^2|g|_{\mathcal{H}^4}+\sqrt{\alpha}|g|_{\mathcal{H}^4}^2+\frac{1}{\alpha^{3/2}}|g|_{\mathcal{H}^4}^3+\frac{1}{\alpha^{5/2}}|g|_{\mathcal{H}^4}^4\Big), \] with $\mathcal{N}$, $\mathcal{N}_*$, and $\mathcal{N}_0$ given as in \eqref{N0}-\eqref{PolarBSLFinal}.
\end{proposition}

We now have the following corollary which follows using Proposition \ref{prop:HkCoercivity}, Proposition \ref{innerproductestimate}, and equation \eqref{gEquationFinal}.

\begin{corollary}
There exists a universal constant $C>0$ so that if $g\in\mathcal{H}^4$, $L_{12}(g)(0)=0$, and $g$ solves \eqref{gEquationFinal}-\eqref{PolarBSLFinal} for some $\alpha>0$, then \[|g|_{\mathcal{H}^2}^2\leq (\mathcal{L}_\Gamma^T g,g)_{\mathcal{H}^4}\leq C\Big(\alpha^2|g|_{\mathcal{H}^4}+\sqrt{\alpha}|g|_{\mathcal{H}^4}^2+\frac{1}{\alpha^{3/2}}|g|_{\mathcal{H}^4}^3+\frac{1}{\alpha^{5/2}}|g|_{\mathcal{H}^4}^4\Big).\]
In particular, if we assume that $|g|_{\mathcal{H}^4}\leq \alpha^{7/4}$ with $\alpha$ is small enough, then we actually have:
\[|g|_{\mathcal{H}^4}\leq 2C\alpha^2.\]
\end{corollary}
Passing from the a-priori estimate to existence will now follow using a compactness method which we explain in the next section.

\subsection{Constructing the solution}\label{Construction}

One can view the estimates of the previous subsection as merely formal. Indeed, we do not know that a solution $g$ to \eqref{gEquationFinal} exists. We now eliminate this shortcoming by introducing a "fake" time variable $\tau$ and viewing the solution of \eqref{gEquationFinal} as the limit of solutions to the $\tau$-dependent equation when $\tau\rightarrow\infty$. 
That is, we solve the following evolution equation for the function $g(R,\theta,\tau)$:
\begin{equation}\label{gequationtime}\partial_\tau g+\mathcal{L}_\Gamma^T(g)=-\Big(L_{12}(\mathcal{N}_0)(0,\tau)+L_{12}(\mathcal{N})(0,\tau)\Big)\frac{\Gamma(\theta)}{c}\frac{2z^2}{(1+z)^3}+\mathcal{N}_0+\mathcal{N}+\mathcal{N}_*,\end{equation}
\[g(z,\theta,0)=0,\] with $\mathcal{N}$, $\mathcal{N}_0$, and $\mathcal{N}_*$ as in \eqref{N0}-\eqref{PolarBSLFinal}. The reader should take note that the $\tau$-independent $g$ of \eqref{gEquationFinal} is not the same as the $g$ of \eqref{gequationtime} but $\lim_{\tau\rightarrow\infty} g(z,\theta,\tau)$ will be shown to solve \eqref{gEquationFinal}. 

\begin{remark}
Another way of "regularizing" \eqref{gEquationFinal} would have been to add a small viscous term $\nu\Delta g$ to the equation and then argue using the Schauder fixed point theorem, for example. While this would generally be a natural avenue to solving such a problem, we found that the viscous term would not interact very well with the weighted norms in the definition of $\mathcal{H}^2$ and getting viscosity independent bounds seems to be difficult, but it is possible that there are ways around this. 
\end{remark}

\subsubsection{\emph{a-priori} estimates on $g$ and $\partial_\tau g$}
First, let us note that $L_{12}(g)(0,\tau)$ satisfies:
\[\partial_\tau L_{12}(g)(0,\tau)=-L_{12}(\mathcal{L}_\Gamma^T(g))(0,\tau)=-L_{12}(\mathcal{L}_\Gamma(g))(0,\tau)=-\mathcal{L}(L_{12}(g))(0)=L_{12}(g)(0,\tau).\]
Thus, since $L_{12}(g)(0,0)=0$, we have that $L_{12}(g)(0,\tau)=0$ for all $\tau\geq 0$. Now that we know this, as a consequence of Proposition \ref{innerproductestimate}, we have that \[\frac{d}{d\tau}|g|_{\mathcal{H}^4}\leq -2c_*|g|_{\mathcal{H}^4}+C(\alpha^2+\sqrt{\alpha} |g|_{\mathcal{H}^4}+\frac{1}{\alpha^{3/2}}|g|_{\mathcal{H}^4}^2+\frac{1}{\alpha^{5/2}}|g|_{\mathcal{H}^4}^3)\]
\[=|g|_{\mathcal{H}^4}(-c_*+C\sqrt{\alpha}+\frac{C}{\alpha^{3/2}}|g|_{\mathcal{H}^4}+\frac{C}{\alpha^{5/2}}|g|_{\mathcal{H}^4}^2)+C\alpha^2,\] for some fixed universal constants $c_*,C>0$. In particular, since $g|_{\tau=0}=0$, if $\alpha$ is sufficiently small (depending only on $c_*$ and $C$) we have that \begin{equation} |g|_{\mathcal{H}^4}\leq C\alpha^2\end{equation}for all $\tau\geq 0$. This already means that $g$ has a (subsequential) limit as $\tau\rightarrow\infty$. Let's show this limit is actually unique. 
Toward this end, we study the equation for $\partial_\tau g.$ 

Note that $F_*$ is independent of $\tau$. Consequently, the equation for $\partial_\tau g$ becomes:
\[\partial_{\tau\tau} g+\mathcal{L}_\Gamma(\partial_\tau g)=-\partial_\tau\Big(L_{12}(\mathcal{N}_0)(0)+L_{12}(\mathcal{N})(0)\Big)\frac{\Gamma(\theta)}{c}\frac{2z^2}{(1+z)^3}+\partial_\tau\mathcal{N}_0+\partial_\tau\mathcal{N}+\partial_\tau{\mathcal{N}_*}.\]
The key point is that all terms which are quadratic in $F_*$ are independent of $\tau$. Thus, using arguments identical to those which led to Proposition \ref{innerproductestimate} along with Proposition \ref{FirstDerivativeCoercivity} we see:
\[\frac{d}{d\tau} |\partial_\tau g|_{\mathcal{H}^3}\leq -2c_*|\partial_\tau g|_{\mathcal{H}^3}+\frac{C}{\alpha^{3/2}}|g|_{\mathcal{H}^4}|\partial_\tau g|_{\mathcal{H}^3}.\]
Thus, so long as $\alpha$ is small enough, we have
\[|\partial_\tau g|_{\mathcal{H}^3}\leq C\alpha^2 \exp(-c_* \tau)\] for all $\tau\geq 0$. It now follows that $g$ has a (unique) limit as $\tau\rightarrow\infty$. Now it is easy to see that \eqref{gEquationFinal}-\eqref{PolarBSLFinal} has a unique $\mathcal{H}^4$ solution in $B_{C\alpha^2}(0)$ and vanishing on $\theta=0$ and $\theta=\pi/2$ when $\alpha$ is sufficiently small. From there we also see that $\mu$ and $\lambda$ are of order $\alpha$. This gives a self similar solution to \eqref{OmegaEvolutionFinal}-\eqref{PolarBSLFinal1} and, in particular, implies Theorem \ref{MainTheorem}. 

\section{Conclusion}
We have established finite-time singularity formation for classical $C^{1,\alpha}$ solutions to the 3D Euler system when $\alpha>0$ is small. This was done by establishing a link between no-swirl axi-symmetric solutions to the 3D Euler system and a simple model which we have called the fundamental model in Section \ref{FundamentalModel}. To make this rigorous, we took advantage of a small parameter $\alpha$ which is related to the degree of vanishing of the vorticity at the origin and on the axis of symmetry. Localizing the self-similar solutions we have constructed to finite-energy solutions with no force is done in a work with Ghoul and Masmoudi \cite{EGM3dE}. 

Several related questions remain open after this work. Here, the data was $C^\infty$ except on the whole $x_3$-axis and the $x_3=0$ plane; the latter is most likely a technical artifact of the proof above and getting solutions which are $C^{1,\alpha}$ everywhere and $C^{\infty}$ except on the $x_3$ axis is likely within reach. In fact, using methods perhaps closer to \cite{KS} or \cite{KRYZ}, one might try to show that no-swirl $C^{1,\frac{1}{3}-}$ solutions could become singular in finite time. For other equations, such as the SQG system, it is possible that the kind of ideas used here could lead to blow-up of $C^{1,\alpha}$ solutions.  In the presence of spatial boundaries, I believe that several advances can be made. It would be very interesting to construct solutions which are $C^{1,\alpha}$ on $\mathbb{R}^3$ and $C^\infty$ except at a single point and become singular in finite time. Specifically, the solution constructed in this paper is non-smooth in the angular and radial variables. I believe that the non-smoothness in the radial variable is not essential but that the non-smoothness in the angular variable is essential for this construction. Constructing a blow-up that is $C^\infty$ in the angular variable would require using the swirl or a different geometry and seems to be a challenging problem, though many of the ideas here may be helpful toward that goal. 

\section{Acknowledgements}
The author thanks T. Buckmaster, P. Constantin, T. Ghoul,  I. Jeong, N. Masmoudi, J. Shatah, V. \v{S}ver\'ak, E.S. Titi, V. Vicol, and A. Zlato\v{s} for helpful comments on earlier versions of this work. The author also thanks T. Tao for observing that the term $\mathcal{N}_*$ in the nonlinear estimates was absent from the first version of this work. The author also acknowledges several important remarks of the referees that greatly improved the paper. He acknowledges funding from the NSF DMS-1817134. 

\bibliographystyle{plain}
\bibliography{3dEuler_away_final}

\begin{thebibliography}{10}

\bibitem{AHK}
H.~Abidi, T.~Hmidi, and S.~Keraani.
\newblock On the global well-posedness for the axisymmetric euler equations.
\newblock {\em Math. Ann.}, 347:15--41, 2010.

\bibitem{BardosTitiReview}
Claude Bardos and Edriss~S. Titi.
\newblock Euler equations for incompressible ideal fluids.
\newblock {\em Uspekhi Mat. Nauk}, 62(3 (375)):5--46, 2007.

\bibitem{BKM}
J.~T. Beale, T.~Kato, and A.~Majda.
\newblock Remarks on the breakdown of smooth solutions for the {$3$}-{D}
  {E}uler equations.
\newblock {\em Comm. Math. Phys.}, 94(1):61--66, 1984.

\bibitem{BVReview}
T.~{Buckmaster} and V.~{Vicol}.
\newblock Convex integration and phenomenologies in turbulence.
\newblock {\em arXiv:1901.09023}.

\bibitem{Chae2007}
Dongho Chae.
\newblock Nonexistence of self-similar singularities for the 3d incompressible
  {E}uler equations.
\newblock {\em Comm. Math. Phys.}, 273(1):203--215, 2007.

\bibitem{ChaeShv}
Dongho Chae and Roman Shvydkoy.
\newblock On formation of a locally self-similar collapse in the incompressible
  {E}uler equations.
\newblock {\em Arch. Ration. Mech. Anal.}, 209(3):999--1017, 2013.

\bibitem{Choi2017}
K~Choi, T.Y Hou, A~Kiselev, G~Luo, V.~{\v{S}}ver{{\'a}}k, and Y~Yao.
\newblock On the finite-time blowup of a one-dimensional model for the
  three-dimensional axisymmetric euler equations.
\newblock {\em Comm. Pure Appl. Math.}, 2017.

\bibitem{ConstantinReview}
P.~Constantin.
\newblock On the {E}uler equations of incompressible fluids.
\newblock {\em Bull. Amer. Math. Soc. (N.S.)}, 44(4):603--621, 2007.

\bibitem{CLM}
P.~Constantin, P.~D. Lax, and A.~Majda.
\newblock A simple one-dimensional model for the three-dimensional vorticity
  equation.
\newblock {\em Comm. Pure Appl. Math.}, 38(6):715--724, 1985.

\bibitem{ConstantinSun}
P.~Constantin and W.~Sun.
\newblock Remarks on {O}ldroyd-{B} and related complex fluid models.
\newblock {\em Commun. Math. Sci.}, 10:33--73, 2012.

\bibitem{Con}
Peter Constantin.
\newblock The {E}uler equations and nonlocal conservative {R}iccati equations.
\newblock {\em Internat. Math. Res. Notices}, (9):455--465, 2000.

\bibitem{CFM96}
Peter Constantin, Charles Fefferman, and Andrew~J. Majda.
\newblock Geometric constraints on potentially singular solutions for the
  {$3$}-{D} {E}uler equations.
\newblock {\em Comm. Partial Differential Equations}, 21(3-4):559--571, 1996.

\bibitem{Danchin}
R.~Danchin.
\newblock Axisymmetric incompressible flows with bounded vorticity.
\newblock {\em Russ. Math. Surv.}, 62(3), 2007.

\bibitem{DSReview}
C.~{De Lellis} and L.~{Sz{\'e}kelyhidi}, Jr.
\newblock On turbulence and geometry: from {N}ash to {O}nsager.
\newblock {\em arXiv:1901.02318}.

\bibitem{JianHouYu}
Jian Deng, Thomas~Y. Hou, and Xinwei Yu.
\newblock Geometric properties and nonblowup of 3{D} incompressible {E}uler
  flow.
\newblock {\em Comm. Partial Differential Equations}, 30(1-3):225--243, 2005.

\bibitem{Denisov1}
S.A. Denisov.
\newblock Infinite superlinear growth fo the gradient for the two-dimensional
  {E}uler equation.
\newblock {\em Contin. Dyn. Syst. A}, 23(3):755--764, 2009.

\bibitem{Denisov2}
S.A. Denisov.
\newblock Double \-exponential growth of the vorticity gradient for the
  two-dimensional {E}uler equation.
\newblock {\em Proc. Amer. Math. Soc.}, 143:1199--1210, 2015.

\bibitem{TDo}
T.~Do.
\newblock On vorticity gradient growth for the axisymmetric 3d {E}uler
  equations without swirl.
\newblock {\em arXiv:1801.07382}, 2018.

\bibitem{EGM3dE}
T.~M. Elgindi, T.~D. Ghoul, and N~Masmoudi.
\newblock On the stability of self-similar blow-up for ${C}^{1,\alpha}$
  solutions to the incompressible {E}uler equations on $\mathbb{R}^3$.
\newblock {\em ArXiv e-prints}, 2019.

\bibitem{EGM}
T.~M. Elgindi, T.~D. Ghoul, and N~Masmoudi.
\newblock Stable self-similar blowup for a family of nonlocal transport
  equations.
\newblock {\em ArXiv e-prints}, 2019.

\bibitem{EJB}
T.~M. {Elgindi} and I.-J. {Jeong}.
\newblock {Finite-time Singularity Formation for Strong Solutions to the
  Boussinesq System}.
\newblock {\em ArXiv e-prints}, August 2017.

\bibitem{EJVP}
T.~M. {Elgindi} and I.~J. {Jeong}.
\newblock On singular vortex patches, {I}: Well-posedness issues.
\newblock {\em ArXiv e-prints}, 2019.

\bibitem{EJDG}
Tarek Elgindi and In-Jee Jeong.
\newblock On the effects of advection and vortex stretching.
\newblock {\em To appear in Arch. Rat. Mech. Anal.}

\bibitem{E1}
Tarek~M. Elgindi.
\newblock Remarks on functions with bounded {L}aplacian.
\newblock {\em arXiv:1605.05266}, 2016.

\bibitem{EJE}
Tarek~M. Elgindi and In-Jee Jeong.
\newblock Finite-time singularity formation for strong solutions to the
  axi-symmetric $3d$ {E}uler equations.
\newblock {\em Ann. PDE}, 2019.

\bibitem{EM1}
Tarek~M Elgindi and Nader Masmoudi.
\newblock ${L}^\infty$ ill-posedness for a class of equations arising in
  hydrodynamics.
\newblock {\em {T}o appear in Arch. Rat. Mech. Anal.}

\bibitem{Gibbon2008}
J.~D. Gibbon.
\newblock The three-dimensional {E}uler equations: where do we stand?
\newblock {\em Phys. D}, 237(14-17):1894--1904, 2008.

\bibitem{GMS}
J.~D. Gibbon, D.~R. Moore, and J.~T. Stuart.
\newblock Exact, infinite energy, blow-up solutions of the three-dimensional
  {E}uler equations.
\newblock {\em Nonlinearity}, 16(5):1823--1831, 2003.

\bibitem{problems}
Loukas Grafakos, Diogo~Oliveira e~Silva, Malabika Pramanik, Andreas Seeger, and
  Betsy Stovall.
\newblock Some problems in harmonic analysis.
\newblock {\em arXiv:1701.06637}.

\bibitem{Gunther}
N.~Gunther.
\newblock On the motion of fluid in a moving container.
\newblock {\em Izvestia Akad. Nauk USSR, Ser. Fiz.--Mat.}, 20(1323-1348), 1927.

\bibitem{HouLei}
Thomas~Y. Hou and Zhen Lei.
\newblock On the stabilizing effect of convection in three-dimensional
  incompressible flows.
\newblock {\em Comm. Pure Appl. Math.}, 62(4):501--564, 2009.

\bibitem{JSS}
H.~Jia, S.~Stewart, and V.~Sverak.
\newblock On the {D}e {G}regorio modification of the {C}onstantin-{L}ax-{M}ajda
  model.
\newblock {\em Arch. Ration. Mech. Anal.}, 231(2):1269--1304, 2019.

\bibitem{KatoPonce}
T.~Kato and G.~Ponce.
\newblock Commutator estimates and the {E}uler and {N}avier-{S}tokes equations.
\newblock {\em Comm. Pure Appl. Math.}, 41(7):891--907, 1988.

\bibitem{Kato86}
Tosio Kato.
\newblock Remarks on the {E}uler and {N}avier-{S}tokes equations in $\bf
  {R}^2$.
\newblock In {\em Nonlinear functional analysis and its applications, Part 2
  (Berkeley, Calif., 1983)}, pages 1--7. Amer. Math. Soc., Providence, R.I.,
  1986.

\bibitem{Ker1}
R.~M. Kerr.
\newblock Evidence for a singularity of the three-dimensional, incompressible
  {E}uler equations.
\newblock In {\em Topological aspects of the dynamics of fluids and plasmas
  ({S}anta {B}arbara, {CA}, 1991)}, volume 218 of {\em NATO Adv. Sci. Inst.
  Ser. E Appl. Sci.}, pages 309--336. Kluwer Acad. Publ., Dordrecht, 1992.

\bibitem{Ker2}
Robert~M. Kerr.
\newblock Evidence for a singularity of the three-dimensional, incompressible
  {E}uler equations.
\newblock {\em Phys. Fluids A}, 5(7):1725--1746, 1993.

\bibitem{KiselevReview}
A~Kiselev.
\newblock Small scales and singularity formaiton in fluid dynamics.
\newblock {\em arXiv:1807.00184}.

\bibitem{KRYZ}
Alexander Kiselev, Lenya Ryzhik, Yao Yao, and Andrej Zlato\v{s}.
\newblock Finite time singularity for the modified {SQG} patch equation.
\newblock {\em Ann. of Math. (2)}, 184(3):909--948, 2016.

\bibitem{KS}
Alexander Kiselev and Vladimir {\v{S}}ver{\'a}k.
\newblock Small scale creation for solutions of the incompressible
  two-dimensional {E}uler equation.
\newblock {\em Ann. of Math. (2)}, 180(3):1205--1220, 2014.

\bibitem{LPTW}
Adam Larios, Mark Petersen, Edriss~S. Titi, and Beth Wingate.
\newblock A computational investigation of the finite-time blow-up of the 3{D}
  incompressible {E}uler equations based on the voigt regularization.
\newblock {\em Theor. Comput. Fluid Dyn.}, 32:23--34, 2018.

\bibitem{LariosTiti}
Adam Larios and Edriss~S. Titi.
\newblock Global regularity versus finite-time singularities: some paradigms on
  the effect of boundary conditions and certain perturbations.
\newblock {\em Recent Progress in the Theory of the Euler and Navier-Stokes
  Equations, London Mathematical Society Lecture Notes}, 430:96--125, 2016.

\bibitem{LeiLiuRen}
Z.~Lei, J.~Liu, and X.~Ren.
\newblock On the {C}onstantin-{L}ax-{M}ajda model with convection.
\newblock {\em arXiv:1811.09754}.

\bibitem{Lich25}
Leon Lichtenstein.
\newblock \"{U}ber einige {E}xistenzprobleme der {H}ydrodynamik.
\newblock {\em Mat. Zeit. Phys.}, 23:89--154, 1925.

\bibitem{HL}
Guo Luo and Thomas~Y. Hou.
\newblock Potentially singular solutions of the 3d axisymmetric euler
  equations.
\newblock {\em Proceedings of the National Academy of Sciences},
  111(36):12968--12973, 2014.

\bibitem{HouLuo}
Guo Luo and Thomas~Y. Hou.
\newblock Toward the finite-time blowup of the 3{D} axisymmetric {E}uler
  equations: a numerical investigation.
\newblock {\em Multiscale Model. Simul.}, 12(4):1722--1776, 2014.

\bibitem{MB}
Andrew~J. Majda and Andrea~L. Bertozzi.
\newblock {\em Vorticity and incompressible flow}, volume~27 of {\em Cambridge
  Texts in Applied Mathematics}.
\newblock Cambridge University Press, Cambridge, 2002.

\bibitem{N}
N.~Nadirashvilli.
\newblock Wandering solutions of the two-dimensional {E}uler equation.
\newblock {\em Funktsional. Anal. i Prilozhen.}, 25(70-71), 1991.

\bibitem{NRS96}
J.~Necas, M.~Ruzicka, and V.~Sver\'ak.
\newblock On {L}eray's self-similar solutions of the {N}avier-{S}tokes
  equations.
\newblock {\em Acta Math.}, 176(2):283--294, 1996.

\bibitem{SaintRaymond}
X.~Saint-Raymond.
\newblock Remarks on axisymmetric solutions of the incompressible {E}uler
  system.
\newblock {\em Comm. Partial Differential Equations}, 19(1-2):321--334, 1994.

\bibitem{SarriaSaxton}
Alejandro Sarria and Ralph Saxton.
\newblock Blow-up of solutions to the generalized inviscid {P}roudman-{J}ohnson
  equation.
\newblock {\em J. Math. Fluid Mech.}, 15(3):493--523, 2013.

\bibitem{ShirotaY}
T.~Shirota and T.~Yanagisawa.
\newblock Note on global existence for axially symmetric solutions of the
  {E}uler system.
\newblock {\em Proc. Japan Acad. Ser. A Math. Sci.}, 70(10):299--304, 1994.

\bibitem{Stuart1988}
J.~T. Stuart.
\newblock Nonlinear {E}uler partial differential equations: singularities in
  their solution.
\newblock In {\em Applied mathematics, fluid mechanics, astrophysics
  ({C}ambridge, {MA}, 1987)}, pages 81--95. World Sci. Publishing, Singapore,
  1988.

\bibitem{Tao2}
T.~Tao.
\newblock On the universality of the incompressible {E}uler equation on compact
  manifolds, {II}. {N}on-rigidity of {E}uler flows.
\newblock {\em arXiv:1902.0631}.

\bibitem{TaoManifold}
T.~{Tao}.
\newblock {On the universality of the incompressible Euler equation on compact
  manifolds}.
\newblock {\em ArXiv e-prints}, July 2017.

\bibitem{Tao2016}
Terence Tao.
\newblock Finite time blowup for {L}agrangian modifications of the
  three-dimensional {E}uler equation.
\newblock {\em Ann. PDE}, 2(2):Art. 9, 79, 2016.

\bibitem{Tsai1998}
Tai-Peng Tsai.
\newblock On {L}eray's self-similar solutions of the {N}avier-{S}tokes
  equations satisfying local energy estimates.
\newblock {\em Arch. Rational Mech. Anal.}, 143(1):29--51, 1998.

\bibitem{UY}
M.~R. Ukhovskii and V.~I. Iudovich.
\newblock Axially symmetric flows of ideal and viscous fluids filling the whole
  space.
\newblock {\em Prikl. Math. Meh.}, 32(1):59--69, 1968.

\bibitem{YLoss}
V.~I. Yudovich.
\newblock On the loss of smoothness of the solutions of the {E}uler equations
  and the inherent instability of flows of an ideal fluid.
\newblock {\em Chaos}, 10(3):705--719, 2000.

\bibitem{Z}
Andrej Zlatos.
\newblock Exponential growth of the vorticity gradient for the {E}uler equation
  on the torus.
\newblock {\em Adv. Math.}, 268:396--403, 2015.

\end{thebibliography}

\end{document}